\documentclass[11pt]{amsart}
\usepackage[utf8]{inputenc}
\usepackage{lmodern}
\usepackage{amsmath, amsthm, amssymb, amsfonts,bm}
\usepackage[normalem]{ulem}
\usepackage{hyperref}
\usepackage{varwidth}
\usepackage{adjustbox}
\usepackage{enumerate}

\usepackage{verbatim} 
\usepackage{longtable}

\usepackage{mathtools}

\usepackage{enumitem}
\usepackage{tikz}
\usetikzlibrary{decorations.pathmorphing}
\tikzset{snake it/.style={decorate, decoration=snake}}
\usepackage{tikz-cd}
\usetikzlibrary{positioning}

\usepackage{caption}

\usepackage{tikz-cd}
\usetikzlibrary{arrows}
\usepackage{ytableau}
\usepackage{mathdots}
\theoremstyle{plain}

\newtheorem{theorem}{Theorem}[section]

\newtheorem{definition}[theorem]{Definition}

\newtheorem{prop}[theorem]{Proposition}
\theoremstyle{definition}
\newtheorem{exmp}[theorem]{Examples}

\newtheorem{cor}[theorem]{Corollary}
\theoremstyle{remark}
\newtheorem{rem}[theorem]{Remark}
\newtheorem{lem}[theorem]{Lemma}
\newtheorem{ques}[theorem]{Question}

\DeclareFontFamily{OT1}{rsfs}{}
\DeclareFontShape{OT1}{rsfs}{n}{it}{<-> rsfs10}{}
\DeclareMathAlphabet{\curly}{OT1}{rsfs}{n}{it}

\def\Q{\mathbb{Q}}

\def\Z{\mathbb{Z}}

\def\C{\mathbb{C}}

\def\P{\mathbb{P}}

\def\O{\mathcal{O}}

\def\a{\alpha}
\def\b{\beta}

\def\m{\mu}

\def\sm{\Sigma}

\def\m{\mathcal}
\def\f{\mathfrak}
\def\c{\mathscr}

\def\bb{\mathbb}
\def\del{\partial}

\def\td{\widetilde}

\def\Spec{\textrm{Spec}}
\def\deg{\textrm{deg}}
\def\ch{\textrm{ch}}
\def\c{\textrm{c}}

\def\rk{\textrm{rank}}
\def\dim{\textrm{dim}}

\def\ker{\textrm{ker}}

\def\res{\mathrm{res}}

\def\para{\textrm{par-}}
\def\parb{\textrm{par}}
\def\uxi{{\underline{\xi}}}
\def\vxi{\vec{\xi}}
\def\vm{\vec{m}}
\def\va{\vec{\a}}

\def\Sigmam{\Sigma(\vm)}
\def\Sigmamxi{\Sigma(\vm)_{\vxi}}
\def\lmmu{L(\vm)_\mu}
\def\bmxi{B(\vm)_{\vxi}}
\def\hvxi{\mathcal{H}(\vec{m})_{\vec{\xi}}}
\def\mvxi{\mathcal{M}(\vec{m})_{\vec{\xi}}}
\def\hpar{\mathrm{Higgs}^{\rm{par}}}
\def\hspar{\mathrm{Higgs}^{\rm{s-par}}}

\def\brelativesheaf{\bm{\mathcal{M}}(\vm)}
\def\brelativehiggs{\bm{\mathcal{H}}(\vm)}

\def\um{{\underline{m}}}

\def\nm{\mathcal{N}(\vm)}

\def\Tot{\mathrm{Tot}}

\newsavebox{\leftbox} \newsavebox{\rightbox}%

\NewDocumentCommand{\lrboxbrace}{s O{\{} O{\}} O{0.05\linewidth} m O{0.5\linewidth} m}{% \lrboxbrace[<lbrace>][<rbrace>][<lwidth>]{<ltext>}[<rwidth>]{<rtext>}
  \begin{lrbox}{\leftbox}% Left box
    \IfBooleanTF{#1}% starred/unstarred
      {\begin{varwidth}{#4}#5\end{varwidth}}
      {\begin{minipage}{#4}#5\end{minipage}}
  \end{lrbox}
  \begin{lrbox}{\rightbox}% Right box
    \IfBooleanTF{#1}% starred/unstarred
      {\begin{varwidth}{#6}#7\end{varwidth}}
      {\begin{minipage}{#6}#7\end{minipage}}
  \end{lrbox}
  \ensuremath{\usebox\leftbox\left\{\,\usebox\rightbox\,\right\}}
}

\usepackage[backend=biber,style=alphabetic,sorting=nyt, giveninits=true,doi=false,isbn=false,maxbibnames=99]{biblatex} %main package for bibtex
\usepackage{vmargin}
\setpapersize{A4}
\setmarginsrb{27mm}{12mm}{27mm}{12mm}%
{12mm}{10mm}{5mm}{10mm}

\usepackage{tikz}
\usepackage{lmodern}
\usetikzlibrary{decorations.pathmorphing}

\begin{document}
\title{Relative spectral correspondence for parabolic Higgs bundles and Deligne-Simpson problem}

\author{Jia Choon Lee and Sukjoo Lee}
\address{Institute of Geometry and Physics, University of Science and Technology of China, 96 Jinzhai Road, Hefei 230026 P.R. China}
\email{jiachoonlee@ustc.edu.cn}

\address{Institute for Basic Science - Center for Geometry and Physics, 79 Jigok-ro 127 beon-gil, Nam-gu, Pohang, Gyeongbuk, KOREA 37673
}
\email{sukjoo216@ibs.re.kr}

\begin{abstract}
In this paper, we generalize the spectral correspondence for parabolic Higgs bundles established by Diaconescu--Donagi--Pantev to the relative setting. We show that the relative moduli space of $\vec{\xi}$-parabolic Higgs bundles on a curve can be realized as the relative moduli space of pure dimension one sheaves on a family of holomorphic symplectic surfaces. This leads us to formulate the image of the relative moduli space under the Hitchin map in terms of linear systems on the family of surfaces. Then we explore the relationship between the geometry of these linear systems and the so-called $OK$ condition introduced by Balasubramanian--Distler--Donagi in the context of six-dimensional superconformal field theories. As applications, we obtain (a) the non-emptiness of the moduli spaces and (b) the Deligne--Simpson problem and its higher genus analogue. In particular, we prove a conjecture proposed by Balasubramanian--Distler--Donagi that the $OK$ condition is sufficient for solving the Deligne--Simpson problem.
\end{abstract}

\baselineskip=14.5pt
\maketitle

\setcounter{tocdepth}{2}

\tableofcontents
\section{Introduction}

\subsection{Overview} 
In the 1980s, Nigel Hitchin introduced the moduli space of Higgs bundles (often called Hitchin system) and showed that it carries an integrable system structure via the so-called Hitchin fibration. The generic fibers of the Hitchin fibration can be described by the Beauville-Narasimhan-Ramanan (BNR) \cite{bnr} or spectral correspondence, which identifies Higgs bundles with line bundles (or more generally, torsion-free sheaves) on spectral curves lying in the cotangent bundle of the base curve. 

In fact, the spectral correspondence extends beyond the generic fibers \cite{simpsonmoduli}: it can be extended to an isomorphism between the moduli space of Higgs bundles and the moduli space of pure dimension one sheaves on the total space of the cotangent bundle. This perspective of Hitchin systems in terms of moduli of sheaves on a surface has led to many applications, for example, 
it has been used in \cite{chuangdiaconescupan}\cite{Diaconescu_2018}\cite{CDDSP}\cite{maulikshen-chiindependence}\cite{Kinjo_2024} to connect topological invariants of the Hitchin system to enumerative geometry. 

In this work, we focus on a variant of Higgs bundles—parabolic Higgs bundles—which originated from the tame non-abelian Hodge correspondence (tame NAHC for short) developed by Simpson \cite{simpson-noncompact}. The additional parabolic structure enables the specification of ordered eigenvalues of the residues of the Higgs field at fixed marked points. We are interested in the relative moduli space of parabolic Higgs bundles as these ordered eigenvalues vary over a base. This relative viewpoint has been useful in recent work, for example in the study of the topological \cite{hausel2022pw} \cite{maulik2023perverse} or motivic \cite{maulik2024algebraiccycleshitchinsystems} aspects  of the Hitchin systems through specialization arguments. 

The first goal of this work is to generalize the spectral correspondence for the parabolic Higgs bundles, established by Diaconescu--Donagi--Pantev \cite{Diaconescu_2018}, to the setting of relative moduli spaces of parabolic Higgs bundles, which we refer to as the \emph{relative spectral correspondence}. This construction produces a family of surfaces together with a corresponding family of linear systems. Then we explore the connections between the geometry of these linear systems and the so-called $OK$ condition introduced by Balasubramanian--Distler--Donagi \cite{baladistlerdonagi}, motivated by six-dimensional superconformal field theories. As a by-product, we found a new perspective on the classical multiplicative Deligne–Simpson problem through these surfaces and linear systems. In particular, our approach provides a geometric realization of the $OK$ condition and shows its sufficiency for the Deligne–Simpson problem, as conjectured in \emph{loc. cit.}

\subsection{Relative spectral correspondence}
Let $C$ be a smooth complex curve of genus $g$ and $D=p_1+\cdots+p_n$ be a reduced divisor. We reserve $r$ and $d$ for the rank and degree of a vector bundle over $C$, respectively. To each point $p_i$, we assign a partition of $r$, $\um_i=(\um_{i,1}, \cdots, \um_{i, \ell(i)})$ satisfying $\um_{i,1}\geq \cdots \geq\um_{i, \ell(i)}$. We encode this data as $\vec{m}=(\um_1, \cdots, \um_n)$. Given such data, we define a \emph{(quasi-)parabolic Higgs bundle} that is a Higgs bundle $(E, \Phi:E \to E\otimes K_C(D))$ over $C$ with a filtration associated to $\um_i$ on each restriction $E_{p_i}$ preserved by $\Phi|_{E_{p_i}}$ (see Definition \ref{def:parabolicHiggs}).

Consider a collection of sections  $\vec{\xi}=(\uxi_1, \cdots, \uxi_n)$ where $\uxi_i=(\xi_{i,1},\cdots, \xi_{i,\ell(i)}) \in H^0(p_i, K_C(D)_{p_i})^{\times \ell(i)}$, satisfying the residue condition $\sum_{i,j}m_{i,j}\mathrm{res}_{p_i}(\xi_{i,j})=0$. The collection of such sections is denoted by $\nm$. Then for $\vec{\xi}\in \nm$, we say a parabolic Higgs bundle is \emph{$\vec{\xi}$-parabolic} if the induced Higgs field on the $j$-th associated graded piece of the filtration on $E_{p_i}$ is the identity map scaled by $\xi_{i,j}$ (see Definition \ref{def:xi-parabolic}). When $\xi_{i,j}=0$ for all $i,j$, they are usually known as strongly parabolic Higgs bundles. Therefore, $\vxi$-parabolic Higgs bundles can be seen as a variation of strongly parabolic Higgs bundles by deforming the nilpotent residue condition. 

    In our previous work \cite{LeeLee2024}, we construct a family of moduli space of (irregular) $\vec{\xi}$-parabolic Higgs bundles over $\nm$, $\bm{\pi}_\m{N}:\brelativehiggs\to \nm$. Moreover, we show that this is a smooth family (see \cite[Section 2]{LeeLee2024}).\footnote{In this article, we only focus on the regular parabolic case where the divisor $D$ is reduced.} For $\vxi\in \nm$, we denote $\hvxi$ the fiber of $\bm{\pi}_\m{N}$ at $\vxi$.

    When all values $\xi_{i,j}(p_i)$'s are distinct and non-zero for each $i \in I$, a condition that we refer to as \emph{generic}, Diaconescu--Donagi--Pantev \cite{Diaconescu_2018} established the spectral correspondence for $\vxi$-parabolic Higgs bundles. For simplicity, we often abuse notation by writing $\xi_{i,j}=\xi_{i,j}(p_i)$. To elaborate, consider the projectivization of the twisted canonical line bundle $\pi:M:=\mathbb{P}(K_C(D)\oplus \m{O}_C) \to C$, and blow up $M$ at each point $(p_i, \xi_{i,j})$, in increasing lexicographic order of $(i,j)$ starting from $(i,j) = (1,1)$. Since the $\xi_{i,j}$'s are distinct, the order of the blow-ups does not matter. Denote the exceptional divisor appearing in the  $(i,j)$-th blow-up by $E_{i,j}$. Let $Z_{\vec{\xi}}$ be the resulting surface and $p_{\vec{\xi}}: Z_{\vec{\xi}} \to M$ be the blow-up map. Finally, we remove the strict transform of the fibers $M_{p_i}$ and the infinity section $C_\infty$ from $Z_{\vec{\xi}}$, and the resulting open surface is denoted by $S_{\vec{\xi}}$.

    To define the moduli of pure dimension one sheaves, consider the curve class $\Sigma(\vec{m})_{\vec{\xi}} := r p_{\vec{\xi}}^* C_0 - \sum_{i,j} m_{i,j} E_{i,j}$, where $C_0$ is the strict transform of the zero section of $\pi$. An effective curve $\Sigma \in |\Sigma(\vec{m})_{\vec{\xi}}|$ is characterized as a degree $r$ branched cover of $C_0$ lying in $S_\uxi$ that intersects each exceptional divisor $E_{i,j}$ exactly $m_{i,j}$ times. One can then define a moduli space $\mvxi$ of pure dimension one sheaves on $S_{\vec{\xi}}$ whose support is $\Sigma(\vec{m})_{\vec{\xi}}$, after choosing a suitable stability condition. Then the spectral correspondence in \cite{Diaconescu_2018} says that, for a generic $\vxi$, there is an isomorphism of moduli spaces: 
        \[
        \mvxi \cong \hvxi.
        \]
    In fact, their work applies more generally for irregular Higgs bundles (see also the work of Kontsevich--Soibelman \cite{Kontsevich-Soibelman} and Szab\'{o} \cite{Szab2017TheBG}).
        
        On the other hand, there is also a spectral correspondence introduced by Su--Wang--Wen \cite{Suwangwen2022}\cite{su2022topologicalmirror} for strongly parabolic Higgs bundles with integral spectral curves. This is in fact the extreme opposite of the case of generic $\vxi$ studied in \cite{Diaconescu_2018}, because all $\xi_{i,j}=0$. While the work of \cite{su2022topologicalmirror} is not formulated in terms of surfaces, one can still spot some similarities between the constructions of \cite{Diaconescu_2018} and \cite{su2022topologicalmirror} which involve some pattern of blow-ups. Hence, a natural question is how to establish a family version of this spectral correspondence over the whole base $\nm$ which specializes to both of their constructions. One of the main results of this article is to address this question. 
        \begin{ques}
            Is there a family of  holomorphic symplectic surfaces over $\nm$ such that the relative moduli spaces of sheaves on this family is isomorphic to $\brelativehiggs$ over $\nm$?
        \end{ques}
    We aim to construct appropriate surfaces and curve classes for non-generic $\vxi$ in a way that fits into a family over $\nm$. For such a non-generic $\vxi$, the eigenvalues $\xi_{i,j}$ can be repeated so that the blow-up process must be handled with care. A natural approach is to construct these objects as limits of those associated with generic $\vec{\xi}$. Therefore, the desired family of surfaces will be constructed as a sequence of blow-ups starting from the trivial family of surfaces. A key observation is that the order of blow-ups becomes crucial, and the choice of the curve class must be adjusted accordingly when we perform the blow-ups in a family. To illustrate this idea, we assume $i = 1$ and simplify the notation by dropping the subscript. For instance, let $D = p$ and temporarily write $\vec{m} = \um_1 = \um$ and $\vec{\xi} = \uxi_1 = \uxi$. 

We perform the blow-up construction globally by starting from the trivial family of surfaces, $\pi_\m{N}:\bm{M} := M \times \nm \to \nm$, and blowing up the $\ell$ tautological sections $\bm{\xi}_j := \{(p, \xi_j(p)) \times \uxi \mid \uxi \in \nm\}$ iteratively, as follows: First, blow up $\bm{\xi}_1$, and let $\bm{p}_1: \bm{M}_1 \to \bm{M}$ be the corresponding blow-up map. Next, blow up $\bm{M}_1$ along the strict transform of $\bm{\xi}_2$, and denote the resulting map by $\bm{p}_2^1: \bm{M}_2 \to \bm{M}_1$. We continue this construction iteratively, yielding the following chain of blow-ups:
        \[
        \begin{tikzcd}
            \bm{Z}=\bm{M}_\ell \arrow[r, "\bm{p}_{\ell}^{\ell-1}"] & \bm{M}_{\ell-1} \arrow[r] & \cdots \arrow[r, "\bm{p}_2^1"] & \bm{M}_1 \arrow[r, "\bm{p}_1"] & \bm{M} \arrow[d, "\bm{\pi}_\m{N}"]\\
            &&&& \nm
        \end{tikzcd}
        \]
        We denote the composition $\bm{p}_{j}^{i}\circ \bm{p}_{k}^{j}$ by $\bm{p}_{k}^{i}$. Also, let $\bm{E}_i$ be the exceptional divisor appearing at the $i$-th blow up. In the end, we remove the strict transform of $M_p \times \nm$ and $C_\infty \times \nm$ from $\bm{Z}$ to obtain a family of holomorphic symplectic surfaces $\bm{S}\subset \bm{Z}$. 

        To construct a relative moduli of sheaves, we choose the relative divisor class in $\bm{S}$, \[\bm{\Sigma}({\um})=r\bm{f}^*\bm{C}_0-\sum_{j=1}^\ell m_j\bm{\Xi}_j,\quad \textrm{where}\quad \bm{\Xi_j}=(\bm{p}_\ell^j)^*\bm{E}_j.\]
        By specializing the pair $(\bm{S}, \bm{\sm}({\um}))$ to any (including non-generic) $\uxi\in \nm$, we obtain the natural pair for establishing the corresponding spectral correspondence (see Figure \ref{fig:gensurf} and Figure \ref{fig:sparsurf} for examples). By choosing a suitable stability condition and numerical data, one can construct a relative moduli space $\brelativesheaf \to \nm$ of pure dimension one sheaves on $\bm{S}$ whose support is $\bm{\Sigma}({\um})$.

       In the general setup where $D = p_1 + \cdots + p_n$, the above construction of a family of  holomorphic symplectic surfaces and the relative moduli space extends naturally. These are described in Sections \ref{sec:construction of surfaces} and \ref{sec:relativemoduli of pure}, respectively. 
        \begin{theorem}[Relative spectral correspondence, Theorem \ref{spectral-corr}]\label{thm:intro-rel spec}
        There is a closed embedding of relative moduli spaces 
         \begin{equation}\label{eq:introrelsp}
        \begin{tikzcd}
             \bm{Q}: \brelativehiggs \arrow[rr, hook] 
            \arrow[rd, swap] && \brelativesheaf\arrow[ld] \\
            & \nm &
        \end{tikzcd}
    \end{equation}
    In particular, this embedding is an isomorphism for generic $\vxi\in \nm$. 
    \end{theorem}

Our construction of $\bm{Q}$ begins with a parabolic analogue of the classical BNR correspondence over each $\vxi \in \nm$,  which transform a $\vxi$-parabolic Higgs bundle on $C$ into a parabolic sheaf on $M$. We then iteratively apply a transformation at each blow-up which reduces the length of the parabolic filtration by one, ultimately yielding a pure dimension one sheaf on the final surface $Z_{\vxi}.$

Let us also highlight the distinguished features and difficulties that arise in the non-generic case. Conceptually, this iterative process works by encoding the parabolic structures in the exceptional divisors. In the generic case, where the parabolic structure on a Higgs bundle is uniquely determined by the ordered eigenvalues, this corresponds geometrically to the support of sheaves intersecting the exceptional divisors at a finite set of distinct points. In contrast, for non-generic $\vxi\in \nm$, a fixed Higgs bundle can be equipped with different parabolic structures (e.g. for $\Phi=0$, any flag is allowed). In this case, a curve in $\Sigmamxi$ may have components at the exceptional divisors. The different possible parabolic structures correspond to the variations of sheaves supported on the exceptional divisors. However, the presence of sheaves with support on the exceptional divisors also introduces some technical complications, which is the main challenge in extending the work of \cite{Diaconescu_2018} to the non-generic case. In particular, many exact sequences in the generic case arguments are no longer guaranteed to be exact. To address this, we reformulate the original construction by reducing it to the case of a single blow-up in Section \ref{s:sheaf local}, which allows for a more careful analysis. Our resulting correspondence for parabolic sheaves along a single blow-up is a key technical result of independent interest. Finally, we show that over the open locus of integral curves, these issues are resolved and  we obtain an isomorphism between the moduli spaces (see Theorem \ref{thm:open locus}). 

    \begin{rem}
       There is another version of the parabolic spectral correspondence developed in \cite{CDDP}, which considers holomorphic symplectic orbifolds instead of holomorphic symplectic surfaces. While this yields a spectral correspondence for a fixed $\xi \in \nm$, it appears difficult to extend this construction to the relative setting. This suggests that one may characterize the essential image of the relative embedding in terms of resolutions of the corresponding orbifolds. We will not pursue this approach in the present paper and leave it for future investigation. 
    \end{rem}

Furthermore, the spectral correspondence is compatible with two important structural morphisms on each side. On the Higgs side, for each $\vxi \in \nm$, there is the Hitchin map $h:\hvxi \to A:= \bigoplus_{\mu=1}^r H^0(C, K_C(D)^{\otimes \mu})$ that sends a $\vxi$-parabolic Higgs bundle to the characteristic polynomial of its Higgs field. Clearly, this cannot be surjective as there are constraints on the eigenvalues over each $p_i$. For strongly parabolic Higgs bundles ($\vxi=0$), the image of the Hitchin map has been studied previously by Baraglia--Kamgarpour \cite{baragliaKamgarpour} and Su--Wang--Wen \cite{Suwangwen2022} in the case $g \geq 2$. Consider the subsheaves of $\left(K_C(D)\right)^{\otimes \mu}$:
\begin{equation}\label{eq:L-sheaf}
    \lmmu
    :=\left(K_C(D)\right)^{\otimes \mu} \otimes \O_C\left(-\sum_{i\in I}  \gamma_{P^i}\left(\mu\right)p_i)\right)
\end{equation}
where $\gamma_{P^i}(\mu)$ is the level function associated to the partition $\um_i=P^i$ (we write $P^i$ to emphasize its role as a partition) for $i\in I$. Then they show that when $g \geq 2$, the image of strongly parabolic Hitchin map is given by 
\begin{equation*}
    A({\vec{m}})_0 = \bigoplus_{\mu=1}^r H^0\left(C, \lmmu\right) \subset A.
\end{equation*}
When $\um_i=(1,\dots,1)$ for all $i$, the image of the Hitchin map for $\vxi\neq 0$ can be described as the fibers of $A \to A/A({\vm})_0$, where each $\uxi_i$ is identified as a set of unordered eigenvalues which sits in $A/A({\vm})_0$ (see e.g. \cite{markmanspectral}\cite{logaresMartens}) However, this description does not generalize directly for partitions $\um_i\neq (1,\dots,1)$. Instead, we characterize the image of $\hvxi\to A$ from the viewpoint of $\mvxi$.

    On the surface side, for each $\vec{\xi} \in \nm$, there is the Fitting support map $\mvxi\to \bmxi$, which sends a pure dimension one sheaf to its Fitting support. Here, $\bmxi$ is defined as the affine subspace of effective curves in $|\Sigma({\vec{m}})_{\vec{\xi}}|$ lying in $S_{\vec{\xi}}$. In the relative setting, this extends to a  family of base spaces $\bm{B}({\vec{m}})$ over $\nm$, which gives rise to the relative support map
    \[\bm{supp}:\brelativesheaf \to \bm{B}({\vec{m}}).\] 
    To better understand the supports represented by curves in $\bmxi$, we provide a detailed characterization in terms of curves with prescribed multiplicities (determined by $\vm$) at the centers (determined by $\vxi$) of the sequence of blow-ups (Section \ref{subsec:linearsystem}). In turn, this leads to another characterization of $\bmxi$ in terms of the Hitchin base $A$ (Proposition \ref{evaluation description}), which will be used later in our study of non-emptiness and the Deligne--Simpson problem. Over $\nm$, this characterization gives rise to a map 
    \[\bm{\iota}: \bm{B}({\vm})\hookrightarrow \bm{A}(\vm):= A\times \nm.\]
    As a consistency check, we verify that the image $A({\vm})_0$ of the strongly parabolic Hitchin map described in \cite{baragliaKamgarpour}\cite{Suwangwen2022} is compatible with the base $B({\vm})_0$ (Corollary \ref{identification of hitchin bases}). 
    
    Now, we can compare the two structural morphisms on both sides. 
\begin{cor}[Corollary \ref{cor:sp commutativity}]\label{cor:intro relmaps}
        The relative spectral correspondence \eqref{eq:introrelsp} identifies the Hitchin map with the support map. In other words, there is a commutative diagram over $\nm$,
        \[\begin{tikzcd}
         \bm{\m{H}}(\vm) \arrow[r, hook, "\bm{Q}"]\arrow[d] & \bm{\m{M}}(\vm) \arrow[d]\\
        \bm{A}(\vm) & \bm{B}(\vm) \arrow[l, hook', "\iota"]
    \end{tikzcd}\]
    \end{cor}
\noindent Hence, it is natural to refer to $\bm{B}({\vm})\to \nm$ as the \emph{family of parabolic Hitchin bases associated to $\vm$}. 
\begin{rem}
  In general, the image of the Hitchin map for symplectic leaves of parabolic Higgs bundles is difficult to describe, as it is not defined by linear equations. However, the result above provides a linear description of this image, which may be of independent interest. 
\end{rem}

    \subsection{The $OK$ condition, flatness and non-emptiness}
    Although the bases $\bmxi$ have been defined and characterized in various ways, its non-emptiness is not immediately clear from these descriptions. For instance, determining whether there exists a curve passing through specified points on a given surface with prescribed multiplicities is far from trivial. Even if the non-emptiness is granted, it remains a priori unclear whether the dimension of $\bmxi$ behaves well as $\vxi$ varies, or whether it may jump in the non-generic case $\vxi$. Interestingly, these questions are closely related to the so-called \emph{OK condition} proposed by Balasubramanian--Distler--Donagi \cite{baladistlerdonagi} in the context of 6d superconformal field theories (SCFT). 

    In the study of 4d $\m{N}=2$ theories arising from compactification of 6d $(2,0)$ SCFTs on a punctured Riemann surface (base curve), the Coulomb branch of the theory is described by the base of a Hitchin system defined on the base curve. In \cite{baladistlerdonagi}, they consider the setup where the base curves vary over the Deligne-Mumford moduli space of stable pointed curves.  In this setting, they show that the $OK$ condition ensures that the family of Hitchin bases (hence Coulomb branches) vary in a well-behaved manner, fitting together into a vector bundle over the moduli space of curves. \begin{definition}
    We say that the collection of line bundles $L({\vm})_\mu$, for $\mu=2,\dots,r$ satisfies the \emph{$OK$ condition} if     \begin{equation}\label{eq:H1assumption}
        H^1\left(C, \lmmu\right)=0.
    \end{equation}
    If it is clear from the context, we simply say the $OK$ condition holds. 
\end{definition}
    While our base curve is fixed and the varying parameters are the eigenvalues, we find that the OK condition ensures the family of parabolic Hitchin bases behaves well in a family. It would be interesting to explore the connection between our work and theirs, as well as with the more recent work\cite{donagi2024meromorphichitchinfibrationstable}.
    \begin{prop}[Lemma \ref{affinebase} and Proposition \ref{constantdim}]
    Suppose that the expected dimension $\mathrm{expdim}(\bmxi)\geq 0$ and $\bmxi$ is non-empty for every $\vxi\in \nm$. If the OK condition holds,
     then the family of parabolic Hitchin bases $\bm{B}({\vm})\to \nm$ associated to $\vm$ forms an affine bundle. 
    \end{prop}
    In fact, the $OK$ condition is also sufficient for the non-emptiness of $\bmxi$. 
    
    \begin{theorem}[Theorem \ref{thm:nonemptiness of Hitchin bases}] \label{thm:intro-nonemptiness of Hitchin bases}
    Suppose that the OK condition holds. Then $\bmxi\neq \emptyset$ for any $\vec{\xi} \in \nm$. In particular, $\bmxi \neq \emptyset$ in the following cases:
    \begin{enumerate}
        \item When $n\geq 3$ and $g=0$, if the inequalities
        \begin{equation}\label{eq:intro-defect}
             \sum_{i=1}^n \gamma_{P^i}(\mu) < (n-2)\mu +2
        \end{equation}
        hold for $\mu=2, \dots, r$. 
        \item When $n \geq 1$ and $g = 1$, if at least one of the partitions $P^i$ is not the singleton partition $m_1=r$. 

        \item When $n\geq 1$, $g\geq 2$. 
    \end{enumerate}
    \end{theorem}
    
     Having established the non-emptiness of $\bmxi$, we can now use this result to deduce the non-emptiness of the moduli space $\hvxi$.  Specifically, to obtain an element in the $\hvxi$, it suffices to find an integral curve $\sm$ in $\bmxi$. By applying the spectral correspondence $\mvxi\to \hvxi$ to a line bundle on $\sm$, we obtain a stable $\vxi$-parabolic Higgs bundle in $\hvxi$. Given the advantage of having the parabolic Hitchin bases $\bmxi$ in families, we begin by first constructing such an element for the strongly parabolic case, where $\vxi=0$, and then deforming it to an integral curve in $\bmxi$ for other $\vxi\neq 0$. However, in order to guarantee the existence of integral curves, we need to slightly strengthen the $OK$ condition. See also Remark~\ref{rem:optimalcond nonempty} for a variation of the following proposition. 
  
    \begin{prop}[Proposition \ref{nonemptiness of Higgs}]
        For every $\vxi\in \nm$, the moduli space of stable $\vec{\xi}$-parabolic Higgs bundle $\hvxi$ is non-empty in the following cases
    \begin{enumerate}
        \item When $n\geq 3$ and $g=0$, if the inequalities
        \begin{equation}\label{eq:intro-defect}
             \sum_{i=1}^n \gamma_{P^i}(\mu) < (n-2)\mu +1
        \end{equation}
        hold for $\mu=2, \dots, r$
        \item When $n \geq 2$ and $g = 1$, if at least two of the partitions $P^i$ are not the singleton partition $m_{i,1}=r$.
        \item When $n\geq 1$ and $g\geq 2.$
    \end{enumerate}
    \end{prop}

    \subsection{Multiplicative Deligne--Simpson problem}\label{sec:intro DSP}
    As another application of the non-emptiness of $\bmxi$, we obtain a new approach to the classical multiplicative Deligne--Simpson problem (DSP for short), which asks the following question:
    \begin{ques}
       Given conjugacy classes $C_i$ in $GL_r(\C)$ for $i=1,\dots, n$, when does there exist  irreducible solutions to the equation $T_1\cdots T_n =  Id_r$ with $T_i\in C_i$? Here, an irreducible solution means that the matrices $T_j$ have no common invariant subspace. 
   \end{ques}
    
    This problem has been studied by many authors, starting from Simpson, who first obtained a sufficient and necessary condition in the case when one of the $C_i$ is regular i.e. semisimple with distinct eigenvalues \cite{simpson-productofmatrices}. We recall this in Appendix \ref{comparison result}. Katz later studied the rigid case by middle convolution \cite{katz}. A  sufficient and necessary condition in the case of multiplicatively generic eigenvalues (see Definition \ref{def:multiplicative generic}) was given by Kostov (see \cite{kostovsurvey} for a survey). More generally, Crawley-Boevey and Shaw gave a sufficient (conjecturally necessary) condition in terms of quivers without the genericity assumption on the eigenvalues via representations of certain multiplicative preprojective algebra \cite{crawleyboeveyshaw}. 

    Note that the DSP can also be reformulated as finding an irreducible local system on $\P^1\setminus \{p_1,\dots, p_n\}$ whose monodromy transformation $T_i$ around the puncture $p_i$ has the prescribed conjugacy class $C_i$ for $i=1,\dots,n$. So, a natural generalization of DSP is to replace $\P^1$ by curves of higher genus.  This higher genus analogue of DSP has been addressed in the work of Hausel--Letellier--Rodriguez-Villegas \cite[Section 5.2]{HLRV2013} \cite[Corollary 3.15]{letellier} for multiplicatively generic eigenvalues (see Definition \ref{def:multiplicative generic}), based on the quiver-theoretic approach of Crawley-Boevey and Shaw. 

In \cite[Appendix B]{baladistlerdonagi}, it is conjectured that the $OK$ condition is sufficient for solving the DSP ($g=0$) and the authors verify that the $OK$ condition is numerically equivalent to Simpson's criterion in the special case when one of the $C_i$ is semisimple with distinct eigenvalues. For completeness, we include a direct combinatorial proof of this numerical equivalence in Appendix \ref{comparison result}. 

Moreover, another challenge is to understand how the $OK$ condition leads to actual solutions of the DSP from a geometric viewpoint. Since the $OK$ condition guarantees the non-emptiness of $\bmxi$, our approach turns this into the necessary geometric input to construct such solutions.

Identify $C_i$ with a collection of partitions $\{P^{\lambda_{i,1}}, \dots, P^{\lambda_{i,e(i)}}\}$, labeled by the eigenvalues $\lambda_{i,j}$ of $C_i$. Then we define the following partition of $r$
\[
P^i := \widehat{P}^{\lambda_{i,1}} \cup \dots \cup \widehat{P}^{\lambda_{i,e(i)}}
\]
where $\widehat{P}^{\lambda_{i,j}}$ is the conjugate partition of $P^{\lambda_{i,j}}$. 

\begin{theorem}[Theorem~\ref{thm:DSP2}]
Let $n\geq 3$. Let $C_1, \dots, C_n \subset GL_r(\mathbb{C})$ be a collection of conjugacy classes whose collection of eigenvalues is multiplicatively generic. 
    Suppose that the following conditions hold: 
    \begin{enumerate}
    \item  $\prod_{i=1}^n\det(C_i)=1$.  
        \item $\sum_{i=1}^n \gamma_{P^i}(\mu) <(n-2)\mu+ 2$ for $\mu=2,\dots, r$.
    \end{enumerate}
    Then the DSP is solvable for the tuple of conjugacy classes $(C_1,\dots,C_n)$. 
\end{theorem}

This result answers the question posed in \cite{baladistlerdonagi}: the $OK$ condition \eqref{eq:H1assumption} is sufficient for solving the DSP under the multiplicatively genericity assumption (Remark \ref{rem:bdd conjecture}). Finally, aside from an exceptional case in $g=1$, this approach also recovers the result in \cite{HLRV2013}\cite{letellier} concerning the higher genus analogue of DSP (Theorem \ref{thm:DSP higher genus}) for $g\geq 1$. 

Viewing the DSP as the problem of finding irreducible local systems on $\P^1\setminus \{p_1,\dots,p_n\}$ with prescribed monodromies, we can apply the tame non-abelian Hodge correspondence (tame NAHC) developed by Simpson \cite{simpson-noncompact} to convert the DSP into the existence problem of stable parabolic Higgs bundles of parabolic degree $0$ with prescribed residues data of the Higgs field. In fact, the original approach of Simpson to the DSP in \cite{simpson-productofmatrices} proceeds by constructing parabolic Higgs bundles explicitly through systems of Hodge bundles. Here we proceed by constructing parabolic Higgs bundles through spectral correspondence. 

As in our approach to the non-emptiness of $\hvxi$, whenever $\bmxi$ is non-empty and contains at least an integral member $\sm\subset S_{\vxi}$, pushing a line bundle over $\sm$ forward to the base curve $C$  yields a stable parabolic Higgs bundle on $C$. Crucially, the geometry of the surface $S_{\vxi}$ together with the choice of curve class defining $\bmxi$ uniquely determines the conjugacy classes of the residues of the Higgs field at the marked points $p_i$. This is the necessary step to ensure that the corresponding local system has the prescribed conjugacy classes of monodromies. The local analysis of the conjugacy classes is studied in Lemma~\ref{prop:localconjugate} in the appendix. However, the construction of parabolic structures in this case differs from that of the relative spectral correspondence (Theorem~\ref{thm:intro-rel spec}) and some technical modification are required to achieve the desired conjugacy classes.

\begin{rem}
\begin{enumerate}
    \item The approach to solving the DSP in fact can be used to show the existence of stable filtered local systems with generic eigenvalues and filtered weights. This implies that the $OK$ condition is also sufficient for the existence of stable filtered local systems with prescribed residue diagrams. 
    \item  A spectral correspondence approach to the DSP when all the eigenvalues are zero was previously proposed by Wen \cite{wen2021parabolichiggsbundlesprojecitive}, building on the spectral correspondence developed in \cite{Suwangwen2022}. 
    \item  (Irregular DSP) The original parabolic spectral correspondence in \cite{Diaconescu_2018} is established for irregular parabolic Higgs bundles. By suitably generalizing their spectral correspondence, we expect our approach here should also provide solutions to the irregular version of DSP which addresses the existence of Stokes data with prescribed local data, as posed by Boalch in \cite[Section 9.4]{boalchStokes}. 
\end{enumerate}
\end{rem}

\subsection{Notations}\label{sec:notation}

\begin{itemize}
    \item $C$: a smooth projective curve of genus $g \geq 0$. 
    \item $D=p_1+\cdots+p_n$: a reduced effective divisor with $n \geq 1$. 
    \item $r$ is used to denote the rank of vector bundles.
    \item Parabolic data: 
    \begin{itemize}
        \item $\Vec{m}=(\um_1, \cdots, \um_n)$ where each component $\um_i=(m_{i,1}, \cdots, m_{i,\ell(i)})$ is a partition of $r$. In other words, \[m_{i,1} \geq m_{i,2} \geq \cdots \geq m_{i,\ell(i)}\quad \textrm{and  }\um_i=(m_{i,1}, \cdots, m_{i,\ell(i)})\]
        When we wish to emphasize its role as a partition instead of a vector, we will write $P^i= \um_i $.
        \item $\Vec{\xi}=(\uxi_1, \cdots, \uxi_n)$ where each component $\uxi_i=(\xi_{i,1}, \cdots, \xi_{i,\ell(i)})$ is a collection of eigenvalues.
        
    \end{itemize}
    \item (Distinct part of $\vxi$) The entries in $\uxi_i$ might be repeated, we define the distinct part $\uxi^\circ_i=(\xi_{i,1}^\circ, \cdots, \xi_{i,e(i)}^\circ)$ to be the tuple of distinct entries in $\uxi_i$, listed in the order of their first appearance.
  
    \item (Partitions and Young diagrams) We often represent a partition $P=(m_1,\dots,m_{\ell})$ using the Young diagram. We call $\ell$ the number of rows in the Young diagram, $m_1$ the number of columns, and $|P|=r=\sum_j m_j$ the number of boxes in the Young diagram.
    \item (Conjugate partition) To each partition $P= (m_1,\dots,m_\ell)$, we define the conjugate partition $\hat{P}:=(n_1,\dots, n_{m_1})$ by reflecting its Young diagram across the main diagonal. More formally, $n_j = \#\{k|m_k\geq j\}$ for $j=1,\dots, m_1.$
        \item (Level functions) To each Young diagram of $P$, we can write down two fillings of the boxes: (1) number the boxes in strictly increasing order starting from the top-left corner with 1, then proceed from top to bottom within each column and from left to right across columns (2) number the boxes in the $j$-th column by $j$. For example, 
\begin{center}\label{Young diagram}
        (1)\quad \begin{ytableau}
            1&4&6&8\\
            2&5&7\\
            3
        \end{ytableau} \qquad  (2)\quad \begin{ytableau}
            1&2&3&4\\
            1&2&3\\
            1
        \end{ytableau}
    \end{center}
    By pairing the numbers in the boxes (of the same position) of the two fillings, we define the level function $\gamma_P:\{1,\dots,r\} \to \{1,\dots, m_1\}, j\mapsto \gamma_P(j)$ associated to the Young diagram. In the example, we have 
    \begin{equation*}
    \gamma_P(1) =\gamma_P(2) = \gamma_P(3) = 1,\quad  \gamma_P(4)=\gamma_P(5)=2,\quad \gamma_P(6)= \gamma_P(7)=3 ,\quad \gamma_P(8) = 4.
    \end{equation*}

    \item (Level domain) Given a partition $P=(m_1, \dots ,m_\ell)$, we denote by 
    \[G(P) = \{(u,a)\in \Z^2|0\leq u<|P|, 0\leq a< \gamma_P(|P|- u)\}\]
    We call $G(P)$ the level domain associated to the partition $P$. 
    \item (Union of partitions) Given partitions $P^1$ and $P^2$, we define the union of the partitions $P^1\cup P^2$ by combining all the parts of both partitions and arranging them in non-increasing order. For example, $P^1= (2,1), P^2= (4,3,1,1,1)$ then $P^1\cup P^2 =(4,3,2,1,1,1,1)$.
    \item For each $i$, the parabolic data $(\um_i, \uxi_i)$ determines a decomposition of the partition $P^i$ into subpartitions indexed by $\uxi^\circ_i$. For each $\xi_{i,j}^\circ \in \uxi^\circ_i$, we define $P^{\xi_{i,j}^\circ}$ to be the subpartition of $P^i$ consisting of the collection of $m_{i,k}$'s such that $\xi_{i,k} = \xi_{i,j}^\circ$. Then we have 
    \[ P^i = P^{\xi^\circ_{i,1}}\cup\dots\cup P^{\xi^\circ_{i,e(i)}} \]
  
    \item (Example) Suppose we have a single marked point (so $n = 1$). Let $\ell(1) = 6$ and define
\[
\vm = \um = (3, 2, 2, 1, 1, 1), \quad \vxi= \uxi = (\xi_1, \xi_1, \xi_2, \xi_3, \xi_2, \xi_1).
\]
We also write $ P=\um = (3, 2, 2, 1, 1, 1)$. The \emph{distinct part} of $\uxi$ is
\[
\uxi^\circ = (\xi_1^\circ, \xi_2^\circ, \xi_3^\circ) = (\xi_1, \xi_2, \xi_3),
\]
So, $e(1) = 3$. The subpartitions of $P$ corresponding to each $\xi_j^\circ \in \uxi^\circ$:
\begin{itemize}
    \item $P^{\xi_1^\circ}=(m_1,m_2,m_6) = (3, 2, 1)$
    
    \item $P^{\xi_2^\circ}=(m_3,m_5)= (2, 1)$
    
    \item $P^{\xi_3^\circ}= (m_4) = (1)$
\end{itemize}
Then the full partition $P$ is decomposed as:
\[
P = P^{\xi_1^\circ} \cup P^{\xi_2^\circ} \cup P^{\xi_3^\circ}
\]
The level domain $P^{\xi_1^\circ}$ associated to the partition $P^{\xi_1^\circ}=(3,2,1)$ is visualized as 
\begin{center}
\begin{tabular}{c|cccccc}
$a=2$ & $\bullet$ &  & &&& \\
$a=1$ & $\bullet$ & $\bullet$& $\bullet$&&&\\
$a=0$ & $\bullet$ &$\bullet$ &$\bullet$&$\bullet$&$\bullet$&$\bullet$ \\
\hline
 & $u=0$ & $u=1$ & $u=2$ & $u=3$ & $u=4$ &$u=5$ \\
\end{tabular}
\end{center}
where each dot corresponds to a pair $(u,a)\in G(P^{\xi^\circ_i})$.

\end{itemize}
\section{Moduli spaces of parabolic Higgs bundles}
\subsection{Parabolic Higgs bundles}\label{sec:par Higgs}
Let $C$ be a smooth projective curve of genus $g\geq 0$ and $D= p_1+ \dots + p_n$ be a reduced effective divisor  with $n \geq1 $. Throughout this paper, we reserve the letter $r$ (resp. $d$) the rank (resp. degree) of Higgs bundles. Let $I=\{1, \cdots, n\}$ be the index set of marked points and for each point $p_i$, we denote the associated partition of $r$ by $\um_i=(m_{i,1}, \cdots ,m_{i,\ell(i)})$. We use poly-multivector notation and denote the collection of partitions of $r$ by $\Vec{m}=(\um_1, \cdots, \um_n)$. We write the index set of such partition by $J_i=\{1, \cdots, \ell(i)\}$ for $i \in I$. 

\begin{definition}\label{def:parabolicHiggs}
   A parabolic Higgs bundle on $C$ of type $\Vec{m}$ with poles at $D$ is a quadraple $(E, E^\bullet_D,\Phi,\vec{\a})$ where
    \begin{enumerate}
        \item A Higgs bundle $(E, \Phi)$ where $\Phi: E\to E\otimes K_C(D)$.
        \item A quasi-parabolic structure of type $\um_i$ at each $p_i$:
        \begin{equation*}
            E^\bullet_{p_i}:  0 = E^{\ell(i)}_{p_i}\subset E^{\ell(i)-1}_{p_i} \subset ... \subset E^{1}_{p_i}\subset E^0_{p_i} = E_{p_i}
        \end{equation*}
        such that $\Phi_{p_i}(E^j_{p_i}) \subset E^j_{p_i}\otimes_D {K_C(D)}_{p_i}$ and $\dim (E^{j-1}_{p_i}/E^{j}_{p_i})= m_{i,j}$. We simply refer this condition as $\Phi_D$ preserving $E_D^\bullet$, where $E_D^\bullet$ is denoted as a collection of filtrations $E_{p_i}^\bullet$ for all $i \in I$. 
        
        \item A collection of parabolic weights $\Vec{\alpha}=(\underline{\a}_1, \cdots, \underline{\a}_n)$ where $\underline{\a}_i= (\a_{i,1},...,\a_{i,\ell(i)})\in \Q^{\ell(i)}$: 
        \begin{equation*}
            1> \a_{i,\ell(i)} > \a_{i,\ell(i)-1}>...>\a_{i,1} \geq 0
        \end{equation*}
        for each $i \in I$. 
        \end{enumerate}
        \end{definition} 
        When $\Phi_{p_i}(E^{j-1}_{p_i}) \subset E^{j}_{p_i}\otimes_D K_C(D)_{p_i}$ for all $i,j$ in condition $(2)$, we call it a \emph{strongly parabolic Higgs bundle}.
    \begin{rem}
        We reverse the filtration indices used in our previous paper \cite{LeeLee2024}. This modification aligns with the spectral correspondence argument in \cite{Diaconescu_2018} for the relative setting.
    \end{rem}
    
     Recall that in \cite[Section 1]{Maruyama-Yokogawa} and \cite[Section 1]{Yokogawa-compactification},  the parabolic Euler characteristic and the (reduced) parabolic Hilbert polynomial of $E_*:= (E, E_D^\bullet,\vec{\a}) $ are defined as 
    \begin{equation*}
        \para \chi({E_*}) = \chi(E )+\sum_{i=1}^n\sum_{j=1}^{\ell(i)}\a_{i,j}\chi(E^{j-1}_{p_i}/E^{j}_{p_i}) , \qquad \para P_{E_*}(t) = \frac{\para \chi(E_*(t))}{\rk(E)}
    \end{equation*}
    where $E^{j-1}_{p_i}/E^{j}_{p_i}$ is viewed as a torsion sheaf on the curve $C$ in the expression and $E_*(t) = (E\otimes \O(t), E^\bullet_D\otimes \O(t),\vec{\a}).$  
    
    \begin{definition}
    A parabolic Higgs bundle $(E_*, \Phi) $ is said to be $\vec{\a}$-(semi)stable if for any nontrivial proper subbundle $0\subset F\subset E$ preserved by $\Phi$, we have 
    \begin{equation*}
        \para P_{F_*}(t) < \para P_{E_*}(t)  \quad \textrm{for } t\gg 0 \quad (\textrm{resp.  }\leq )
    \end{equation*}
    where $F_*= (F, F^\bullet_D, \vec{\a})$ is defined by the induced filtration $F^\bullet_D= F_D\cap E^\bullet_D$ that is preserved by $\Phi_D$. 
        
    \end{definition}
    \begin{rem}
        In the literature, the stability of parabolic Higgs bundles is often defined in terms of the usual slope by replacing the degree 
        with the parabolic degree $\mathrm{pardeg}(E):= \deg(E) + \a_{i,j}\sum \dim( E_{p_i}^{j-1}/E_{p_i}^j)$. It can be checked that the two definitions are equivalent. 
    \end{rem}
    Due to the  work of Yokogawa \cite[Theorem 4.6]{Yokogawa-compactification}, there exists a coarse moduli space for $\vec{\a}$-stable parabolic Higgs bundles of rank $r$, degree $d$ such that $\chi(E^{j-1}_{p_i}/E^{j}_{p_i})=m_{i,j}$ for all $i\in I,j\in J_i$. We will denote this moduli space by $\hpar(\vm)$. Moreover, one can define the coarse moduli space of $\vec{\a}$-stable strongly parabolic Higgs bundles as a closed subscheme in $\hpar(\vm)$, which we denote by $\hspar(\vm)$.

    One can fix the polar part $\Phi_D$ of the Higgs fields by choosing a collection of sections $\uxi_i=(\uxi_{i,1}, \cdots, \uxi_{i,\ell(i)}) \in  H^0(p_i, {K_C(D)}_{p_i})^{\times \ell(i)}$ for $i \in I$ and denote by $\vec{\xi}=(\uxi_1, \dots, \uxi_n)$ the collection of such sections.  We have the residue maps   
    \[\mathrm{res}_{p_i}:H^0(p_i,K_C(D)_{p_i})\to H^0(D, K_C(D)|_D) \to H^1(C, K_C)\cong \C\]
    We define the following: 
    \begin{equation*}
    \nm=\m{N}(\vec{m},D) := \left\{ \vec{\xi}= (\uxi_{i})_{i\in I}\left\vert \xi_{i,j}\in H^0(p_i, {K_C(D)}_{p_i}),  \sum_{i=1}^{n} \sum_{j=1}^{\ell(i)} \res_{p_i} (\xi_{i,j})= 0 \right.\right\}
\end{equation*}
\begin{definition}\label{def:xi-parabolic}\cite{Diaconescu_2018}
    For $\vec{\xi}\in \nm$, we call a parabolic Higgs bundle $(E ,E^\bullet_D, \Phi,\vec{\a})$ \textit{a (regular) $\vec{\xi}$-parabolic Higgs bundle} if the induced morphism of $\O_{p_i}$-modules $\textrm{gr}_j\Phi_{p_i}:= \Phi_{p_i, j}: E^{j-1}_{p_i}/E^{j}_{p_i}\to E^{j-1}_{p_i}/E^{j}_{p_i} \otimes {K_C(D)}_{p_i}$ satisfies the following condition
    \begin{equation}\label{xi-parabolic}
        \Phi_{p_i, j}= \mathrm{Id}_{E^{j-1}_{p_i}/E^{j}_{p_i}}\otimes \xi_{i,j}, \quad 1\leq j\leq \ell(i), \quad i \in I.
    \end{equation} 
\end{definition}
    Note that a $\vxi$-parabolic Higgs bundle $(E, E^\bullet_D,\Phi,\va)$ induces a rank one Higgs bundle $(\det(E), \Phi_0)$ where $\Phi_0:=\Phi\wedge \dots \wedge \mathrm{Id}+ \dots +\mathrm{Id}\wedge\dots\wedge \Phi\in H^0(C, K_C(D))$. By restricting $\Phi_0$ to $D$, we see that the identity $\sum_{i=1}^{n} \sum_{j=1}^{\ell(i)} \res_{p_i} (\xi_{i,j})= 0$ must hold. A priori, each $\xi_{i,j}\in H^0(p_i, K_C(D)_{p_i})$ is a section. When the context is clear, we will, by slight abuse of notation, identify $\xi_{i,j}$ with its residue $\mathrm{res}(\xi_{i,j})\in \C$.

In \cite{LeeLee2024}, we construct a relative moduli space of $\vec{\xi}$-parabolic Higgs bundles on $C$. 

    \begin{theorem}\cite{LeeLee2024}\label{existence-of-coarse-moduli}
        Fix the numerical data: a rank $r\geq 1$, a degree $d\in \Z$ and a collection of partitions $\vec{m}$ with a parabolic structure $\vec{\a}$. There exists a relative coarse moduli scheme $\bm{\pi}_\m{N}:\bm{\m{H}}(C,D;r,d,\vec{\a}, \vec{m})\to \nm$. In fact, every fiber of $\bm{\pi}_\m{N}$ is smooth.
        \end{theorem}
    
    \noindent If it is clear from the context, we abbreviate to write this moduli space as $\bm{\m{H}}(\vm)$.

    \begin{exmp}
        When $n=1$ and $\um_1$ is the trivial flag $(\um=(r))$, the base $\nm=\{0\}$ so that $\bm{\m{H}}(\vm)$ consists of stable holomorphic Higgs bundles $(E,\Phi)$.
    \end{exmp} 
    \begin{exmp}
        When $\vec{m}=\vec{1}$, meaning that each $\um_i$ is the full flag, the moduli space $\bm{\m{H}}(\vm)$ becomes isomorphic to the moduli space of parabolic Higgs bundles $\hpar(\vec{1})$ \cite[Proposition 2.8]{LeeLee2024}.
    \end{exmp}

  \begin{rem}\label{r:comparison}
    For a fixed partition $\vec{m}$, there is a morphism $G(\vm):\bm{\m{H}}(\vm)\to \hpar(\vm)$ defined to be the composition 
    \[
    G(\vec{m}): \bm{\m{H}}(\vm) \hookrightarrow \hpar(\vm) \times \nm \twoheadrightarrow \hpar(\vm)
    \]
    where the first morphism is a canonical inclusion and the second morphism is the projection. In general, the morphism $G(\vec{m})$ is not surjective because the block diagonal part of a Higgs field $\Phi_D$ over $D$ may not be diagonal. 
    \end{rem}

\subsection{Parabolic Hitchin maps and the $OK$ condition}
The Hitchin morphisms for the moduli of parabolic Higgs bundles have been considered before \cite{Yokogawa-compactification}, \cite{logaresMartens}, \cite{baragliaKamgarpour}, \cite{Suwangwen2022}. Define the Hitchin base 
\begin{equation}\label{usual hitchin base}
    A = \bigoplus_{\mu=1}^r H^0(C, (K_C(D))^{\otimes \mu}).
\end{equation}
The Hitchin map $h(\vm)^{\rm{par}}: \hpar(\vm) \to A $ is defined by sending a parabolic Higgs bundle to the coefficients of the characteristic polynomial of the underlying Higgs bundle with coefficients in $K_C(D)$. 

As in the introduction, consider the affine subspace 
\begin{equation*}
    A(\vm)_0 = \bigoplus_{\mu=1}^r H^0\left(C, \lmmu\right) \subset A
\end{equation*}
where $\lmmu$ is defined in \eqref{eq:L-sheaf}. Recall that the set of line bundles $\lmmu$, for $\mu=2,\dots,r$ satisfies the \emph{OK condition} if $H^1(C, \lmmu)=0$ for all $2 \leq \mu \leq r$. By Serre duality, one can see that this condition holds when 
\begin{enumerate}
    \item $g=0, n \geq 3$, $\sum_{i=1}^n \gamma_{P^i}(\mu) < (n-2)\mu +2$;
    \item $g \geq 1, n \geq 1$, except in the case when $g=1$ and $\gamma_{P_i}(\mu)=\mu \text{ for all $i$ }$. 
\end{enumerate}

In particular when $g \geq 2$, it is shown in the work of \cite{baragliaKamgarpour}, \cite{Suwangwen2022} that the image of $\hspar$ under the Hitchin map is $A(\vm)_0$. We will denote the corresponding Hitchin map by $h(\vm)^{\rm{s-par}}: \hspar(\vm) \to A(\vm)_0$.

Combining our previous discussion, we get the following commutative diagram
\[
    \begin{tikzcd}
        \bm{\m{H}}(\vec{m})\arrow[r, "G(\vec{m})"]\arrow[dd]&\hpar({\vec{m}})\arrow[d,"h^{\rm{par}}"] &\\
        &A\arrow[d, two heads, "\mathrm{pr}_{\vec{m}}"] \\
        \m N(\vec{m}) \arrow[d, "\iota_{\vec{m}}"] & A/A(\vec{m})_0 \arrow[d, two heads, "\mathrm{pr}^{\vec{1}}_{\vec{m}}"]\\
        \m N(\vec{1})\arrow[r, "/\mathfrak{S}_n"]&  A/A({\vec{1}})_0 
    \end{tikzcd}
\]
where $\mathrm{pr}$'s are the canonical projections and the the bottom isomorphism follows from the $OK$ condition. For later use, we denote by $\mathrm{pr}_{\vec{1}}$ the composition $\mathrm{pr}^{\vec{1}}_{\vec{m}} \circ \mathrm{pr}_{\vec{m}}$.

\section{Family of surfaces}\label{sec:family of surfaces}
\subsection{Construction of holomorphic symplectic surfaces}\label{sec:construction of surfaces}
The goal of this section is to introduce a family of surfaces associated to the data $(C,D,\vec{m},\vec{\xi})$. Consider the projectivization $M:=\mathbb{P}(K_C(D) \oplus \mathcal{O}_C)$ of the twisted canonical line bundle $K_C(D)$. We write $M^\circ=\mathrm{Tot}(K_C(D))$ for the total space of the twisted line bundle.\footnote{While it is more common to denote $\mathrm{Tot}(K_C(D))$ by $M$ and $\mathbb{P}(K_C(D)\oplus \mathcal{O}_C)$ by $\overline{M}$, we adopt a different convention here, as our focus is primarily on the projectivized space.} We will use the bold face to indicate the objects in families. For example, let $\bm{M}:=M \times \nm$ (resp. $\bm{C}:=C \times \nm$) be the trivial family of the projective surfaces (resp. the base curves) over $\nm$. We denote the projection to $\nm$ by $\bm{\pi}_{\m{N}}$. We also write $\bm{\pi_C}:\bm{M} \to \bm{C}$ for the canonical projection over $\nm$. There are tautological sections $\{\bm{\xi}_{i,j}=\xi_{i,j} \times \nm|i \in I, j \in J_i\}$ of $\bm{\pi}_{\m{N}}$ where we abuse to write $\xi_{i,j}(p_i)$ simply as $\xi_{i,j}$. These are non-singular codimension two subschemes in $\bm{M}$ that are flat over $\nm$ whose fiber at $\vec{\xi}$ is $(\xi_{i,j}, \vec{\xi}) \in \bm{M}$. 

We construct a family of projective surfaces over $\nm$ as follows. We begin with $\bm{\uxi_1}=(\bm{\xi}_{1,1}, \cdots, \bm{\xi}_{1,\ell(1)})$. Blow up $\bm{M}$ along the subscheme  $\bm{\xi}_{1,1}$. Denote the resulting blow-up by $\bm{p}_{1,1}:\bm{M}_{1,1} \to \bm{M}$ and the exceptional divisor by $\bm{E}_{1,1}$. Next, we take the strict transform of $\bm{\xi}_{1,2}$ in $\bm{M}_{1,1}$ which we denote by $\bm{\xi}_{1,2}$ as well. Then we further blow up $\bm{M}_{1,1}$ along $\bm{\xi}_{1,2}$. Denote the resulting blow-up by $\bm{p}^{1,1}_{1,2}:\bm{M}_{1,2} \to \bm{M}_{1,1}$  and the exceptional divisor by $\bm{E}_{1,2}$. Let $\bm{p}_{1,2}=\bm{p}^{1,1}_{1,2} \circ\bm{p}_{1,1}$. We proceed this iteratively until the $\ell(1)$-th step and obtain the following sequence of blow-ups
    \begin{equation*}
        \begin{tikzcd}
            \bm{Z}_1:= \bm{M}_{1,\ell(1)} \arrow[r, "\bm{p}^{1,\ell(1)-1}_{1,\ell(1)}"] \arrow[rrrd,swap, "\bm{p}_{1,\ell(1)}"]& \cdots \arrow[r, "\bm{p}^{1,2}_{1,3}"] & \bm{M}_{1,2} \arrow[r,"\bm{p}^{1,1}_{1,2}"] \arrow[rd,"\bm{p}_{1,2}"]& \bm{M}_{1,1} \arrow[d,"\bm{p}_{1,1}"] \\
            &&& \bm{M}
        \end{tikzcd}
    \end{equation*}

Next, we apply the same construction for $\bm{\uxi}_{2}$. Since $p_1 \neq p_2 \in C$, the strict transform of $\bm{\xi}_{2,j}$ along $\bm{p}_{1,\ell(1)}$ is the same as the pullback. Thus, we abuse to denote $\bm{\xi}_{2,j}$ its inverse image, and apply the same construction. Then we get a sequence of blow-ups 
\begin{equation*}
    \begin{tikzcd}
            \bm{Z}_2:= \bm{M}_{2,\ell(2)} \arrow[r, "\bm{p}^{2,\ell(2)-1}_{2,\ell(2)}"] \arrow[rrrd,swap, "\bm{p}^{1,\ell(1)}_{2,\ell(2)}"]& \cdots \arrow[r, "\bm{p}^{2,2}_{2,3}"] & \bm{M}_{2,2} \arrow[r,"\bm{p}^{2,1}_{2,2}"] \arrow[rd,"\bm{p}^{1, \ell(1)}_{2,2}"]& \bm{M}_{2,1} \arrow[d,"\bm{p}^{1, \ell(1)}_{2,1}"] \\
            &&& \bm{M}_{1,\ell(1)}
    \end{tikzcd}
\end{equation*}
By applying this construction for the rest of the tautological sections $\bm{\uxi}_3, \cdots, \bm{\uxi}_n$ iteratively, we obtain a tower of blown-up surfaces 
\[\bm{p}^{i,j}_{i',j'}:\bm{M}_{i',j'} \to \bm{M}_{i,j}\]
for $i' >i$ or $i'=i, j'>j$ such that $\bm{p}^{i',j'}_{i'',j''} \circ \bm{p}^{i,j}_{i',j'}=\bm{p}^{i,j}_{i'',j''}$ for all valid pairs. We also write a canonical projection to $\bm{M}$ as $\bm{p}_{i,j}:\bm{M}_{i,j} \to \bm{M}$ and $\bm{f}_{i,j}:=\bm{\pi_C} \circ \bm{p}_{i,j}$. For later use, we distinguish the resulting surface by writing it as $\bm{Z}=\bm{Z}_n=\bm{M}_{n, \ell(n)}$ and the projections by $\bm{p}=\bm{p}_{n,\ell(n)}:\bm{Z} \to \bm{M}$ and $\bm{f}=\bm{f}_{n,\ell(n)}:\bm{Z} \to \bm{C}$.

\begin{exmp}\label{exmp:blownup surface generic}(Generic $\xi$)
Consider the case $n=1$ with $D=p$. When $\vec{\xi}=\uxi$ is generic i.e. $\xi_1,\dots, \xi_\ell$ are mutually distinct, the restriction of $\bm{p}:\bm{Z}\to \bm{M}$ to $\uxi \in \nm$, is given by the blow-up of $M$ at the distinct points $\xi_1,\dots, \xi_\ell$. See Figure \ref{fig:gensurf}.
    \begin{figure}[ht]
	\centering
	\begin{tikzpicture}

        \draw[thick] (-2,-2) -- (-2,3);
        \fill (-2,-2.3) node {$M_p$};
        \draw[thick] (-3,-1) -- (4,-1);
	\draw[thick] (-2.5,2) -- (0.5,3);
        \draw[thick] (-2.5,0) -- (0.5,1);
        \fill (-0.5,1.8) node {$\vdots$};
        \fill (1,3) node {$E_1$};
        \fill (-2.3,2.5) node {$\xi_1$};
        \fill (1,1) node {$E_{\ell}$};
        \fill (-2.3,0.5) node {$\xi_\ell$};
        \fill (4.5,-1) node {$C_0$};
	\end{tikzpicture}
	\caption{\scshape Surface $Z_{\uxi}$ for a generic $\uxi$ \label{fig:gensurf}}
\end{figure}
\end{exmp}

\begin{exmp}\label{exmp:blownup surface strgly par}(Strongly parabolic $\xi=0$)
Consider the case $n=1$ with $D=p$. When $\uxi=0$, the restriction of $\bm{p}:\bm{Z}\to \bm{M}$ to $\uxi \in \nm$ is given by the successive blow-up of $M$ along the strict transform of $M_p$ and the previous exceptional divisor. See Figure \ref{fig:sparsurf}.

    \begin{figure}[ht]
	\centering
	\begin{tikzpicture}
        \draw[thick] (-2,-2) -- (-2,3);
        \fill (-2,-2.3) node {$M_p$};
	\draw[thick] (-2.5,0) -- (-0.5,1);
        \draw[thick] (-1,1) -- (1,0);
        \draw[thick] (2,0) -- (4,1);
        \draw[thick] (3.5,1) -- (5.5,0);
        \draw[thick] (4.5,-1) -- (6.5,3);
        \fill (1.5,0.5) node {$\cdots$};
        \fill (-1.5,0) node {$E_\ell$};
        \fill (-2.3,0.5) node {$0$};
        \fill (0,0) node {$E_{\ell-1}$};
        \fill (3,0) node {$E_{2}$};
        \fill (4.5,0) node {$E_{1}$};
        \fill (7,2) node {$C_0$};
	\end{tikzpicture}
	\caption{\scshape Surface $Z_{\uxi}$ for $\uxi=0$ \label{fig:sparsurf}}
\end{figure}
\end{exmp}

Next, we study the basic intersection theory of each fiber $Z_{\vec{\xi}}$. Recall that we have an exceptional divisor $\bm{E}_{i,j}$ in $\bm{M}_{i,j}$ for $i \in I, j \in J_i$. Define the pullback of such exceptional divisor to $\bm{Z}$ as $\bm{\Xi}_{i,j}=(\bm{p}^{i,j}_{n, \ell(n)})^*\bm{E}_{i,j}$. Denote the restrictions to $Z_{\vec{\xi}}$ by $E_{i,j,\vec{\xi}} = \bm{E}_{i,j}|_{Z_{\vec{\xi}}}$ and $\Xi_{i,j,\vec{\xi}}= \bm{\Xi}_{i,j}|_{Z_{\vec{\xi}}}$. We let $C_0$ and $C_\infty$ be the zero section and the infinity section of $\pi:M \to C$, respectively. Also, denote by $F_{i,\vxi}$ and $\td{C}_{\infty,\vxi}$ the strict transform of the fiber $M_{p_i}$ and the infinity section $C_\infty$ in $Z_{\vxi}$, respectively.
\begin{lem}
For each $\vec{\xi}\in \nm$, we have 
    \begin{equation*}
        \mathrm{Num}(Z_{\vxi}) = \mathrm{Num}(M)\bigoplus \bigoplus_{i\in I, j \in J_i} \Z[\Xi_{i,j,\vxi}]
    \end{equation*}
\end{lem}

\begin{lem}\label{intersectionpattern}
For each $\vec{\xi}\in \nm$, we have 
    \begin{equation*}
        \Xi_{i,j, \vxi}\cdot \Xi_{i',j',\vec{\xi}} = \begin{cases}
            -1  &\textrm{if  } (i,j)=(i',j')\\
            0  &\textrm{if  } \text{ otherwise}
        \end{cases}, \quad \Xi_{i,j, \vxi} \cdot F_{k,\vec{\xi}}= \begin{cases}
            1 &\textrm{if  } i=k \\
            0 &\textrm{if  } i \neq k
        \end{cases}, \quad \Xi_{i,j, \vxi}\cdot p_{\vxi}^*C_{0} = 0.
    \end{equation*}
    \end{lem}
    \begin{proof}
    For simplicity, we drop out $\vec{\xi}$ in the notation. We prove the first assertion. If $i \neq i'$, then it is clear that $\Xi_{i,j}$ does not intersect with $\Xi_{i',j'}$
        If $i=i'$ and $j\neq j'$, say $j>j'$, we have 
            \begin{align*}
                \Xi_{i,j} \cdot \Xi_{i',j'} = (p^{i,j})^*E_{i,j}\cdot (p^{i',j'})^*  E_{i',j'} = E_{i,j}\cdot (p_{i,j}^{i',j'})^*E_{i',j'} =0
            \end{align*}
            If $(i,j)=(i',j')$, then clearly $(\Xi_{i,j})^2= ((p^{i,j})^*E_{i,j})^2= (E_{i,j})^2 =-1$. For the second computation, note that 
            \begin{align*}
                \Xi_{i,j}\cdot F_k = \Xi_{i,j} \cdot \left( p^* M_{p_k}- \sum_{j'=1}^{\ell(k)} (p^{k,j'})^* E_{k,j}  \right) = \Xi_{i,j}\cdot\left( p^* M_{p_k}- \sum_{j'=1}^{\ell(k)} \Xi_{k,j'}  \right)
            \end{align*}
        Due to the first assertion, this becomes $1$ when $k=i$. The final equality is clear $\Xi_{i,j}\cdot p^*C_0 = E_{i,j}\cdot p_j^*C_0=0$. This lemma implies that the intersection pattern of the $\Xi_{i,j,\vec{\xi}}$'s is independent of $\vec{\xi}$ and behaves similarly to the generic case. 
    \end{proof}
    For the purpose of establishing the relative spectral correspondence, we will be interested in the following family of (non-compact) holomorphic symplectic surfaces. Let $\bm{F}_i$ and $\bm{\widetilde{C}_\infty}$ denote the family of the strict transforms of $M_{p_i}$ and $C_\infty$, respectively. We define $\bm{S}$ to be the complement of the divisors $\bm{F}_i$'s and $\bm{\widetilde{C}_\infty}$ in $\bm{Z}$. The next proposition shows that $\bm{S}$ is a family of holomorphic symplectic surfaces. 
\begin{prop}
    For each $\vec{\xi}\in \nm$, the non-compact surface $Z_{\vec{\xi}}\setminus (F_{\vec{\xi}}\cup \td{C}_{\infty,\vec{\xi}})$ is holomorphic symplectic. 
\end{prop}
\begin{proof}
    It is easy to check that the canonical divisor of $K_{Z_{\vxi}}$ is given by \[K_{Z_{\vec{\xi}}}=-p_{\vec{\xi}}^*M_{p_1} - \cdots -p_{\vec{\xi}}^*M_{p_n} -   2p_{\vec{\xi}}^*C_\infty+\sum_{i\in I}\sum_{j \in J_i} \Xi_{i,j,\vxi} = -F_{1,\vec{\xi}}- \cdots -F_{n,\vec{\xi}}-2\td{C}_{\infty,\vec{\xi}}\] since $F_{i,\vec{\xi}}\sim p_{{\vec{\xi}}}^*M_{p_i} - \sum_{j\in J_i} \Xi_{i,j,\vxi}$ for $i \in I$. 
\end{proof}

\subsection{Linear systems and Hitchin bases}\label{subsec:linearsystem}
Since we will be considering pure dimension one sheaves on the holomorphic symplectic surfaces $S_{\vxi}$ constructed by removing the divisors $F_{i,\vxi}$ and $\td{C}_{\infty,\vxi}$ from $Z_{\vxi}$, the support of these sheaves must also be disjoint from these divisors. The next proposition characterizes the curve classes that are disjoint from these divisors. 
\begin{prop}\label{p:class}
         A class $\Sigma$ in $\mathrm{Num}(Z_{\vec{\xi}})$ that satisfies 
        \begin{equation*}
           \Sigma\cdot \td{C}_{\infty,\vec{\xi}} =0, \qquad \Sigma\cdot F_{i,\vec{\xi}}=0 \quad \text{for } i \in I 
        \end{equation*}
        is of the form $ap_{\vec{\xi}}^*C_{0} + \sum_{i\in I}\sum_{j \in J_i} c_{i,j}\Xi_{i,j,\vec{\xi}}$ where $a= -\sum_{j \in J_i} c_{i,j}$ for any $i \in I$. Moreover, if we require 
        \begin{equation*}
            \Sigma\cdot \Xi_{i,j,\vec{\xi}}= m_{i,j}
        \end{equation*}
        for all $i \in I, j \in J_i$, then $\Sigma$ is of the form $rp_{\vec{\xi}}^*C_0 -\sum_{i\in I}\sum_{j \in I_i} m_{i,j} \Xi_{i,j,\vec{\xi}}$, where $r=\sum_{j\in J_i}m_{i,j}$ for any $i \in I$.
    \end{prop}
    \begin{proof}
    A class $\Sigma\in \mathrm{Num}(Z_{\vec{\xi}})$ is of the form $ap_{\vec{\xi}}^*C_0+ bp_{\vec{\xi}}^*M_p + \sum_{i\in I}\sum_{j \in J_i} c_{i,j}\Xi_{i,j,\vec{\xi}}$ where $p \neq p_i$ for $i \in I$. Then by Lemma \ref{intersectionpattern}
    \begin{equation*}
        0= \Sigma\cdot \td{C}_{\infty,\vxi} =  \Sigma\cdot p_{\vec{\xi}}^*C_{\infty,\vxi} = b\cdot M_p\cdot C_{\infty,\vxi} \implies b=0
    \end{equation*}
    Moreover, 
    \begin{equation*}
        0= \Sigma \cdot F_{i,\vxi} = a C_0\cdot M_{p_i} + \sum_{i\in I}\sum_{j \in J_i} (c_{i,j} \Xi_{i,j,\vxi}\cdot F_{i,\vxi}) = a+ \sum_{j \in I} c_{i,j} = 0. 
    \end{equation*}
    The last condition on $\Sigma\cdot \Xi_{i,j,\vxi}$ implies that $-c_{i,j}=m_{i,j}.$
    \end{proof}
Hence, we will consider the following effective relative Cartier divisor \[\bm{\sm}(\vec{m}):= r\bm{p}^*\bm{C}_0 -\sum_{i \in I}\sum_{j \in J_i}  m_{i,j} \bm{\Xi}_{i,j}\] on $\bm{Z}$ and the linear systems formed by its restrictions to $Z_{\vxi}$
\[\Sigmamxi= rp_{\vec{\xi}}^*C_0 - \sum_{i \in I}\sum_{j \in J_i}  m_{i,j} {\Xi}_{i,j,\vxi}\]

The complete linear system $|\Sigmamxi|$ admits a useful description as a subset of $|rC_0|$. To see this, recall the following elementary fact: Let $X$ be a smooth projective surface and $D\subset X$ be a divisor. Let $\pi: X'\to X$ be the blow-up of $X$ at a point $x\in X$ and $E\subset X'$ be the exceptional divisor. Recall that we can identify the complete linear system $|\pi^* D- mE|$ as the sublinear system of $|D|$ consisting of divisors $D'\in |D|$ which pass through $p$ with multiplicity $\geq m$ by mapping $D'$ to $\pi^*D' - mE$ which is effective. Applying this iteratively, we can identify the linear system $|\Sigmamxi|$ with a projective subspace in $ |rC_0|$ for any given $\vm$ and $\vxi$. 

\begin{exmp}\label{exmp:linearsystem1}
Consider the case $n=1$ and $D=p$ for simplicity. Let $\vxi =\uxi=  (\xi_1,\dots,\xi_\ell)\in \C^{\times \ell}$ and  $\vm =\um=  (m_1,\dots, m_\ell)\in \Z^{\times \ell}$. Then $Z_{\vxi}$ is constructed as a sequence of blow-ups
\[Z_{\vxi}= M_{\ell} \xrightarrow{p_\ell^{\ell-1}} \dots \to M_{1} \xrightarrow{p_1^0} M \quad \textrm{with the exceptional divisors }E_j\subset M_j. \]
We also have a chain of complete linear systems: Let $L_0 := |rC_0|$ on $M$, $L_1 :=|R_1|$ on $M_1$ where $R_1:=(p_1^0)^*(rC_0)-m_1E_1$ and for $j=2,\dots, \ell$, let $L_j = |R_j|$ on $M_j$ where $R_j:= (p_{j}^{j-1})^* R_{j-1}-m_jE_j$. Then we have 
\begin{equation*}\label{inclusions}
    L_\ell\xhookrightarrow{i_\ell}\dots \xhookrightarrow{i_2} L_1 \xhookrightarrow{i_1}L_0 
    \end{equation*}
    where the inclusion $i_j: L_j \to L_{j-1}$ identifies $L_j$ with the subset of $L_{j-1}$ consisting of effective divisors in $L_{j-1}$ that passes through the center of the blow-up $p_{j}^{j-1}:M_{j}\to M_{j-1}$ with multiplicity $\geq m_j$. In particular, $L_\ell = |\Sigmamxi|$ can be seen as a projective subset of $|rC_0|$ consisting of effective divisors with prescribed multiplicities at the points in $M_j$ determined by $\vxi$ and $\vm$. 
    \end{exmp}
    
\begin{exmp}\label{exmp:linearsystem2}
       Continuing with the previous example above, suppose now that $\xi_1=\dots = \xi_\ell$. In this case, the center $c_j$ of the blow-up $p_j^{j-1}:M_j\to M_{j-1}$ lies in the exceptional divisor $E_j$ (see Figure \ref{fig:sparsurf}).
       Let $\iota : |\Sigmamxi|\to |rC_0|$ denote the inclusion described above.  If $X\in i_1(L_1)$, then $X$ passes through $c_1$ with multiplicity at least $m_1$. This implies that the total transform   $(p_1^0)^*X  = X^1 + m_1E_1$ where $X^1\in L_1$ is a divisor in $M_1$. If, further, $X\in i_1\circ i_2(L_2)$, then $X^1\in i_2(L_2)$ and $X^1$ passes through $c_2$ with multiplicity at least $m_2$. Thus, the total transform $(p_2^0)^*X =  (p_2^1)^*(X^1+m_1E_1)$ passes through $c_2$ with multiplicity at least $m_1+m_2$ where $p_2^0= p_2^1\circ p_1^0$. Proceeding it inductively, we find that if $X\in \iota(L_\ell)$, then for each $j=2,\dots,\ell$, the total transform of $X$ in $M_{j-1}$ passes through $c_j$ with multiplicity at least $m_1+\dots +m_j$. In fact, this property can also be used as a characterization of $X\in \iota (|\Sigmamxi|).$
      
\end{exmp}
\begin{exmp}\label{exmp:linearsystem3}
            In general, the argument in Example \ref{exmp:linearsystem2} works for $\vxi= \uxi= (\xi_1,\dots,\xi_{\ell})$ with repeated entries since the conditions imposed by $|\Sigmamxi|$ is local. Let $\{P^{\xi^\circ_1}, \dots ,P^{\xi^\circ_e}\}$ be the collection of partitions labeled by the reduced vector of distinct eigenvalues $(\xi_1^\circ,\dots, \xi_e^\circ)$ determined by the parabolic $(\vm,\vxi)$. For example, if $P^{\xi_1^\circ} = (m_{j_1},\dots, m_{j_s})$ and let $c_{j_a}$ be the center of the blow-up $M_{j_a}\to M_{j_a-1}$ which lies in the exceptional divisor $E_{j_a}$, where $a=1,\dots, s$. Then $X\in \iota (|\Sigmamxi|)$ satisfies the following local condition:
            \begin{center}
                $X$ passes through $c_{j_1}$ with multiplicity at least $ m_{j_1}$ and, for $a=1,\dots ,s$, the total transform of $X$ in $M_{j_a}$ passes through $c_{j_a}$ with multiplicity at least $m_{j_1}+\dots + m_{j_a}$.
            \end{center}
             The analogous condition holds for other $\xi_{j}^\circ$, $j=2,\dots, e.$ Hence, a divisor $X$ in $ |\Sigmamxi|$ can be characterized completely by these local multiplicity conditions over the centers associated to each $\xi_j^\circ$. 
        \end{exmp}

\begin{definition}
    For each $(\vec{m}, \vec{\xi})$, define $\bmxi$  as the subset of $|\Sigmamxi|$ consisting of effective divisors in $Z_{\vxi}$ that are compactly supported away from $F_{i,\vxi}$ for all $i \in I$, as well as from  $\td{C}_{\infty,\vec{\xi}}$.
\end{definition}

\begin{lem}\label{affinebase}
$\bmxi$ is an (open) affine space. 
    
\end{lem}
\begin{proof}
Recall that (see e.g. \cite[Remark 1.1]{pantevkouvidakis} and also the discussion after this proof) the subset $W\subset |rC_0|$ consisting of effective divisors that touch or contain $C_\infty$ is a projective subspace of codimension $1$ so that the complement $W^c$ of $W$ is an affine space in the linear system $|rC_0|$, which can be identified with the usual Hitchin base $A$ introduced in \eqref{usual hitchin base}. Then an effective divisor $Q\in W$ if and only if $Q\cap C_\infty\neq \emptyset$ if and only if $C_\infty \subseteq Q$. 

 Let $\iota: |\Sigmamxi|\to |rC_0|$ be the inclusion described above. Since the centers of the successive blow-up construction of $Z_{\vxi}$ are away from $C_\infty$ (and its strict transforms), the subset of effective divisors that are disjoint from $\td{C}_{\infty,\vxi}$ can be identified with $|\Sigmamxi|\cap \iota^{-1}(W^c)$ which is an affine space. We claim that $|\Sigmamxi|\cap \iota^{-1}(W^c) = \bmxi$. Clearly, we have $\bmxi\subset |\Sigmamxi|\cap \iota^{-1}(W^c)$. For the reverse inclusion, note that $X\in  |\Sigmamxi|\cap 
\iota^{-1}(W^c)$ implies $X\cdot F_{i,\vxi}=\Sigmamxi \cdot F_{i,\vxi}=0$ and $X$ does not  contain the strict transform $F_{i, \vxi}$ as an irreducible component, otherwise $X$ would intersect $\td{C}_{\infty, \vxi}$ non-trivially. 

\end{proof}

As the linear system $|\Sigmamxi|$ can be described in terms of $|rC_0|$, $\bmxi$ also admits a description in terms of the Hitchin base $A$. First, let us recall the identification between Hitchin base and the (open subset of) linear system $|rC_0|$ on the ruled surface $M$ (see e.g. \cite[Section 1.1]{pantevkouvidakis} for details). 
Let $y\in H^0(M,\pi^*(K_C(D))\otimes \O_{M}(1))$ be the zero section and $w\in H^0(M,\O_M(1))$ be the infinity section. So, we have $Div(y)=C_0$ and $Div(w)=C_\infty$. There are natural isomorphisms $H^0(M, rC_0)\cong H^0\left(M, \pi^*\left(K_C(D)\right)^{\otimes r}\otimes \O_M(r)\right) $ and
\begin{align*}
     H^0\left(M, \pi^*\left(K_C(D)\right)^{\otimes r}\otimes \O_M(r)\right)  &\cong   \bigoplus_{\mu=0}^r H^0(C,(K_C(D))^{\otimes \mu})  \\
    s_0y^r + s_1y^{r-1}w + \dots + s_rw^r&\longleftrightarrow(s_0,s_1,\dots ,s_r)
\end{align*}
where we identify $H^0(C,(K_C(D))^{\otimes \mu})$ with $\pi^*H^0(C,(C,K_C(D))^{\otimes \mu})\subset H^0( M, \pi^* (K_C(D))^{\otimes \mu})$ on the left-hand side. Note that $s_0\neq 0$ defines the divisors that do not touch or contain $C_\infty. $ Denote by $W\subset |rC_0|$ the subset of divisors that touch or contain $C_\infty$ as a component.
This describes an identification between the Hitchin base $A$ and divisors compactly supported in $M^\circ$:   
\begin{align}\label{divisors vs sections}
     |rC_0|\setminus W &\cong   \bigoplus_{\mu=1}^r H^0(C,(K_C(D))^{\otimes \mu}) =A \\
    Div(y^r + s_1y^{r-1} + \dots + s_r)&\longleftrightarrow(s_1,\dots ,s_r)
\end{align}
where $y$ is now treated as the tautological section of $\pi^*(K_C(D))$ on $M^\circ$. 

As $\bmxi\subset |rC_0|\setminus W$, we can provide another description in terms of $A$.  For each $p_i\in D$, we choose a local trivialization of $M$ around $p_i$ i.e. $\Spec(\C[[x_i]] [y])$ with horizontal coordinate $x_i$ and vertical coordinate $y$. Then the evaluation of the local derivative at $p_i$, $s_\mu^{(a)}(p_i):=\left. \frac{\del}{\del x_i^a}\right \vert_{0}s_\mu$, is well defined and we can define the evaluation maps
\begin{equation*}
    \mathrm{ev}_{u,a}(\xi_{i,j}): A\to \C, \quad s= (s_1,\dots, s_r) \mapsto \left. \frac{\partial}{\partial y^u}\frac{\partial}{\partial x_i^a} \right\vert_{(0, \xi_{i,j})}q_s(x_i,y)
\end{equation*}
where $q_s(x_i,y) = y^r+ s_1(x_i)y^{r-1}+\dots + s_r(x_i)$ and $u,a\in \Z_{\geq 0}$. 

Given a pair $(\vm, \vxi)$, let $\{P^{\xi^\circ_{i,1}},\dots,P^{\xi^\circ_{i,e(i)}}\}$ be the collection of partitions labeled by distinct eigenvalues $\uxi_i^\circ = (\xi_{i,1}^\circ,\dots \xi_{i,e(i)}^\circ)$ (see Section \ref{sec:notation}). To specify orders of partial derivatives, we introduce the following notation: Given a partition $P=(m_1\geq \dots \geq m_\ell)$, we define 
\begin{equation}\label{eq:leveldomain}
    G(P) := \{(u,a)\in \Z^2|0\leq u<|P|, 0\leq a< \gamma_P(|P|- u)\},
\end{equation} and call it  the \emph{level domain} associated to the partition $P$.

\begin{prop} \label{evaluation description}
 $ \bmxi $ is isomorphic to the subspace of \( A \) defined by the vanishing of evaluations over the level domains associated to the partitions \( P^{\xi^\circ_{i,j}} \), that is,
\[
\bmxi \cong 
\bigcap_{i \in I} \ \bigcap_{j = 1}^{e(i)} \ \bigcap_{(u,a) \in G\left(P^{\xi^\circ_{i,j}}\right)} 
\operatorname{ev}_{u,a}(\xi^\circ_{i,j})^{-1}(0)
\subset A.
\]
\end{prop}

\begin{proof}
    Let $\vxi^\circ = (\uxi^\circ_1,\dots, \uxi^\circ_n)$ be the distinct part of $\vxi$ where $\uxi^\circ_i = (\xi^\circ_{i,1},\dots, \xi^\circ_{i,e(i)})$. Apply the argument in Example~\ref{exmp:linearsystem3} and its direct generalization to the multiple points case $D=p_1+\dots +p_n$, we can identify $\bmxi$ as the subset of $|rC_0|\setminus W$ which consists of effective divisors satisfying the local multiplicity conditions through the centers of blow-ups associated to each $\xi^\circ_{i,j}$. It remains to check that the local multiplicity conditions over each point $p_i$ is equivalent to the vanishing of the evaluation maps under the isomorphism \eqref{divisors vs sections}.
    
    If $\xi^\circ_{i,j}=0$, then the local multiplicity condition at $(p_i,0)$ is equivalent to the vanishing of $\operatorname{ev}_{u,a}(0)$ by applying Lemma \ref{lem:localconditions}. For $\xi^\circ_{i,j}\neq 0$, by doing a translation $y= \overline{y}+\xi_{i,j}^\circ$, $q'_s(x,\overline{y}) := q_s(x,\overline{y}+\xi_{i,j})$ vanishes at $(x,\overline{y})=(0,0)$ and \begin{equation}\label{partialderivative}
           \operatorname{ev}_{u,a}(\xi^\circ _{i,j})=  \left. \frac{\partial}{\partial y^u}\frac{\partial}{\partial x^a} \right\vert_{(0, \xi^\circ_{i,j})}q_s(x,y) = \left. \frac{\partial}{\partial \overline{y}^u}\frac{\partial}{\partial x^a} \right\vert_{(0, \overline{y}=0)}q_s'(x,\overline{y}).
        \end{equation}
        Again, by applying Lemma \ref{lem:localconditions}, the local multiplicity condition at $(p_i,\xi^\circ_{i,j})$ is equivalent to the vanishing of $\operatorname{ev}_{u,a}(\xi^\circ_{i,j})$.

\end{proof}
As claimed, the affine space $\bmxi$ will later serve as the Hitchin base for the moduli space of $\vxi$-parabolic Higgs bundles later. For now, we verify that $B(\vm)_0$ matches the Hitchin base of the strongly parabolic Higgs bundles $A(\vec{m})_0$.
\begin{cor}\label{identification of hitchin bases}
    $B(\vm)_0$ is isomorphic to $A(\vm)_0.$ 
\end{cor}
\begin{proof}
    This follows directly from the equivalence between (2) and (3) in Lemma \ref{lem:localconditions}.
\end{proof}

                As discussed in Example \ref{exmp:linearsystem1}, the linear system $|\Sigmamxi|$ can be described by imposing multiplicity conditions at points on the successive blown-up surfaces. In general, requiring a divisor to pass through a point with multiplicity $\geq k$ will impose $k(k+1)/2$ linear conditions. Since $\vxi\in \nm$, there is a linear relation between the eigenvalues $\xi_{i,j}$. So, the expected dimension of the linear system $|\Sigmamxi|$ is 
    \begin{equation*}
    \begin{aligned}
        \mathrm{expdim}(|\Sigmamxi|):&= \mathrm{dim}( |rC_0|)  - 1+\sum_{i \in I}\sum_{j \in J_i}\frac{m_{i,j}(m_{i,j}+1)}{2}.
        \end{aligned}
\end{equation*}
We can compute $\dim(|rC_0|)$ simply by computing $\dim(A)$ since $|rC_0|\setminus W \cong A$. A direct computation via the Riemann--Roch theorem yields
        \begin{equation}\label{eq:expdim}
            \begin{aligned}
        \mathrm{expdim}(|\Sigmamxi|)&= \left(r^2(g-1)+ \frac{nr(r+1)}{2}\right) - 1+\sum_{i \in I}\sum_{j \in J_i}\frac{m_{i,j}(m_{i,j}+1)}{2}\\
        &=1+r^2(g-1)+\frac{nr^2-\sum_{i \in I}\sum_{j \in J_i} m_{i,j}^2}{2}.
    \end{aligned}
    \end{equation}

\begin{prop}\label{constantdim}
    Suppose that $\mathrm{expdim}(|\Sigmamxi|)\geq 0$ and the family of linear systems $|\bm{\sm}(\vec{m})|$ (resp. $\bm{B}(\vm)$) over $\nm$ is non-empty over every $\vec{\xi}\in \nm$. If the OK condition holds, i.e. 
    \begin{equation}\label{vanishing condition}
    H^1\left(C,\lmmu\right)=0 \quad \textrm{for }\mu=2,\dots, r,
    \end{equation}
     then the family $|\bm{\sm}({\vm})|$ (resp. $\bm{B}({\vm})$)  has constant dimension as \eqref{eq:expdim}.
    
\end{prop}
\begin{proof}
    First, note that there is a natural $\C^*$-action on $\nm$ by scaling $(\xi_{i,j})\mapsto (\lambda\xi_{i,j})$ where $\lambda\in \C^*$. There is also a $\C^*$-action on $\bm{Z}$ induced by the $\C^*$-action on $M$ via scaling on fibers over $C$ (fixing $C_\infty)$. Then it is easy to see that the family of surfaces $\bm{Z}\to \nm$ is $\C^*$-equivariant with respect to these actions. In particular, a non-zero $\lambda\in \C^*$ defines an isomorphism between $Z_{\vec{\xi}}$ and $Z_{\lambda \vec{\xi}}$ and induces an isomorphism between $|\Sigmamxi|$ and $|\Sigmam_{\lambda\vxi}|$. For each non-zero $\vec{\xi}\in \nm$, if we restrict the family of linear systems to the line through the origin and $\vec{\xi}$, then by upper semicontinuity, the dimensions of the family of linear systems $|\Sigmam_{\lambda\vxi}|$ can only jump at $\lambda=0$ i.e. $\dim(|\Sigmam_0|) \geq \dim(|\Sigmam_{\lambda \vxi}|)$ for $\lambda\in \C^*$ . By Proposition \ref{identification of hitchin bases}, we can compute the dimension of $|\Sigmam_0|$ by computing the dimension of $A(\vm)_0$ via a Riemann--Roch computation:
    \begin{equation}\label{riemann-roch}
        \dim(|\Sigmam_0|) = \dim(A(\vm)_0) = 1+ r^2(g-1) + \frac{nr(r+1)}{2} - \sum_{i =1}^n\sum_{\mu=1}^r \gamma_{P^i}(\mu). 
    \end{equation}
    Note that this formula relies on the vanishing assumption \eqref{vanishing condition},  so the only non-trivial $h^1$-term in the Riemann--Roch computation is  \[h^1\left(C, \left(K_C(D)\right)\otimes \O\left(-\sum_{i=1}^n\gamma_{P^i}(1) p_i\right)\right)=h^1(C,K_C) =1.\]
    In other words, the term for $\mu=1$  contributes to the $"+1"$ in \eqref{riemann-roch}.

 By the following obvious identity for a partition $P^i=(m_{i,1},\dots, m_{i,\ell(i)})$,  
    \[\sum_{\mu=1}^{r}\gamma_{P^i}(\mu) = \sum_{j=1}^{\ell(i)} \frac{m_{i,j}(m_{i,j}+1)}{2} \quad \textrm{for  } i=1,\dots, n,\]
    we see that $\dim(|\Sigmam_0|) = \textrm{expdim}(|\Sigmamxi|)$ as computed in \eqref{eq:expdim}. Since the expected dimension is always lower than the actual dimension, we deduce that 
    \[\dim(|\Sigmam_0|)\geq \dim(|\Sigmamxi|) \geq \textrm{expdim}(|\Sigmamxi|)=  \dim(|\Sigmam_0|).\]
    Hence, 
    $\dim(|\Sigmam_0|)=  \dim(|\Sigmam_{\lambda\vxi}|)$ for $\lambda\in \C^*$ and the dimension of the family of linear systems is constant everywhere. 

\end{proof}

\begin{rem}\label{affine bundle}
It follows from Proposition \ref{constantdim} that $(\bm{\pi}_{\m N}\circ \bm{f})_*\O(\bm{\Sigma}(\vm))$ is a vector bundle on $\nm$ whose fiber is $H^0(Z_{\vxi},\O(\Sigmamxi))$. We denote its projectivization by $\overline{\bm{B}(\vm)}$. By Lemma \ref{affinebase}, the open subset of $\overline{\bm{B}(\vec{m})}$ consisting of divisors compactly supported away from $F_{\vec{\xi}}$ and $\td{C}_{\infty,\vec{\xi}}$ for each $\vec{\xi}$ forms an affine bundle $\bm{B}(\vm)\to \nm$ whose fiber over each $\vec{\xi}$ is the affine space $\bmxi$. We will call 
\[\bm{B}(\vec{m})\to \nm\]
 the family of Hitchin bases associated to $\vm$. 

Recall that by definition we have the inclusion  $\iota: \O(\bm{\sm}(\vm))\to \O(r\bm{f}^*\bm{C_0}) $ induced by the section $\O\to \O(\sum_{i\in I}\sum_{j \in J_i}  m_{i,j} \bm{\Xi}_{i,j}).$ Since $(\bm{\pi}_{\m N}\circ \bm{f})_*\O(r\bm{f}^*\bm{C_0})\cong H^0(C,rC_0)\times\nm $, the projectivization of the inclusion $\iota$ induces the morphism $\overline{\bm{B}(\vec{m})}\to |rC_0|\times \nm$. Restricting to the open subsets of divisors away from the infinity divisors, we get a morphism $\bm{B}(\vec{m})\to \bm{A}(\nm):=A\times \nm$.
\end{rem}

The advantage of having a constant dimension for the family of Hitchin bases is that we can deform smooth or integral curves from the fiber at $0 \in \nm$ to nearby $\vec{\xi} \in \nm.$ This property will be useful later when we study the non-emptiness of the moduli spaces in Section \ref{sec:nonempty moduli}. 

\begin{prop}\label{deformation of curves}
    Suppose that the family of Hitchin bases $\bm{B}(\Vec{m})\to \nm$ is non-empty for every $\Vec{\xi}\in \nm$ and has constant dimension (e.g. by Proposition \ref{constantdim}). Then, whenever $B(\Vec{m})_0$ contains a smooth (resp. integral) member,   $\bmxi$ also contains a smooth (resp. integral) member for every $\Vec{\xi}\in \nm$.
\end{prop}
\begin{proof}
    As explained in Remark \ref{affine bundle}, the equi-dimensional family of Hitchin bases $\bm{B}(\Vec{m})\to \nm$ is an affine bundle, which is irreducible because $\nm$ is itself irreducible. Since $B(\vm)_0\subset \bm{B}(\vm)$ contains a smooth member and smoothness is an open condition, there exists a non-empty open subset $U\subset \bm{B}({\vec{m}})$ of smooth curves. As $\bm{B}(\vec{m})\to \nm$ is a surjective morphism between irreducible varieties, the image of $U$ in $\nm$ is also an open subset $V\subset \nm$ containing $0\in \nm$. Now, recall that, as explained in the proof of Proposition \ref{constantdim}, there is a $\C^*$-action on the family of surfaces $\bm{Z}\to \nm$ which identifies the curves in $\bmxi$ and $B(\vm)_{\lambda\vxi}$ for $\lambda \in \C^*$. In particular, the $\C^*$-action should preserves the smoothness of the curves. Hence, $V$ is $\C^*$-invariant and it follows that $V=\nm$. The same argument goes for the case of integral curves since the family of surfaces is smooth and so the family of curves (which are Cartier divisors) parametrized by $\bm{B}(\vec{m})$ are all Cohen-Macaulay. In this case, the locus of integral curves is open. 
\end{proof}

We will analyze the non-emptiness assumption in Proposition \ref{constantdim} and Proposition \ref{deformation of curves} in Section \ref{sec:nonemptiness}.

\section{Relative spectral correspondence}
\subsection{Relative moduli of pure dimension one sheaves}\label{sec:relativemoduli of pure}

We first recall the absolute case, particularly the  definition of $\beta$-twisted $A$-Gieseker semistablity condition first introduced in \cite{Matsuki-Wentworth} for torsion-free sheaves. For pure dimension one sheaves, it is due to Yoshioka \cite{yoshioka}. Let $X$ be a smooth projective surface and $\beta, A\in NS(X)_\Q$ with $A$ ample. Let $F$ be a pure dimension one sheaf on $X.$ We define the $\beta$-twisted Euler characteristic as
\[  \chi^\beta(F) = \chi(F) + \c_1(F)\cdot \beta \]
and the $\beta$-twisted slope as 
\[ \mu^\beta (F) = \frac{\chi^\beta(F)}{\c_1(F)\cdot A}. \]
\begin{definition}
    A pure dimension one sheaf $F$ is said to be $\beta$-twisted $A$-Gieseker stable (resp. semistable) if for any proper subsheaf $F'\subset F$, one has
    \[ \mu^\beta(F') < \mu^\beta(F), \quad (\textrm{resp.  }\leq) \]
\end{definition}

As we are considering the relative moduli problems, we will need to make a choice of a rational Neron-Severi class $\bm{\beta}$ and an ample line bundle $\bm{A}$ on $\bm {Z}$ both relative to $\bm{\pi}_{\m{N}}$. 
        \begin{itemize}
            \item (A choice of $\bm{\beta}$) For a set of rational numbers $\{\b_{i,j}|i\in I, j \in J_i\}$ with $0 \leq \b_{i,1} < \cdots < \b_{i,\ell(i)} $ for $i \in I$, we choose $\bm{\b}= \sum_{i\in I}\sum_{j \in J_i} \b_{i,j}\bm{\Xi}_{i,j}$ on $\mathrm{NS}_{\bm{\pi}}(Z)_\Q$. 
            \item (A choice of $\bm{A}$) One can prove that there exists a relatively ample divisor $\bm{A}$ on $\bm{Z}$ which can be written as 
            $\bm{A}=\kappa \bm{f}^*(\bm{D}) + \lambda \bm{C_\infty} -\sum_{i \in I, j \in J_i}\kappa_{i,j}\bm{\Xi}_{i,j}$ for sufficiently large enough $\kappa \gg\kappa_{i,j}$ and $\kappa\gg\lambda$. 
           
        \end{itemize}

With the fixed data: $r \geq 1$, $c \in \mathbb{Q}$, $d\in \mathbb{Z}$, $\bm{\beta}$, $\bm{A}$, and $\bm{\sm}({\vec{m}})$, we define the relative moduli stack $\bm{\overline{\f M}}$ over the base $\nm$ of $\bm{\beta}$-twisted $\bm{A}$-Gieseker semistable pure dimension one sheaves $\bm{\m F}$ on the family $\bm{Z}\to \nm$ with topological invariants $ (0,\bm{\sm}(\vm),c)$. In other words, for each $\vxi\in \nm$, if we let $\m{F}_{\vxi}$ denote the restriction of $\bm{\m{F}}$ to $Z_{\vxi}$, then $\m{F}_{\vxi}$ is $\beta_{\vxi}$-twisted  $A_{\vxi}$-Gieseker semistable and satisfies the following topological constraints:\[(\ch_0(\m{F}_{\vxi}),\ch_1(\m{F}_{\vxi}), \chi(\m{F}_{\vxi}))= (0, \Sigmamxi, c).\] 
Here $\beta_{\vxi}$ and $A_{\vxi}$ are the restriction of $\bm{\b}$ and $\bm{A}$ to $Z_{\vxi}$.

Recall that $\bm{S}\to \nm$ is the family of open holomorphic symplectic surfaces in $\bm{Z}\to \nm$. We define the open substack $\bm{\f M}(\vm)\subset \overline{\bm{\f M}(\vm)}$ of sheaves compactly supported on $\bm{S}\to \nm$. Equivalently, for each $\vxi\in \nm$, the restriction $\m{F}_{\vxi}$ is a pure dimension one sheaf on $S_{\vxi}\subset Z_{\vxi}$ with compact support. 

By \cite[Theorem 2.8]{yoshioka}, the relative moduli stack $\overline{\bm{\f{M}}(\vm)}\to \nm$ admits a relative coarse moduli space $\overline{\bm{\m{M}}(\vm)}\to \nm$, which is projective over $\nm$. There is a natural morphism defined by taking the Fitting support morphism:  
\begin{equation*}
    \begin{tikzcd}
    \overline{\bm{\m{M}}(\vm)}\arrow[rr,"\overline{\bm{\phi}}"]\arrow[rd]&& \overline{\bm{B}(\vm)}\arrow[ld]\\
        &\nm &
    \end{tikzcd}
\end{equation*}
As $\overline{\bm{\m{M}}(\vm)}$ is proper over $\nm$ and $\overline{\bm{B}(\vm)}$ is clearly separated over $\nm$, it follows that the Fitting support morphism $\overline{\bm{\phi}}$ is also proper. By definition, the preimage $\bm{\m M}(\vm):= \bm{\overline{\phi}}^{-1}(\bm{B}({\vec{m}}))$ becomes the relative coarse moduli space of $\bm{\f M}(\vm)$. Restricting the target of $\overline{\bm{\phi}}$ to the open subset $\bm{B}(\vec{m})\subset \overline{\bm{B}(\vec{m})}$ (see Remark \ref{affine bundle}), we have the following commutative diagram
\begin{equation*}
     \begin{tikzcd}
        \bm{\m{M}}(\vm) \arrow[rr,"\bm{\phi}"]\arrow[rd]&& \bm{B}({\vec{m}})\arrow[ld]\\
        &\nm &
    \end{tikzcd}
\end{equation*}
 
\subsection{Single blow-up}\label{s:sheaf local}
In this subsection, we first study the case of a blow-up at a single point in a more general setup to establish a relation between parabolic sheaves on a surface and its blow-up. We begin with a general definition of parabolic sheaf (cf. \cite[Definition 3.5]{AkerSzabo}). 

\begin{definition}[Parabolic sheaf]\label{defn:parabolic_sheaf}
Let $M$ be a smooth projective surface and let $L \subset M$ be an effective divisor. A parabolic sheaf $\mathcal{G}_\bullet$ on $M$ with the parabolic divisor $L$ consists of:
\begin{enumerate}
    \item A pure dimension one sheaf $G$ whose support does not share any irreducible components with $L$.
    \item A filtration of length $\ell$ by subsheaves:
    $$ G_\bullet : G_\ell \subset G_{\ell-1} \subset \dots \subset G_1 \subset G_0 = G $$
    such that $G_\ell = G \otimes \mathcal{O}_M(-L)$.
\end{enumerate}
There is an equivalent description of this filtration: a tower of quotient sheaves $R_j := G/G_j$ for $j=1, \dots, \ell$, which form a sequence of surjections:
$$ R_\bullet : G \twoheadrightarrow R_\ell \twoheadrightarrow R_{\ell-1} \twoheadrightarrow \dots \twoheadrightarrow R_1. $$
The condition on the filtration implies that $R_\ell \cong G|_L$.
\end{definition}
\begin{rem}
   In the literature, a parabolic weight is usually included in the definition of a parabolic sheaf. However, since we will not deal directly with the stability of parabolic sheaves, we omit this additional data.
\end{rem}
\begin{rem}\label{rem:support}
    A direct consequence of this definition is that all the quotient sheaves $R_j$ are $\O_L$-modules, that is, 
    \[R_j\xrightarrow{\sim} R_j|_{L}\quad \textrm{for}\quad j=1,\dots, \ell.\]
    Indeed, since $R_\ell\cong G|_L$ is an $\O_L$-module, the surjections $R_\ell\twoheadrightarrow R_j$ imply that all $R_j$ are $\O_L$-modules as well. We will refer to this property frequently. 
\end{rem}

Since we are mainly interested in the case where the parabolic sheaves come from $\vxi$-parabolic Higgs bundles, we also define the corresponding notion for parabolic sheaves.
\begin{definition}\label{def:xi,L par}
    Let $\uxi = (\xi_1,\dots, \xi_\ell)$ where $\xi_j$'s are not necessarily distinct points in $L$. A parabolic sheaf $\mathcal{G}_\bullet$ is said to satisfy the $(\uxi,L)$-parabolic condition if for $j=1,\dots, \ell$, the kernel $K_j:=\operatorname{ker}(R_j\to R_{j-1})\cong G_{j-1}/G_{j}$ (with $R_{-1}=0$) is a skyscraper sheaf supported at a single reduced point $\xi_j\in L$. 
\end{definition}

Let $M$ be a smooth projective surface and let $p:M_1 \to M$ be the blow-up map at a point $\xi \in M$, with exceptional divisor $E$. Let $L$ be a smooth curve in $M$ that contains $\xi$ and $L_1$ its strict transform in $M_1$. Our goal is to transform a parabolic sheaf $\mathcal{G}_\bullet$ on $M$ to a new one on $M_1$. Let us define a sheaf $F_1$ on $M_1$ as follows:
\[F_1:= \operatorname{ker}(p^*G\to p^*R_1) = \operatorname{Im}(p^*G_1\to p^*G).\] 
Consider the following diagram obtained by pulling back the tower of surjections of $G$ along the blow-up map $p$ and composing with the restriction to $L_1$: 
    \begin{equation*}
        \begin{tikzcd}
            p^*G\arrow[d,two heads]\arrow[rrd,two heads]\arrow[rrrd,two heads] &&&\\
            p^*R_\ell\arrow[d,two heads]\arrow[r,two heads] &\dots \arrow[r, two heads] &p^*R_2\arrow[d,two heads]\arrow[r,two heads] &p^*R_1\arrow[d,two heads]\\
            p^*R_\ell|_{L_1}\arrow[r,two heads] &\dots \arrow[r, two heads] &p^*R_2|_{L_1}\arrow[r,two heads] &p^*R_1|_{L_1}
        \end{tikzcd}
    \end{equation*} 
    Since $F_1$ maps to $p^*G$, this diagram induces a tower of surjections on $F_1$ 
    \begin{equation}
        \begin{tikzcd}
            F_1\arrow[d,two heads]\arrow[rrd,two heads]\arrow[rrrd,two heads] &&&\\
            I_\ell\arrow[r,two heads] &\dots \arrow[r, two heads] &I_2\arrow[r,two heads] &I_1= 0
        \end{tikzcd}
    \end{equation} 
    where $I_j = \operatorname{Im}(F_1 \to p^*R_j|_{L_1})$. One can then define a filtration $F_\bullet$ by taking their kernels $F_j=\operatorname{ker}(F_1\twoheadrightarrow I_{j})$. We will denote by $\Psi_p(\mathcal{G}_\bullet)$ this sheaf $F_1$ with its filtration $F_\bullet$ (or its tower of surjections $I_\bullet$). Since $I_1=0$, the length of the filtration $F_\bullet= F_\ell\subset ...\subset F_2\subset F_1$ on $F_1$ is $\ell-1$. 

\begin{prop}\label{prop:parabolic_properties}
Let $\mathcal{G}_\bullet$ be a parabolic sheaf of length $\ell$ on $M$ with the parabolic divisor $L$, and fix a point $\xi \in L$.
\begin{enumerate}
    \item $\Psi_p(\mathcal{G}_\bullet)$ is a parabolic sheaf of length $\ell-1$ on $M_1$ with parabolic divisor $L_1$. 
    
    \item Let $\uxi= (\xi_1,\dots,\xi_\ell)$ be a tuple of points in $L$ with $\xi_1=\xi$. If $\mathcal{G}_\bullet$ satisfies the $(\uxi,L)$-parabolic condition, then $\mathcal{F}_\bullet:=\Psi_p(\mathcal{G}_\bullet)$ satisfies the $(\uxi',L_1)$-parabolic condition. Here we denote by $\uxi'=(\xi'_2, \dots, \xi'_\ell)$ the tuple of points in $L_1$ where $\xi'_j=E\cap L_1$ if $\xi_j=\xi$ and $\xi'_j= p^{-1}(\xi_j)$ otherwise. 
\end{enumerate}
\end{prop}
\begin{proof}
    First, we prove (1). The purity of $F_1$ follows from the purity of $p^*G$ (Lemma \ref{l:local pullback}). Since $\operatorname{Supp}(F_1)\subset \operatorname{Supp}(p^*G)$ and the latter is the pullback of $\operatorname{Supp}(G)$, it is clear that both of them do not share any irreducible components with $L$. Note that the restriction of $p$ to $L_1$ is an isomorphism $p|_{L_1}: L_1\xrightarrow{\sim} L$. So, $p^*(R_{\ell})|_{L_1} \cong p^*(G|_L)|_{L_1} \cong (p^*G)|_{L_1}$. Then the map $F_1\twoheadrightarrow I_\ell$ sits in 
        \begin{equation*}
            \begin{tikzcd}
            p^*G_1\arrow[r,two heads]\arrow[d, two heads]&F_1\arrow[r,hook]\arrow[d,two heads]&p^*G\arrow[d,two heads]\\
            p^*G_1|_{L_1}\arrow[r,two heads]&I_\ell\arrow[r,hook]&p^*G|_{L_1}
            \end{tikzcd}
        \end{equation*}
        It follows that $I_\ell\cong F_1|_{L_1}$. 

         Next we prove (2). We need to show that the kernel $K'_j:=\operatorname{ker}(I_j\to I_{j-1})$ is supported at the point $\xi'_j$ for $j=2,\dots, \ell$. Since $\mathcal{G}_\bullet$ is a $(\uxi,L)$-parabolic sheaf, the sheaves $K_j, R_j , R_{j-1}$ are all $\O_L$-modules, so the sequence $0\to p^*(K_j)|_{L_1}\to p^*(R_j)|_{L_1}\to p^*(R_{j-1})|_{L_1}\to 0 $ is exact by the reason that $L_1\to L$ is an isomorphism. Then we have $K'_j\hookrightarrow I_j\twoheadrightarrow I_{j-1}$ sits in 
        \begin{equation*}
            \begin{tikzcd}
            p^*(G_{j-1}/G_j)|_{L_1}\arrow[r,dashed]\arrow[d]&K'_j\arrow[d,hook]\arrow[r,dashed]& p^*(K_{j})|_{L_1}= p^*(G_{j-1}/G_{j})|_{L_1}\arrow[d,hook]\\
            p^*(G_1/G_j)|_{L_1}\arrow[r,two heads]\arrow[d, two heads]&I_j\arrow[r,hook]\arrow[d,two heads]&p^*(R_j)|_{L_1} = p^*(G/G_j)|_{L_1}\arrow[d,two heads]\\
            p^*(G_1/G_{j-1})|_{L_1}\arrow[r,two heads]&I_{j-1}\arrow[r,hook]&p^*(R_{j-1})|_{L_1} = p^*(G/G_{j-1})|_{L_1}
            \end{tikzcd}
        \end{equation*}
        It follows that $K'_j = p^*(G_{j-1}/G_J)|_{L_1}$ is a skyscraper sheaf supported at $\xi'_j$ since $G_{j-1}/G_j\cong K_j$ is supported at $\xi_j\in L.$ 
        
\end{proof}        

In order to study how the numerical invariants of the underlying sheaf transform under $\Psi_p$, we recall the following fact which can be found in {\cite[Lemma 3.5 and Lemma 3.6]{CGKT2020}}. For the completeness, we include their proof here. 
\begin{lem}\label{l:local pullback}
    Let $G$ be a pure dimension one sheaf on $M$. Then $p^*G$ is pure with $\chi(p^*G)=\chi(G)$. Moreover, $p_*p^*G = G$. 
\end{lem}
\begin{proof}
    Since $G$ is pure of dimension one, we take a locally free resolution 
    \[0 \to E_2\to E_1\to G\to 0.\]
    Applying $p^*$ is left exact because the kernel of $p^*E_2\to p^*E_1$ must be supported on $E$, which cannot be a subsheaf of the locally free sheaf $p^*E_2$. Then the purity of $p^*G$ follows from the Auslander-Buchsbaum formula. 
    
    Taking the pushforward $p_*$, we get 
        \[ 0\to p_*p^*E_2\to p_*p^* E_1 \to p_*p^*G \to R^1p_*p^*E_2 \to \dots.  \]
    By the projection formula, since $E_1$ and $E_2$ are locally free and $R^jp_*\O_{M_1}=0$ for $j>1$, we deduce that the first two terms are just $E_2$ and $E_1$, and  $R^1p_*p^*E_i=0$ for $i=1,2$. Hence, it follows that $p_*p^*G=G$ and $R^1p_*p^*G=0$. By Leray theorem, $\chi(p^*G)=\chi(p_*p^*G)-\chi(R^1p_*p^*G)=\chi(G)-0=\chi(G)$.
\end{proof}
\begin{prop}\label{l:local pullback 2}
Let $\uxi= (\xi_1,\dots,\xi_\ell)$ be a tuple of points in $L$ with $\xi_1=\xi$.
Let $\mathcal{G}_\bullet$ be a $(\uxi,L)$-parabolic sheaf on $M$ and let $F_1$ be the underlying sheaf of $\Psi_p(\mathcal{G}_\bullet)$. Then $\c_1(F_1)= p^*\c_1(G)-mE$, where $m=\operatorname{rank}(R_1)$,  and $\chi(F_1\otimes \O(E))=\chi(G)$. 
\end{prop}
\begin{proof}
     Recall that $F_1=\ker(p^*G \to p^*R_1)$. Since $R_1$ is a skyscraper sheaf supported at $\xi_1=\xi$, $p^*R_1 = \O_E^{\oplus m}$. Hence, we have $\c_1(p^*R_1) = mE$ and  $\c_1(F_1) = p^*c_1(G) - mE$.
     
     Then we claim that $\chi(p^*R_1 \otimes \O(E)) = 0$. In fact, 
    \[
    \begin{aligned}
        & p_*(p^*R_1\otimes \O(E)) = H^0(E, p^*R_1 \otimes \O_E(E)) \cong H^0(\P^1, \O(-1)^{\oplus m})=0, \\
        & R^1p_*(p^*R_1\otimes \O(E)) = H^1(\P^1, \O(-1)^{\oplus m}) = 0.
    \end{aligned}
    \]
    Finally, it follows from Theorem \ref{thm:recovery} and its proof that 
    \[ \chi(F_1\otimes \O(E)) = \chi(F_0)=  \chi(p_*F_0) - \chi(R^1p_*F_0) = \chi(G). \]    
    
\end{proof}

\begin{theorem}\label{thm:recovery}
Let $\uxi= (\xi_1,\dots,\xi_\ell)$ be a tuple of points in $L$ with $\xi_1=\xi$.
Let $\mathcal{G}_\bullet$ be a $(\uxi,L)$-parabolic sheaf on $M$ with filtration $G_\ell \subset \dots \subset G_1 \subset G_0=G$. Let $F_\ell \subset \dots \subset F_2 \subset F_1$ be the filtration of the parabolic sheaf $\mathcal{F}_\bullet= \Psi_p(\mathcal{G}_\bullet)$. Consider the following chain of morphisms:
$$\mathcal{F}_{\bullet \to 0}:  F_\ell \hookrightarrow \dots \hookrightarrow F_2 \hookrightarrow F_1 \xrightarrow{\delta} F_0:= F_1 \otimes \mathcal{O}_{M_1}(E) $$
where $\delta$ is obtained from the natural map $\O_{M_1}\to \O_{M_1}(E)$.

Applying the pushforward functor $p_*$ to this chain recovers the original filtration term-by-term. That is, for each $i=0, \dots, \ell$, we have
$$ p_*(F_i) \cong G_i. $$
Furthermore, the pushforwards of the maps in the chain $\mathcal{F}_{\bullet \to 0}$ recover the inclusion maps of the original filtration.

\end{theorem}

\begin{rem}
 When the support of $F_1$ doesn't share a component with $E$, the map $\delta$ is injective so that the chain of morphisms becomes a filtration of $F_0$ and makes $F_0$ a parabolic sheaf. 
\end{rem}
\begin{proof}
    First, we show that the pushforward of $ F_0$ recovers $G_0$. By definition, we have 
    \[ 0\to F_0\to p^*G_0 \otimes \O(E) \to p^*R_1 \otimes \O(E)\to 0.\]
    Pushing it forward along $p$ yields    
\begin{equation}\label{e:Up les}
    \begin{aligned}
    0 & \to p_*F_0 \to p_*(p^*G_0 \otimes \O(E)) \to p_*(p^*R_1 \otimes \O(E))\\
    & \to R^1p_*F_0 \to  R^1p_*(p^*G_0 \otimes \O(E)) \to R^1p_*(p^*R_1 \otimes \O(E)) \to 0.
\end{aligned}
\end{equation}
We claim that 
\begin{equation}\label{e:pushforward G}
    p_*(p^*G_0\otimes \O(E))\cong G_0, \quad R^1p_*(p^*G_0\otimes \O(E))=0. 
\end{equation}

As done in the proof of Lemma \ref{l:local pullback}, we take a locally free resolution $0 \to E_2\to E_1\to G_0\to 0$ and pull it back to $M_1$. Tensoring further with the line bundle $\O(E)$ and taking the pushforward along $p_*$, we get 
        \[ 0\to p_*(p^*E_2 \otimes \O(E))\to p_*(p^* E_1\otimes \O(E)) \to p_*(p^*G_0\otimes \O(E))\to 0 \]
    since $R^1p_*(p^*E_2\otimes \O(E)) =E_2 \otimes R^1p_*\O_{M_1}(E)=0$. Hence, the relation \eqref{e:pushforward G} follows. 
        
        Next, we show that $p_*(p^*R_1 \otimes \O(E))=0$. Since $p_*(p^*R_1 \otimes \O(E))$ is a coherent sheaf with $0$-dimensional support, it is enough to show that its length is $0$. By Leray theorem, we have 
        \[\chi(p^*R_1 \otimes \O(E))=\chi(p_*(p^*R_1 \otimes \O(E)))-\chi(R^1p_*(p^*R_1 \otimes \O(E))).\]
        We have already shown that $\chi(p^*R_1 \otimes \O(E))=0$ and $R^1p_*(p^*R_1 \otimes \O(E))=0$ which follows from \eqref{e:Up les} and \eqref{e:pushforward G}. Thus, $p_*(p^*R_1 \otimes \O(E))=0$. From the long exact sequence \eqref{e:Up les}, one can see that the relation \eqref{e:pushforward G} implies that $p_*F_0 \cong G_0$. 

        In order to recover the inclusion $G_1\hookrightarrow G_0$ from $\delta: F_1\to F_0$, we consider the morphism of two short exact sequences 
        \[
        \begin{tikzcd}
            0 \arrow[r] & F_1 \arrow[r] \arrow[d,"\delta"] & p^*G_0 \arrow[r] \arrow[d] & p^*R_1 \arrow[r] \arrow[d] & 0 \\
            0 \arrow[r] & F_0 \arrow[r] & p^*G_0 \otimes \O(E) \arrow[r] & p^*R_1 \otimes \O(E) \arrow[r] & 0.
        \end{tikzcd}
        \]
        Since we have shown $p_*F_0  \cong p_*(p^*G_0 \otimes \O(E)) \cong  G_0$, $p_*p^*G_0\cong G_0$ (Lemma \ref{l:local pullback}), and $p_*p^*R_1 \cong R_1$ due to the $(\uxi,L)$-parabolic condition, pushing forward the above diagram yields
        \[\begin{tikzcd}
            0 \arrow[r] & p_*F_1 \arrow[r] \arrow[d,"p_*\delta"] & G_0 \arrow[r] \arrow[d,"\sim"] & R_1 \arrow[r] \arrow[d] & 0 \\
            0 \arrow[r] & p_*F_0 \arrow[r,"\sim"] & G_0\arrow[r]& 0 \arrow[r] & 0.
        \end{tikzcd}
        \]
        Hence, $p_*F_1\cong G_1$ and the inclusion $G_1\hookrightarrow G_0$ is recovered. 
        
        For the other inclusion of $G_j$ where $j=2,\dots, \ell$, it suffices to show that the pushforward of the map $F_1\twoheadrightarrow I_j$ recovers $G_1\twoheadrightarrow G_1/G_j$ since that will imply that the pushforward of $F_j\to F_1$ is $G_j\hookrightarrow G_1$. By definition, the map $F_1\twoheadrightarrow I_j$ sits in 
        \begin{equation}\label{eq:filtration recovery}
            \begin{tikzcd}
            p^*G_1\arrow[r,two heads]\arrow[d, two heads]&F_1\arrow[r,hook]\arrow[d,two heads]&p^*G_0\arrow[d,two heads]\\
            p^*(G_1/G_j)|_{L_1}\arrow[r,two heads]&I_j\arrow[r,hook]&p^*(G_0/G_j)|_{L_1}.
            \end{tikzcd}
        \end{equation}
        Note that if an $\O_M$-module $S$ on $M$ is actually an $\O_L$-module, then we have $p_*(p^*S|_{L_1})\cong S$ since the restriction of $p$ to $L_1$ is an isomorphism $L_1\xrightarrow{\sim}L$. In particular, if we apply this to $S= G_1/G_j$ and $G_0/G_j$ (which are $\O_L$-modules by Remark \ref{rem:support}) and push the diagram \eqref{eq:filtration recovery} forward along $p$, we get 
        \begin{equation*}
            \begin{tikzcd}
            p_*p^*G_1\arrow[r]\arrow[d]&G_1\arrow[r,hook]\arrow[d]&G_0\arrow[d,two heads]\\
            G_1/G_j\arrow[r]&p_*I_j\arrow[r,hook]&G_0/G_j
            \end{tikzcd}
        \end{equation*}
        where the identification in the first row follows from the first part of the proof. It follows immediately that $p_*I_j\cong G_1/G_j$ and the claim follows. 
\end{proof}

Theorem \ref{thm:recovery} says that the original parabolic sheaf $\mathcal{G}_\bullet$ can be recovered from $\Psi_p(\mathcal{G}_\bullet)$ via the pushforward functor $p_*$ when applied to the corresponding chain of maps $\mathcal{F}_{\bullet\to 0}$ on $M_1$. For the rest of this section, we establish the other direction starting from a parabolic sheaf $\mathcal{F}_\bullet$ on $M_1$ under the additional assumption that the support of $F_1$ is integral, but not equal to the exceptional divisor.

\begin{prop}
\label{prop:parabolic_converse}
Let $\mathcal{F}_\bullet$ be a parabolic sheaf on $M_1$ of length $\ell-1$ with parabolic divisor $L_1$ and filtration $F_\ell\subset ...\subset F_2 \subset F_1$ on the underlying sheaf $F_1$. Assume the support of $F_1$ is integral, but not equal to the exceptional divisor $E$.  Consider the following chain of morphisms 
$$\mathcal{F}_{\bullet \to 0}:  F_\ell \hookrightarrow \dots \hookrightarrow F_2 \hookrightarrow F_1 \xrightarrow{\delta} F_0:= F_1 \otimes \mathcal{O}_{M_1}(E). $$
\begin{enumerate}
    \item Applying the pushforward functor $p_*$ to this chain defines a parabolic sheaf $\mathcal{G}_\bullet$ of length $\ell$ on $M$ with parabolic divisor $L$. 
    \item Let $\uxi'=(\xi'_2, \dots, \xi'_\ell)$ be a tuple of points in $L_1$ and let $\uxi = (\xi_1, \dots, \xi_\ell)$ be the corresponding tuple on $L$, where $\xi_1=\xi$ and $p(\xi'_j)=\xi_j$ for $j \ge 2$. If $\mathcal{F}_\bullet$ satisfies the $(\uxi',L_1)$-parabolic condition, then $\mathcal{G}_\bullet$ satisfies the $(\uxi,L)$-parabolic condition.
\end{enumerate}
\end{prop}

\begin{proof}
    By the assumption on the support of $F_1$, $\delta: F_1\to F_0$ is injective. Pushing forward $\mathcal{F}_{\bullet\to 0}$ term-by-term, we have the chain of inclusions (filtration on $G_0$)
    \[  G_\ell \hookrightarrow \dots \hookrightarrow G_2 \hookrightarrow G_1 \hookrightarrow G_0 , \quad \textrm{where} \quad G_j = p_*F_j. \]
    Moreover, by projection formula, we have 
    \[G_\ell = p_*(F_\ell) = p_*(F_1\otimes \O(-L_1))=p_*(F_0\otimes \O(-E-L_1)) = (p_*F_0)\otimes  \O (-L) = G_0\otimes \O(-L)\]
    which means that $G_0$ with the filtration indeed defines a parabolic sheaf $\mathcal{G}_\bullet$.

    For (2), note that $R^1p_* F_j=0$ implies that $G_{j-1}/G_j\cong p_*(F_{j-1}/F_j)$ which is a skyscraper sheaf supported on $\xi_j$ for $j=2,\dots, \ell$ by the $(\uxi',L_1)$-parabolic condition. Finally, $G_0/G_1= p_*(F_0/F_1)\cong p_*(F_0|_E)$ is clearly a skyscraper sheaf supported on $\xi_1=\xi$. 

\end{proof}

\begin{theorem}
\label{thm:construction_is_inverse}
Let the hypotheses and notations be as in Proposition \ref{prop:parabolic_converse}. Then $\mathcal{F}_\bullet$ can be recovered from $\mathcal{G}_\bullet$ via $\Psi_p$: 
$$ \Psi_p(\mathcal{G}_\bullet) = \mathcal{F}_\bullet. $$
\end{theorem}

\begin{proof}
    Denote by $H_1$ the underlying sheaf of $\Psi_p(\mathcal{G}_\bullet)$ and $H_\ell\subset \dots \subset H_1$ the filtration. We will first show that $H_1 \cong F_1$. We begin by checking that $H_1$ and $F_1$ share the same numerical invariant. Note that $R^1p_*F_1=0$ implies that $0\to p_*F_1\to p_*F_0\to p_*F_0|_E\to 0$ remains exact, so $G_0/G_1\cong  p_*F_0|_E$. For the first Chern class, we have 
    \[\c_1(H_1)= p^*\c_1(p_*F_0) - \chi(F_0|_E) [E] = \c_1(F_0)= \c_1(F_1).\]
    On the other, since $p_*F_0$ is also a pure dimension one sheaf, we can apply Lemma \ref{l:local pullback} to $p_*F_0$ to show that 
    \[\chi(p^*p_*F_0) =  \chi(p_*F_0) = \chi(F_0)\]
    where the final equality follows because $R^1p_*F_0=0$. Moreover, since $p^*p_*F_0|_E\cong \O_E^{\oplus \chi(F_0|_E)}$, we have $\chi(p^*p_*(F_0|_E)) = \chi(F_0|_E)$. Hence, we deduce that $\chi(H_1) = \chi(F_0) - \chi(F_0|_E)= \chi(F_1)$. 
    
    Consider the following diagram
    \begin{equation}\label{eq:recovery_isom}
        \begin{tikzcd}
            0\arrow[r]&H_1\arrow[r]\arrow[d,dashed]&p^*p_*(F_0)\arrow[d]\arrow[r]&p^*p_*(F_0|_E)\arrow[r]\arrow[d]&0 \\
            0\arrow[r]&F_1\arrow[r]&F_0\arrow[r]&F_0|_E\arrow[r]&0 
        \end{tikzcd}
    \end{equation}
    where the second and third vertical maps are the natural maps.
    Since all the morphisms of the left-hand square in the diagram \eqref{eq:recovery_isom} are isomorphic away from $E$, the induced map $H_1\to F_0$ is non-zero. As the supports of $H_1$ and $F_1$ are integral (hence Gieseker stable in the usual sense), the non-zero morphism $H_1\to F_1$ between stable sheaves with the same numerical invariant must be an isomorphism. 

    Next, we shall show that $H_j\subset H_1$ coincides with $F_j\subset F_1$, which is equivalent to show that their quotients are equivalent. By definition, if we write $H_1/H_j$ as $I_j$, the $H_1\twoheadrightarrow I_j$ sits in 
    \begin{equation}\label{eq:filtration recovery}
            \begin{tikzcd}
            p^*G_1\arrow[r,two heads]\arrow[d, two heads]&H_1\arrow[r,hook]\arrow[d,two heads]&p^*G_0\arrow[d,two heads]\\
            p^*(G_1/G_j)|_{L_1}\arrow[r,two heads]&I_j\arrow[r,hook]&p^*(G_0/G_j)|_{L_1}.
            \end{tikzcd}
        \end{equation}
    By $R^1p_*F_j=0$, we have
    \[p^*(G_1/G_j)|_{L_1}\cong p^*\left(p_*(F_1/F_j) \right)|_{L_1}\cong F_1/F_j\]
    where the second isomorphism follows since $F_1/F_j$ is an $\O_{L_1}$-module and $p|_{L_1}:L_1\xrightarrow{\sim}L$ is an isomorphism. Since $\mathcal{G}_\bullet$ is again a parabolic sheaf by Proposition \ref{prop:parabolic_converse}, $G_1/G_j\hookrightarrow G_0/G_j$ is a morphism of $\O_{L}$-modules. Hence, $p^*(G_1/G_j)|_{L_1}\to p^*(G_0/G_j)|_{L_1}$ is still injective, since $p|_{L_1}$ is an isomorphism. In particular, it follows that $I_j\cong F_1/F_j$. 

\end{proof}

\subsection{Relative spectral correspondence}

We establish the relative spectral correspondence over $\nm$.
\begin{theorem}[Relative spectral correspondence] \label{spectral-corr}
        Fix $r\geq 1, d\in \Z$. Let $\bm{\m{M}}(\vm)$ be as in Section \ref{sec:relativemoduli of pure} and $\bm{\m{H}}(\vm)$ be the relative coarse moduli space of semistable $\vec{\xi}$-parabolic Higgs bundles on $C$ of rank $r$, degree $d$ with parabolic weights $\vec{\a}=(\underline{\a}_1,\cdots,\underline{\a}_n)$. There is a closed embedding 
    \begin{equation}
        \begin{tikzcd}
            \bm{Q}: \bm{\m{H}}(\vm)\arrow[rr, hook] 
            \arrow[rd, swap] && \bm{\m{M}}(\vm)  \arrow[ld] \\
            & \nm &
        \end{tikzcd}
    \end{equation}
    where $\a_{i,j}=\b_{i,j}+N, d=c+r(g-1)$ for sufficiently large $\kappa$ and $N$.
    \end{theorem}

    \begin{rem}
    Our construction of $\bm{Q}$ extends the construction of \cite{Diaconescu_2018} in the case of generic $\vxi\in \nm $. In particular, $\bm{Q}$ is an isomorphism over generic $\vxi\in \nm$. 
    \end{rem}

\subsubsection*{Construction of the morphism $\bm{Q}: \bm{\m{H}} \to \bm{\m{M}}$ over $\nm$} For each of notations, we will fix $\vxi\in \nm$ and demonstrate the construction at the level of objects, which extends for families of objects and relatively over $\nm$ in a natural way. If it is clear from the context, we may drop the subscript $\vxi$ for simplicity. Let $(E, E^\bullet_D, \Phi, \va)\in \hvxi$ be a $\vxi$-parabolic Higgs bundle. If 

\noindent\textbf{Step 1}:
First, we apply the classical spectral correspondence \cite{bnr} to a get pure dimension one sheaf on $M$. Let $y$ be a tautological section of $\pi^*M^\circ$ over $M^\circ$. Let $G$ to be the cokernel of 
\[ y\operatorname{Id}_{\pi^* E}- \pi^*\Phi: \pi^*E \otimes _{M^\circ} \pi^*(M^\circ)^{-1} \to \pi^* E\]
which is a pure dimension one sheaf on $M$ whose support is away from $C_\infty$ and does not share a common component with $\pi^{-1}(D)$. Moreover, we have $\c_1(G) = rC_0$ and $\chi(G)= \chi(E) = d+r(1-g)$. 
Recall that the pushforward of $G$ to $C$ together with the tautological section $y$ recovers $(E,\Phi)$. Also, the quasi-parabolic structure on $(E,\Phi^\bullet)$ makes $G$ a parabolic sheaf. More explicitly, for each $i\in I$ and $j\in J_i$, consider $E\twoheadrightarrow E_{p_i}\twoheadrightarrow E_{p_i}/E^j_{p_i} $ and the induced diagram
\begin{equation*}
    \begin{tikzcd}
        \pi^*E \otimes _{M^\circ} \pi^*(M^\circ)^{-1} \arrow[d,two heads]\arrow[rrr," y\operatorname{Id}_{\pi^* E}- \pi^*\Phi"]&&& \pi^* E\arrow[d,two heads]\arrow[r]&G\arrow[d,two heads] \\
        \pi^*(E_{p_i}/E_{p_i}^j) \otimes _{M^\circ} \pi^*(M^\circ)^{-1}\arrow[rrr," (y-\pi^* \xi_{i,j})\operatorname{Id}_{\pi^* (E_{p_i}/E^j_{p_i})}"] &&& \pi^* E_{p_i}/E_{p_i}^j\arrow[r]&R_{i,j} 
    \end{tikzcd}
\end{equation*}
where $R_{i,j}$ is defined as the cokernel of $(y-\pi^* \xi_{i,j})\operatorname{Id}_{\pi^* (E_{p_i}/E^j_{p_i})}$. 
It is easy to check that the surjections $G\twoheadrightarrow R_{n,\ell(n)}\twoheadrightarrow\dots \twoheadrightarrow R_{1,1}$ defines a parabolic sheaf $\mathcal{G}_\bullet$ with parabolic divisor $L= \pi^*(D)$. Moreover, the $\vxi$-parabolic condition of $(E,E^\bullet_D,\Phi,\va)$ implies that $\mathcal{G}_\bullet$ satisfies the $(\vxi,L)$-parabolic condition (see Definition \ref{def:xi,L par})

\noindent\textbf{Step 2}:
Recall that we have constructed $Z_{\vxi}$ as iterated blow-ups on $M$, with intermediate spaces $M_{i,j}$, for $i\in I=\{1, \dots, n\} , j \in J_i=\{1, \dots, \ell(i)\}$, and $Z_{\vxi}= M_{n,\ell(n)}$ is the resulting blown-up surface. We iteratively apply the construction of $\Psi_{p}$ in Section \ref{s:sheaf local} to the parabolic sheaves, which reduces the length of the filtrations at each step. Starting from $i=1$, we have the first blow-up $p_{1,1}: M_{1,1}\to M$ over $\xi_{1,1}$ and the $(\vxi,L)$-parabolic sheaf $\mathcal{G}_\bullet$ with the parabolic divisor $L$ on $M$. We can construct the following parabolic sheaf 
\[\mathcal{F}^{1,1}_\bullet:= \Psi_{p_{1,1}}(\mathcal{G}_\bullet) \]
which satisfies the $(\vxi^{(1,1)}, L_{1,1})$-parabolic condition due to Proposition \ref{prop:parabolic_properties}. Here  $L_{1,1}$ is the strict transform of $L$ and $\vxi^{(1,1)}= (\vxi^{(1,1)}_{i,j})_{i\in I, j\in J_i, (i,j) > (1,1)}$ with respect to the lexicogprahic order of paris $(i,j)$ and $\xi^{(1,1)}_{i,j}=E_{1,1}\cap L_{1,1}$ if $\xi_{i,j}=\xi$ and $\xi^{(1,1)}_{i,j}= p^{-1}(\xi_{i,j})$ otherwise. Moreover, by Lemma \ref{l:local pullback 2}, the underlying sheaf $F^{1,1}$ of $\mathcal{F}^{(1,1)}_\bullet$ has 
\[\c_1(F^{1,1}) = p_{1,1}^*(rC_0) - m_{1,1}E_{1,1} \quad \textrm{and}\quad  \chi(F^{1,1}\otimes \O(E_{1,1})) =\chi(G)  \]

Now, we can iteratively apply the construction $\Psi_{p^{i,j-1}_{i,j}}$ for each blow-up $p^{i,j-1}_{i,j}$ to get
 \[\begin{cases}
       \mathcal{F}^{i,j}_\bullet = \Psi_{p_{i,j}^{i,j-1}}(\mathcal{F}^{i,j-1}_\bullet), &\textrm{if } j\neq 1\\
       \mathcal{F}^{i,j}_\bullet = \Psi_{p_{i,j}^{i-1,\ell(i-1)}}(\mathcal{F}^{i-1,\ell(i-1)}_\bullet), &\textrm{if  } j=1
    \end{cases}\]
which again satisfies the $(\vxi^{(i,j)},L_{i,j})$-parabolic condition by Proposition \ref{prop:parabolic_properties}. In the last step, we will obtain a parabolic sheaf $\mathcal{F}_{n,\ell(n)}$ whose filtration has length $0$, that is, a pure dimension one sheaf $F_{n,\ell(n)}$. Finally, we define 
\[ Q_{\vxi}(E, E^\bullet_D, \Phi, \va) = F_{n,\ell(n)}\otimes \O(\Xi_{1,1}+ \dots + \Xi_{n,\ell(n)})\]
where $Q_{\vxi}$ is the restriction of the functor $\bm{Q}$ over $\vxi \in \nm$.

\begin{prop}\label{prop:sp numerical}
Let $(E, E^\bullet_D, \Phi, \va)\in \hvxi$ be a parabolic Higgs bundle of rank $r$ and degree $d$, and let $F = Q_{\vxi}(E, E^\bullet_D, \Phi, \va)$ be the corresponding sheaf on $Z_{\vxi}$. Then the first Chern class of $Q_{\vxi}(E, E^\bullet_D, \Phi, \va) $ is given by
\[\sm(\vm)_{\vxi}= rp^*C_0 -\sum_{i\in I}\sum_{j\in J_i}m_{i,j}\Xi_{i,j},\quad \]
and its Euler characteristic given by $\chi(G) = \chi(E) = d+r(1-g)$.
\end{prop}
\begin{proof}
    It follows immediately by iteratively applying Lemma \ref{l:local pullback 2}.
\end{proof}
\begin{prop}\label{prop:sp stab}
    Let $\beta_{i,j}=\alpha_{i,j}+ N$ for some $N >0$, and define $\beta_{\vxi}$ and $A_{\vxi}$ as in Section \ref{sec:relativemoduli of pure}. There exist sufficiently large $\kappa$ and $N$ such that if the $\vxi$-parabolic Higgs bundle $(E, E^\bullet_D, \Phi, \va)$ is $\a$-semistable, then the induced pure dimension one sheaf $F = Q_{\vxi}(E, E^\bullet_D, \Phi, \va)$ is $\b_{\vxi}$-twisted $A_{\vxi}$-semistable. 
\end{prop}
\begin{proof}
    Let $F'\subset F$ be a subsheaf. First, suppose that the support of $F'$ does not contain any exceptional divisors such that $R^1f_*F'=0$. Then $E' := f_*F'$ is a rank $r'$ subbundle of $E$ with $r'\leq r$. Then the induced parabolic structure on $E'_D= E'\cap E^\bullet_D$ has length $m_{i,j}'\leq m_{i,j}$. Since $\va$-stability is unchanged after replacing $\va=(\alpha_{i,j})$ by $\alpha_{i,j}+N$ for $N \geq 0$, the assumption that $(E,E^\bullet_D,\Phi,\va)$ is $\va$-semistable implies that 
    \begin{equation}\label{eq:stability}
     \frac{\chi^{\beta}(F')}{r'}=\frac{\chi(E')+ \sum{\b_{i,j}m'_{i,j}}}{r'}\leq \frac{\chi(E)+ \sum {\b_{i,j}m_{i,j}}}{r} = \frac{\chi^{\beta}(F)}{r} .
    \end{equation}
    In order to relate this to the $\beta_{\vxi}$-twisted stability for $F$, recall the general fact that for a fixed numerical class $\Sigmamxi$ , the set of possible first Chern classes $\c_1(F')$ of subsheaves of $F$ with $\c_1(F)=\Sigmamxi$ is finite (see e.g. \cite[Section 1.4]{Yoshioka2000abelian}). Note that $\c_1(F')\cdot A_{\vxi} = \kappa r'-b'$ and $\c_1(F)\cdot A_{\vxi} = \kappa r-b$ for some integer $b',b\geq 0$. By choosing $\kappa$ sufficiently large, we can assume that the value
    \[\nu_{\kappa} =\left( \frac{\c_1(F')\cdot A_{\vxi}}{ \c_1(F)\cdot A_{\vxi} }\right) \chi^{\beta}(F)-\c_1(F')\cdot \beta_{\vxi} = \left(\frac{\kappa r' - b'}{\kappa r- b} \right) \chi^{\beta}(F) -\c_1(F')\cdot \beta_{\vxi} \]is sufficiently close to 
    \[ \nu= \left(\frac{r' }{ r} \right) \chi^{\beta}(F) -\c_1(F')\cdot \beta_{\vxi},\] 
    for all $\c_1(F')$ and $r'$, such that an integer $K$ satisfies $K\leq \nu$ if and only if $K\leq \nu_\kappa$. Applying this to $K=\chi(F')$, we see that the inequality \eqref{eq:stability} implies $\mu^{\beta}(F')\leq \mu^{\beta}(F)$. 

    Suppose $F'$ is supported entirely on the union of some exceptional divisors. Then $f_*F'$ is a $0$-dimensional subsheaf of $f_*F=E$. Since $E$ is locally free, it cannot have torsion subsheaf, so $f_*F'=0$. Note that $\chi(F') = - \chi(R^1f_*(F'))\leq 0$ because $R^1f_*(F')$ has $0$-dimensional support. Thus, we get 
    \[ \mu^{\beta}(F')= \frac{\chi(F') +\c_1(F')\cdot \beta_{\vxi} }{\c_1(F')\cdot A_{\vxi}} \leq \frac{\c_1(F')\cdot \beta_{\vxi} }{\c_1(F')\cdot A_{\vxi}} = -\frac{\sum c_{i,j}\beta_{i,j} }{\sum c_{i,j}\kappa_{i,j}}   \]
    where we let $\c_1(F') = \sum_{i\in I, j\in J_i}c_{i,j}\Xi_{i,j}$ for $c_{i,j}\geq 0$. It is clear that the expression on the right hand side must be negative. On the other hand, the slope for $F$ looks like
    \[\mu^{\beta}(F)= \frac{\chi^\beta(F) }{\c_1(F)\cdot A_{\vxi}}  = \frac{\chi(F) +\sum \beta_{i,j}m_{i,j} }{\kappa r -\sum_{i\in I, j\in J_i} \kappa_{i,j}m_{i,j} }  \] 
    where the expression on the right is positive for sufficiently large $N$. Hence, we again have $\mu^{\beta}(F')\leq \mu^\beta(F)$ in this case.

    Now, if the support of $F'$ is neither of the cases above, we write $\operatorname{Supp}(F') =  \Gamma_v+ \Gamma_h$ where $\Gamma_v$ consists entirely of exceptional divisors and $\Gamma_h$ is a finite cover over $C$ of order $r' \leq r$. Then we have the two short exact sequences
    \[  0\to F_h\to F' \to F_v\to 0, \quad 0\to F_v\to F'(\Gamma_h) \to F'(\Gamma_h)|_{\Gamma_h}\to 0   \]
    where $\c_1(F_h)= \Gamma_h$ and $\c_1(F_v)= \Gamma_v$. 
    Note that $F_{h}\subset F'$ is a subsheaf of $F$ of the first type above, we must have 
    \begin{equation}\label{eq:stability0}
        \mu^\beta(F_h)\leq \mu^{\beta}(F)
    \end{equation}
    Similarly, since $F_v(-\Gamma_h)\subset F'$ is a subsheaf of $F$ of the second type, we have 
    \begin{equation}\label{eq:stability1}
         \mu^{\beta}(F_v(-\Gamma_h)) =\frac{\chi^{\beta}(F_v)} { \c_1(F_v)\cdot A_{\vxi} }- \frac{\Gamma_v\cdot \Gamma_h}{\Gamma_v\cdot A_{\vxi}} \leq -\frac{\sum d_{i,j}\beta_{i,j} }{\sum d_{i,j}\kappa_{i,j}}  
    \end{equation}
    where we write $\Gamma_v = \sum_{i\in I, j\in J_i}d_{i,j}\Xi_{i,j}$ for $d_{i,j}\geq 0$. By the same argument for the finiteness of the choice of $\c_1(F')$ as above, there are only finitely many choices of decomposition $\c_1(F')= \Gamma_v + \Gamma_h$. Let $\omega$ be the minimum of all (finitely many) $-\frac{\Gamma_v\cdot \Gamma_h}{\Gamma_v\cdot A_{\vxi}}$. Choose $N$ sufficiently large such that the last expression in \eqref{eq:stability1} is smaller than $\omega$. Then the inequality \eqref{eq:stability1} implies the following inequalities 
    \begin{equation}\label{eq:stability2}
        \mu^{\beta}(F_v)  \leq \omega + \frac{\Gamma_v\cdot \Gamma_h}{ \c_1(F_v)\cdot A_{\vxi} }\leq 0 \leq \mu^\beta(F) .
    \end{equation}
    Combining the inequalities \eqref{eq:stability0} and \eqref{eq:stability2}, we get
    \begin{align*}
        \chi^{\beta}(F')= \chi^\beta(F_h)+ \chi^{\beta}(F_v) \leq \mu^\beta(F)(\c_1(F_h)\cdot A_{\vxi}+\c_1(F_v)\cdot A_{\vxi})  = \mu^\beta(F)(\c_1(F')\cdot A_{\vxi}) .
    \end{align*}
    Since $\c_1(F')\cdot A_{\vxi}$ is positive, it yields
    \[\mu^{\beta}(F') \leq \mu^\beta(F).\]
    
\end{proof}

\begin{rem}
   As in the proof of Proposition \ref{prop:sp stab}, for fixed $\vxi \in \nm$ and $\Sigma(\vm)_{\vxi}$, there are only finitely many $\c_1(F')$. Hence, it is possible to choose sufficiently large $\kappa$ and $N$. Moreover, up to nonzero scaling, the same $\kappa$ and $N$ can be taken for $\lambda \vxi$ with any $0 \neq \lambda \in \mathbb{C}$. In other words, the choice of $\kappa$ and $N$ depends only on the configuration of exceptional divisors up to nonzero scaling. Since there are only finitely many such configurations, one may choose $\kappa$ and $N$ uniformly.
\end{rem}

\begin{cor}[Fitting support morphism vs Hitchin morphism]\label{cor:sp commutativity}
The following diagram commutes: 
\begin{equation*}
    \begin{tikzcd}
        \hpar(\vm)\times \nm\arrow[d,"h^{\parb}(\vec{m})"] & \bm{\m{H}}(\vm) \arrow[l] \arrow[r, hook, "\bm{Q}"] & \bm{\m{M}}(\vm) \arrow[d]\\
        \bm{A}(\vm) && \bm{B}(\vm) \arrow[ll, hook', "\iota"].
    \end{tikzcd}
\end{equation*}
where $\iota$ is defined in Remark \ref{affine bundle}.
\end{cor}
\begin{proof}
  The commutativity follows from the definition of $\bm{Q}$.   
\end{proof}

\begin{prop}\label{prop:sp proper}
The morphism $\bm{Q}:\bm{\m{H}}(\vm) \to \bm{\m{M}}(\vm)$ is proper. 
\end{prop}
\begin{proof}
    First, note that the morphism $\bm{\m{H}}(\vm)\to \hpar(\vm)\times \nm$ is a closed embedding because the $\vxi$-parabolic condition is a closed condition. Moreover, by a result of Yokogawa \cite[Corollary 5.12]{Yokogawa-compactification}, the morphism $h^{\parb}(\vec{m})$ is proper. Thus, the composition $\bm{\m{H}}(\vm)\to \bm{A}(\vm)$ is proper. Similarly, since the Fitting support morphism $\bm{\mathcal{M}}(\vm)\to \bm{B}(\vm)$ is proper and $\iota$ is a closed embedding, their composition $\bm{\m{M}}(\vm)\to \bm{A}(\vm)$ is also proper, and hence separated. By Corollary \ref{cor:sp commutativity}, we have a commutative diagram where the proper map $\bm{\m{H}}(\vm)\to \bm{A}(\vm)$ factors through $\bm{Q}$ and the separated map $\bm{\m{M}}(\vm)\to \bm{A}(\vm)$, hence $\bm{Q}$ must be proper. 
\end{proof}

\begin{proof}[Proof of Theorem \ref{spectral-corr}]
Proposition \ref{prop:sp numerical} and Proposition \ref{prop:sp stab} imply that the relative morphism $\bm{Q}:\bm{\m{H}}(\vm) \to \bm{\m{M}}(\vm)$ is well-defined. Moreover, Theorem \ref{thm:recovery} provides a left inverse morphism of $\bm{Q}$, defined on the image of $\bm{Q}$. More precisely, it is defined by composing the pushforward functors associated with each step of the blow-up and the projection $\pi: M\to C$ to the base curve. Since $\bm{Q}$ is also proper by Proposition \ref{prop:sp proper}, it is a closed embedding.

\end{proof}

  In order to obtain an isomorphism between the moduli spaces $\bm{\mathcal{H}}(\vm)$ and $\bm{\mathcal{M}}(\vm)$, one has to define an inverse morphism to $\bm{Q}$. As remarked in the introduction, the presence of sheaves with support on the exceptional divisors causes some exactness arguments to break down. For example, over a non-generic $\vxi\in \nm$, the higher direct image sheaf $R^1p_*F$ along a single blow-up $p: M_1 \to M$ (in the intermediate step) may no longer vanish and the Euler characteristics of the pushforwards may not be preserved. So, we cannot define an inverse morphism $\mvxi\to \hvxi$ simply by pushing the sheaves forward, as is possible in the generic case. It is unclear whether this obstruction is fundamental to achieve an isomorphism or it can be resolved through a careful analysis of stability conditions. Nevertheless, one can avoid these issues by working over the integral locus. 
\begin{theorem}\label{thm:open locus}
    Let $\bm{B}_{int}(\vm) \subset \bm{B}(\vm)$ be the open subset parametrizing integral curves. Then $\bm{Q}:\bm{\m{H}}(\vm) \to \bm{\m{M}}(\vm)$ is an isomorphism over this open locus. 
\end{theorem}
\begin{proof}
    Note that by the commutative diagram in Corollary \ref{cor:sp commutativity}, it makes sense to say that the morphism $\bm{Q}$ is over $\bm{B}(\vm)$ which further projects to $\nm$. We denote by $\bm{\m{H}}_{int}(\vm) $ and $\bm{\m{M}}_{int}(\vm)$ the preimages of $\bm{B}_{int}(\vm)$ in $\bm{\m H}(\vm)$ and $\bm{\m M}(\vm)$ respectively. 

    One can then define a morphism $\bm{P}:\bm{\m{M}}_{int}(\vm)\to  \bm{\m{H}}_{int}(\vm)$ as follows. Let $F_0\in \bm{\m M}_{int}(\vm)$ be a pure dimension one sheaf on $Z_{\vxi}$ and let $F_{n,\ell(n)} = F_0\otimes \O(-\Xi_{1,1}-\dots - \Xi_{n,\ell(n)})$. Starting from $F_{n,\ell(n)}$, define $F_{i,j}$ iteratively by setting
    \[\begin{cases}
       F_{i,j-1} = F_{i,j}\otimes \O(\Xi_{i,j}), &\textrm{if } j\neq 1\\
        F_{i-1,\ell(i-1)} = F_{i,j}\otimes \O(\Xi_{i,j}), &\textrm{if  } j=1
    \end{cases}\]
    By the integrality assumption, we have a natural chain of inclusions:
    \[ F_{n,\ell(n)}\hookrightarrow F_{n,\ell(n)-1} \hookrightarrow... \hookrightarrow F_0.\] 
    Pushing it forward to $M$ yields a parabolic sheaf $\mathcal{G}_\bullet$ with parabolic divisor $L=\pi^{-1}(D)$. By Proposition \ref{prop:parabolic_converse}, $\mathcal{G}_{\bullet}$ satisfies the $(\vxi, L)$-parabolic condition. By further pushing $\mathcal{G}_\bullet$ together with the tautological section forward to $C$ yields a $\vxi$-parabolic Higgs bundle $(E,E^\bullet_D, \Phi, \va)$. It is straightforward to check that the resulting parabolic Higgs bundle has the right numerical invariant. By the integrality assumption, the resulting parabolic Higgs bundle is automatically $\va$-stable. 

    If we look at each step of the blow-up, say $q= p_{n,\ell(n)}^{i,j}: Z_{\vxi}\to M_{i,j}$ (with $j\neq 1$), then we have a parabolic sheaf
    \[  q_* F_{n,\ell(n)} \hookrightarrow \dots \hookrightarrow q_*  F_{i,j}\hookrightarrow q_*F_{i,j-1}\cong q_*( F_{i,j})\otimes \O(E_{i,j}). \]
    Then Theorem \ref{thm:construction_is_inverse} from the analysis for a single blow-up says that the next pushforward is reversible by $\Psi_{p^{i,j-1}_{i,j}}$. Hence, it follows that $\bm{Q}$ is a left inverse morphism to $\bm{P}$. Combining with Theorem \ref{spectral-corr}, it yields that $\bm{Q}$ is an isomorphism over $\bm{B}_{int}(\vm)$. 
\end{proof}

\section{Non-emptiness results}\label{sec:nonemptiness}
\subsection{Controllability of line bundles}
Recall that by Proposition \ref{evaluation description} the non-emptiness problem of $\bmxi$ can be formulated as the existence of sections of line bundles whose local derivatives satisfy certain system of linear equations. We begin by some simple criteria which guarantee the existence of sections with prescribed local derivatives. 
\begin{definition}
    Given a set of distinct points $p_1, \cdots, p_n$, we say that a line bundle $L$ is controllable up to order $(t(p_1),\dots, t(p_n))$ at $(p_1, \cdots, p_n)$ if the restriction $H^0(C,L) \to H^0(C,L|_{t(p_1)p_1+\cdots+t(p_n)p_n})$ is surjective. 
\end{definition}
In other words, if $L$ is controllable up to order $(t(p_1),\dots, t(p_n))$ at $(p_1, \cdots, p_n)$, there exists a section $s \in H^0(C,L)$ and local charts around $p_i$'s such that the local derivatives of $s$ at $p_i$, $s^{(a)}(p_i)$, can be any values in $\mathbb{C}$ for $a<t(p_i)$. In this paper, we will only apply this notion to the line bundles $L=(K_C(D))^{\otimes \mu}$ where $D=p_1+\dots +p_n$ and $\mu=1,\dots, r$. In this case, we always have $H^1(C, L)=0$. 

\begin{lem}\label{controllability1}
Suppose that $H^1(C, L)=0$ and let $L' = L\left(-\sum_{i=1}^n t(p_i)p_i\right).$ Then $L$ is  controllable up to order $(t(p_1),\dots, t(p_n))$ at $(p_1, \cdots, p_n)$ if and only if 
    \[ H^1 \left( C,  L'\right)=0. \]
\end{lem}
\begin{proof}
 The result follows directly from the following long exact sequence
\begin{equation}\label{controllability:les}
0 \to H^0(C,L') \to H^0(C,L) \to H^0(C,L|_{t(p_1)p_1+\dots + t(p_n)p_n}) \to H^1(C, L') \to H^1(C,L)\to \cdots
\end{equation}
\end{proof}
\begin{lem}[Controllability inequality]\label{controllability2}
    Suppose that $H^1(C, L)=0$. If the following inequality holds
    \begin{equation}\label{controllability inequality}
        \sum_{i=1}^nt(p_i) < \mathrm{deg}(L) - (2g-2),
    \end{equation}
    then $L$ is controllable up to order $(t(p_1),\dots, t(p_n))$ at $(p_1,\dots, p_n)$. 
\end{lem}
\begin{proof}
    By Serre duality, we have 
    $H^1(C, L') \cong H^0(C,(L')^\vee\otimes K_C)^\vee$. In particular, when $\mathrm{deg}((L')^\vee\otimes K_C)<0$ which is equivalent to the inequality \eqref{controllability inequality}, we have $H^1(C,L')=0$ and hence the controllability of $L$ by Lemma \ref{controllability1}.
\end{proof}
\begin{rem}
    Suppose that $H^1(C, L)=0$ holds. When $L'\cong K_C$, or equivalently, $L\cong K_C(\sum_{i=1}^n t(p_i)p_i)$ we have $H^1(C, L') = H^0(C, \O)^\vee= \C$. In this case, $L$ is not controllable up to order $(t(p_1),\dots, t(p_n))$ at $(p_1,\dots,p_n)$. By the long exact sequence \eqref{controllability:les}, the image of $H^0(C, L)\to H^0(C, L|_{t(p_1)p_1+\dots+ t(p_n)p_n})$ is described by the kernel of $f: H^0(C, L|_{t(p_1)p_1+\dots+ t(p_n)p_n})\to H^1(C, L')$. For example, when $t(p_1)=\dots= t(p_n)=1$ and $L\cong K_C(\sum_{i=1}^n p_i)$, an element $(b_1,\dots,b_n)\in H^0(C, L|_{p_1+\dots +p_n})$ comes from a section $s\in H^0(C,L)$
    if and only if $f(b_1,\dots,b_n)= b_1+\dots+ b_n=0$. 
 \end{rem}

 \begin{rem}[The $OK$ condition]
 In the setting of parabolic Higgs bundles, we are given $\vm=(\um_1,\dots,\um_n)$. Let $\gamma_{P^i}(\mu)$ denote the level function associated to the partition $P^i=\um_i$. Applying  Lemma \ref{controllability1} to the line bundle $L= (K_C(D))^{\otimes \mu}$ with $t(p_i) = \gamma_{P^i}(\mu)$, we find that the condition that $(K_C(D))^{\otimes \mu}$ is controllable up to order $(\gamma_{P^1}(\mu), \dots, \gamma_{P^n}(\mu))$ at $(p_1,\dots,p_n)$ for $\mu=2, \dots, r$ is equivalent to the $OK$ condition for the collection of line bundles $L_{\vm}(\mu)$, $\mu=2,\dots,r$. 
 \end{rem}

\subsection{Non-emptiness of parabolic Hitchin bases}
    In this section, we will study the non-emptiness problem of the subset $\bmxi\subset |rC_0|$ for any given pair $(\vm,\vxi)$. In general, the problem of showing the existence of divisors of a linear system with prescribed multiplicity conditions at some points on a surface is non-trivial. In our case, we will use the alternative description of  $ \bmxi$ in terms of the subspace of $ A $ defined by the vanishing of evaluations over the level domains associated to the partitions \( P^{\xi^\circ_{i,j}} \) (Proposition \ref{evaluation description})

     Recall from Section \ref{subsec:linearsystem} that for each $p_i\in D$, after choosing a local trivialization of $M$ around $p_i$, we have the evaluations of local derivatives of a section $s_\mu$ at $p_i$, $s_\mu^{(a)}(p_i):=\left. \frac{\del}{\del x_i^a}\right \vert_{0}s_\mu$.
By Proposition \ref{evaluation description}, an element in $\bmxi$ is characterized by $s=(s_1,\dots,s_r)\in A$ satisfying the vanishing of the evaluation maps, which can be written out explicitly as follows: For $i=1,\dots, n$, 
    \begin{equation}\label{eq:gen2}
          \begin{cases}
              \frac{\partial}{\partial y^u}\frac{\partial}{\partial x_i^a} \big |_{(0, \xi^\circ_{i,1})}q_s(x_i,y)=0 & (u,a) \in G(P^{\xi^\circ_{i,1}}) \\
              \qquad \vdots & \\
              \frac{\partial}{\partial y^u}\frac{\partial}{\partial x_i^a} \big |_{(0, \xi^\circ_{i,e(i)})}q_s(x_i,y)=0 & (u,a) \in G(P^{\xi^\circ_{i,e(i)}})
          \end{cases}
    \end{equation}
    where $q_s(x_i,y) = y^r+ s_1(x_i)y^{r-1}+\dots + s_r(x_i)$ and $a,u\in \Z_{\geq 0}$. These are all linear equations that the evaluations $s^{(a)}_{\mu}(p_i)$ have to satisfy.
    \begin{exmp}\label{exmp:linear constraints 1}
        Let $n=1$, $\vm = \um = (3,2,1)$ and $\vxi = \uxi = (\xi_1,\xi_2,\xi_3)$ with $\xi_1=\xi_3$. So, the distinct part $\uxi^\circ=  (\xi^\circ_1,
        \xi^\circ_2) = (\xi_1,\xi_2)$ and 
        \[P^{\xi^\circ_1} = (3,1),\quad P^{\xi^\circ_2} =  (2) \]
        The corresponding level domains are shown below
          \begin{center}
$G(P^{\xi^\circ_1} )=$\begin{tabular}{c|cccc}
$a=2$ &$\bullet$&&&\\
$a=1$ &$\bullet$&$\bullet$&&\\
$a=0$ &$\bullet$&$\bullet$&$\bullet$&$\bullet$\\
\hline
 & $u=0$ & $u=1$ & $u=2$ & $u=3$  \\
\end{tabular} \quad $G(P^{\xi^\circ_2} )=$ \begin{tabular}{c|ccc}
$a=1$ &$\bullet$&\\
$a=0$ &$\bullet$&$\bullet$\\
\hline
 & $u=0$ & $u=1$  \\
\end{tabular}
\end{center}
        
        Then the equations indexed by $G(P^{\xi^\circ_{1}})$ are explicitly given by 
        \setlength{\arraycolsep}{2pt}
        \[\begin{array}{ccccccccccccc}
            (\xi_1^\circ)^5s_1(p) &+&(\xi_1^\circ)^4s_2(p) &+&(\xi_1^\circ)^3 s_3(p)&+&
            (\xi_1^\circ)^2s_4(p) &+&(\xi_1^\circ)s_5(p) &+&s_6(p) &=&-(\xi_1^\circ)^6\\
            5(\xi_1^\circ)^4s_1(p) &+&4(\xi_1^\circ)^3s_2(p) &+&3(\xi_1^\circ)^2 s_3(p)&+&
            2(\xi_1^\circ)s_4(p) &+&s_5(p) &&&=&-6(\xi_1^\circ)^5\\
            20(\xi_1^\circ)^3s_1(p) &+&12(\xi_1^\circ)^2s_2(p) &+&6(\xi_1^\circ) s_3(p)&+&
            2s_4(p) & &&&&=&-30(\xi_1^\circ)^4\\
            60(\xi_1^\circ)^2s_1(p) &+&24(\xi_1^\circ)s_2(p) &+& 6s_3(p)&
             & &&&&&=&-120(\xi_1^\circ)^3\\
             (\xi_1^\circ)^5s_1^{(1)}(p) &+&(\xi_1^\circ)^4s_2^{(1)}(p) &+&(\xi_1^\circ)^3 s_3^{(1)}(p)&+&
            (\xi_1^\circ)^2s_4^{(1)}(p) &+&(\xi_1^\circ)s_5^{(1)}(p) &+&s_6^{(1)}(p) &=&0\\
            5(\xi_1^\circ)^4s_1^{(1)}(p) &+&4(\xi_1^\circ)^3s_2^{(1)}(p) &+&3(\xi_1^\circ)^2 s_3^{(1)}(p)&+&
            2(\xi_1^\circ)s_4^{(1)}(p) &+&s_5^{(1)}(p) &&&=&0\\
            (\xi_1^\circ)^5s_1^{(2)}(p) &+&(\xi_1^\circ)^4s_2^{(2)}(p) &+&(\xi_1^\circ)^3 s_3^{(2)}(p)&+&
            (\xi_1^\circ)^2s_4^{(2)}(p) &+&(\xi_1^\circ)s_5^{(2)}(p) &+&s_6^{(2)}(p) &=&0
        \end{array}\]
        and the equations indexed by $G(P^{\xi^\circ_{2}})$ are explicitly given by 
        \setlength{\arraycolsep}{2pt}
        \[\begin{array}{cccccccccccccc}
            (\xi_2^\circ)^5s_1(p) &+&(\xi_2^\circ)^4s_2(p) &+&(\xi_2^\circ)^3 s_3(p)&+&
            (\xi_2^\circ)^2s_4(p) &+&(\xi_2^\circ)s_5(p) &+&s_6(p) &=&-(\xi_2^\circ)^6\\
            5(\xi_2^\circ)^4s_1(p) &+&4(\xi_2^\circ)^3s_2(p) &+&3(\xi_2^\circ)^2 s_3(p)&+&
            2(\xi_2^\circ)s_4(p) &+&s_5(p) &&&=&-6(\xi_2^\circ)^5\\
             (\xi_2^\circ)^5s_1^{(1)}(p) &+&(\xi_2^\circ)^4s_2^{(1)}(p) &+&(\xi_2^\circ)^3 s_3^{(1)}(p)&+&
            (\xi_2^\circ)^2s_4^{(1)}(p) &+&(\xi_2^\circ)s_5^{(1)}(p) &+&s_6^{(1)}(p) &=&0
        \end{array}\]
    \end{exmp}

    Define the vector space 
    \[V:= \bigoplus_{\mu=1}^r H^0\left(C,\left(K_C(D)\right)^{\otimes \mu}|_{\sum_{i=1}^n m_{i,1}p_i }\right).\]
    Given the choice of local trivializations around $p_i\in D$, there is a natural choice of basis $\{e_{\mu, a, i}\}$ of $V$ such that the restriction of $s_\mu$ in $H^0(C,(K_C(D))^{\otimes \mu}|_{ m_{i,1}p_i})$ can be written as $ \sum_{a=0}^{m_{i,1}-1} s_{\mu}^{(a)}(p_i) e_{\mu,a,i}. $ The index set of the basis is given by
    \[\Pi:= \{(\mu, a, i)\in \Z^{\geq 0}| 1\leq\mu\leq r , 0\leq a< m_{i,1}, 1\leq i\leq n\}.\]
    Hence, we can regard $V$ as the set of relevant evaluations $s^{(a)}_{\mu}(p_i)$. 
    We denote by \( \Pi(a_0, i_0) \) the subset of \( \Pi \) where \( a = a_0 \) and \( i = i_0 \), i.e.,  
\[
\Pi(a=a_0, i=i_0) = \{(\mu, a, i) \in \Pi \mid a = a_0, i = i_0 \}.
\]  
Similarly, we define \( \Pi(\mu=\mu_0) \) as the subset of \( \Pi \) where \( \mu = \mu_0 \), i.e.,  
\[
\Pi(\mu= \mu_0) = \{(\mu, a, i) \in \Pi \mid \mu = \mu_0 \},
\]  
and so on for other fixed entries.

    The set of all linear equations \eqref{eq:gen2} can be divided into different systems of equations labeled by $a$ and $i$, each of which only involves $s_1^{(a)}(p_i) ,\dots, s_r^{(a)}(p_i)$. For each eigenvalue $\xi^\circ_{i,j}$, the number of equations involving $x$-derivative of order $a$ is \[c(a,\xi^\circ_{i,j}) := \#\{u| \gamma_{P^{\xi^\circ_{i,j}}}(u)\geq a\}.\]Hence, the number of equations involving $s_1^{(a)}(p_i) ,\dots, s_r^{(a)}(p_i)$ is 
    $c(a,i):= \sum_{j=1}^{e(i)} c(a,\xi^\circ_{i,j})$
    and we denote by $A_{a,i}X_{a,i} = B_{a,i}$ the system of equations of size $c(a,i)\times r$ and regard $X_{a,i} = (x_{\mu,a,i})^T_{1\leq \mu\leq r}$ as variables where $x_{\mu,a,i}:V\to \C$ is the dual basis of $e_{\mu,a,i}$. Note that for $a>0$, the system of equations is homogeneous i.e. $B_{a,i}=0$. More explicitly, define the following $c\times r$ matrix where the first row is $(\xi^{r-1},\xi^{r-2},\dots,\xi,1)$ and the $k$-th row is the $k$-th derivative of the first row (divided by $k!$): 
    \begin{equation*}
        R(\xi,c) = \begin{pmatrix}
        \xi^{r-1}&\dots &\dots & &\xi&1\\
        (r-1)\xi^{r-2}&\dots &\dots& & 1&0\\
        &&\vdots&\iddots&&\\
            \frac{(r-1)!}{(c-1)!}\xi^{r-c}&  \cdots &(r-c-1)!&0&\dots &0
        \end{pmatrix}
    \end{equation*}
    and similarly the $c\times 1$ matrix $K(\xi,c )= (\xi^r,r\xi^{r-1},\dots, \frac{r!}{c!}\xi^{r-c+1})^T$.
    Then the coefficient matrix is given by 
    \begin{equation*}
        A_{a,i} = \begin{pmatrix}
            R(\xi_{i,1}^\circ, c(a, \xi_{i,1}^\circ))\\
            \vdots\\
            R(\xi_{i,e(i)}^\circ, c(a, \xi_{i,e(i)}^\circ)) 
        \end{pmatrix}
    \end{equation*}
    Meanwhile, we have $B_{a,i}=0$ for $a>0$ and 
    \[B_{0,i} = \begin{pmatrix}
        K(\xi_{i,1}^\circ, c(a, \xi_{i,1}^\circ))\\
            \vdots\\
            K(\xi_{i,e(i)}^\circ, c(a, \xi_{i,e(i)}^\circ))
    \end{pmatrix}\]
    
    \begin{rem}\label{rem:vandermonde}
    Note that we divide the equations in \eqref{eq:gen2} by some factor of $k!$ to align with the form of a generalized matrix. The $c(a,i)\times c(a,i)$ submatrix $A'_{a,i}$ formed by the last $c(a,i)$ columns is a generalized Vandermonde matrix \cite{Vandermonde}. A fundamental property about the generalized Vandermonde matrix $A'_{a,i}$ is that $A'_{a,i}$ is invertible if and only if the entries $\xi_{i,j}^\circ$ are pairwise distinct i.e.  $\xi_{i,j}^\circ\neq \xi_{i,j'}^\circ$ for all $j\neq j'$ \cite[Page 19]{Vandermonde}, which is true under our assumption of $\vxi^\circ$. 
    \end{rem}
     
    Thus, the systems of linear equations can be expressed succinctly in terms of a linear map
    \[ T: V \to \C^s, \quad \textrm{where} \quad s= \sum_{a,i} c(a,i) \]
    such that the solution space is given by $T^{-1}(\beta)$ where $\beta$ is the vector formed by all $B_{a,i}$ for $a\geq 0$.
    
    \begin{exmp}\label{exmp:linear constraints 2}
        Continuing with Example \ref{exmp:linear constraints 1}, the sets of equations are divided into systems of equations $A_aX_a= B_a$ where $a=0,1,2$ (here $i=1$, so we drop the subscript $i$) and the coefficient matrices are given by 
        \begin{align*}
        A_0&= 
        \begin{pmatrix}
            (\xi_1^\circ)^5 &(\xi_1^\circ)^4 &(\xi_1^\circ)^3 &
            (\xi_1^\circ)^2 &(\xi_1^\circ) & 1\\
            5(\xi_1^\circ)^4&4(\xi_1^\circ)^3 &3(\xi_1^\circ)^2 &
            2(\xi_1^\circ)&1 &0\\
            10(\xi_1^\circ)^3 &6(\xi_1^\circ)^2&3(\xi_1^\circ)&
            1 &0 &0\\
            10(\xi_1^\circ)^2 &4(\xi_1^\circ) &1&0
             & 0&0\\
(\xi_2^\circ)^5 &(\xi_2^\circ)^4&(\xi_2^\circ)^3 &
            (\xi_2^\circ)^2 &(\xi_2^\circ) &1 \\
            5(\xi_2^\circ)^4 &4(\xi_2^\circ)^3 &3(\xi_2^\circ)^2 &
            2(\xi_2^\circ) &1 &0
        \end{pmatrix}, &B_0 &= \begin{pmatrix}
        -(\xi_1^\circ)^6\\-6(\xi_1^\circ)^5\\-15(\xi_1^\circ)^4\\-20(\xi_1^\circ)^3\\-(\xi_2^\circ)^6\\-6(\xi_2^\circ)^5    
        \end{pmatrix}\\
        A_1 &= \begin{pmatrix}
            (\xi_1^\circ)^5 &(\xi_1^\circ)^4 &(\xi_1^\circ)^3 &
            (\xi_1^\circ)^2 &(\xi_1^\circ) & 1\\
            5(\xi_1^\circ)^4&4(\xi_1^\circ)^3 &3(\xi_1^\circ)^2 &
            2(\xi_1^\circ)&1 &0\\
            (\xi_2^\circ)^5 &(\xi_2^\circ)^4&(\xi_2^\circ)^3 &
            (\xi_2^\circ)^2 &(\xi_2^\circ) &1
        \end{pmatrix},  &B_1&=\begin{pmatrix}
            0\\0\\0
        \end{pmatrix}\\
        A_2&=\begin{pmatrix}
            (\xi_1^\circ)^5 &(\xi_1^\circ)^4 &(\xi_1^\circ)^3 &
            (\xi_1^\circ)^2 &(\xi_1^\circ) & 1
        \end{pmatrix},  &B_2&=\begin{pmatrix}
            0
        \end{pmatrix}
        \end{align*}
        
    \end{exmp}

    On the other hand, we consider the subspace of $V$ which comes from global sections i.e.
    \[ S:= \mathrm{Im}\left(\bigoplus_{\mu=1}^r H^0\left(C,(K_C(D))^{\otimes \mu}\right) \to V\right) \subset V.\]
    As our goal is to find sections satisfying the constraints \eqref{eq:gen2}, the question can be formulated as follows: 
    \begin{ques}
        Does the affine subspace $T^{-1}(\beta)$ intersect the linear subspace $S$ non-trivially in $V$? 
    \end{ques} 
    
    Define the standard decomposition $\Pi=\Pi_{\textrm{pivot}}\sqcup \Pi_{\textrm{free}}$ where 
    \begin{equation}\label{standard}
        \Pi_{\textrm{pivot}}:= \{(\mu, a, i)\subset \Pi| r-c(a,i)+1\leq \mu\leq r\} , \quad \Pi_{\textrm{free}} := \Pi_{\textrm{pivot}}^c\subset \Pi.
    \end{equation}
    \begin{lem}\label{graphlike}
    Let $V= V_{\textrm{free}}\oplus V_{\textrm{pivot}}$ be the direct sum decomposition induced by the standard decomposition \eqref{standard}. Then there is an affine map $H:V_{\textrm{free}}\to V_{\textrm{pivot}}$ whose graph is $T^{-1}(\beta)$.
    \end{lem}
        \begin{proof}
            Since $T^{-1}(\beta)$ is the solution set of the systems of equations $A_{a,i}X_{a,i}=B_{a,i}$. When $a=0$, the matrix $A_{0,i}$ is always a square generalized Vandermonde matrix, which is invertible, so there is always a unique solution for $A_{0,i}X_{0,i}=B_{0,i}$. 
            
            For $a>0$ and $1\leq i\leq n$, in order to solve the homogeneous system of equations $A_{a,i}X_{a,i}=0$, it suffices to find $c(a,i)$ columns which form an invertible submatrix of $A_{a,i}$ such that $A_{a,i}X_{a,i}=0$ is equivalent to an expression of the pivot variables (corresponding to the choice of columns) in terms of the free variables. In our case, we take the submatrix $A'_{a,i}$ formed by the last $c(a,i)$ columns  i.e. $(\mu, a,i)\in \Pi_{\textrm{pivot}}$  which is invertible by Remark \ref{rem:vandermonde}.  Hence, $\Pi_{\textrm{pivot}}$ and $\Pi_{\textrm{free}}$ are the index sets of the pivot variables and free variables respectively.

            Therefore, we get an affine map $R: V_{\textrm{free}}\to V_{\textrm{pivot}}$ from the expression of pivot variables in terms of free variables  and the solution set $T^{-1}(\beta)$ is then the graph of $H$.

            \end{proof}
        
        \begin{prop}\label{prop:intesection}
        Let $\Pi=\Pi_{\textrm{pivot}}\sqcup \Pi_{\textrm{free}}$ be the standard decomposition. Define 
        \[t(\mu_0, i_0):= \#\{a| (\mu_0,a, i_0)\in \Pi_{\textrm{pivot}}\}\] for each pair $1\leq \mu_0\leq r, 1\leq i_0\leq n$. Suppose that $\vec{\xi} \in \nm$ and $(K_C(D))^{\otimes \mu}$ is controllable up to order $(t(\mu,1),\dots,t(\mu,n))$ at $(p_1,\dots,p_n)$ for $\mu=2,\dots, r$. Then $S\cap T^{-1}(\beta)\neq \emptyset.$ 
    \end{prop} 
    \begin{proof}
         In terms of the vector spaces defined above, the assumptions that $\vec{\xi} \in \nm$ and the controllability of the line bundles $K_C(D)^{\otimes \mu}$ means that the composition $S\to V\to V_{\textrm{pivot}}$ is surjective. Moreover, if we write $S=\bigoplus_{\mu=1}^r S_\mu $ and $V_{\textrm{pivot},\mu}=\bigoplus_{\mu=1}^r V_{\textrm{pivot},\mu}$ the direct sum decomposition induced by $V=\bigoplus_{\mu=1}^r V_\mu$, then $S_\mu\to V_{\textrm{pivot},\mu}$ is surjective.

        We will find an element $v =\sum\limits_{(\mu,a,i)\in \Pi}c_{\mu,a,i}e_{\mu,a,i}\in S\cap T^{-1}(\beta)$ by induction. For simplicity, we call $c_{\mu,a,i}$ a pivot (resp. free) coefficient with $\mu=\mu_0$ if $(\mu, a,i)\in \Pi_{\textrm{pivot}}(\mu=\mu_0)$ (resp. $\Pi_{\textrm{free}}(\mu=\mu_0))$. Similarly, we call $c_{\mu,a,i}$ a pivot (resp. free) coefficient with $a=a_0,i=i_0$ if $(\mu, a,i)\in \Pi_{\textrm{pivot}}(a=a_0,i=i_0)$ (resp. $\Pi_{\textrm{free}}(a=a_0,i=i_0))$

        Since $A_{a,i}X_{a,i}=B_{a,i}$ always has a unique solution for $a=0$, the coefficients $c_{\mu,a,i}$ are uniquely determined for $(\mu,a,i)\in \Pi(a=0)$. These determine all the pivot coefficients with $\mu=1$ which yield a vector in $V_{\textrm{pivot},\mu=1}$. Then we can lift this vector via the surjection $S_1\twoheadrightarrow V_{\textrm{pivot},1}$ to a vector in $S_1$ which determines all the free coefficients with $\mu=1$. Now, suppose all the coefficients are determined up to $\mu=k\leq r-1$ and we want to determine the coefficients for $\mu=k+1$. For each fixed $a_0$ and $i_0$, if $c_{k+1,a_0,i_0}$ is a pivot coefficient, the set of coefficients $c_{k-\delta,a_0,i_0}$ with $\delta= 1, \dots, k$ include all the free coeffcients with $a=a_0,i=i_0$ and are all determined by induction hypothesis. So, by applying the map $H:V_{\textrm{free}}\to V_{\textrm{pivot}}$ in Lemma \ref{graphlike}, the pivot coefficient $c_{k+1,a_0,i_0}$ is uniquely determined. Once all the pivot coefficients with $\mu=k+1$ are determined, we can use the surjective map $S_{k+1}\twoheadrightarrow V_{\textrm{pivot}, \mu=k+1}$ to lift to a vector in $S_{k+1}$ which determines the free coefficients with $\mu=k+1$. Hence, all the coefficients $c_{\mu,a,i}$ are determined and the vector $v =\sum\limits_{(\mu,a,i)\in \Pi}c_{\mu,a,i}e_{\mu,a,i}$ lies in both $S$ and $T^{-1}(\beta)=\textrm{graph}(H)$ by construction.

    \end{proof}
    \begin{exmp}\label{exmp:linear constraints 3}
        Continuing with Example \ref{exmp:linear constraints 2}, we can concretely visualize the proof. In the table below, each dot represents a variable labeled by a pair \((a, \mu)\), forming the index set \(\Pi\). The decomposition \(\Pi = \Pi_{\mathrm{free}} \sqcup \Pi_{\mathrm{pivot}}\) is represented by dots of different colors: red dots for free variables \(\Pi_{\mathrm{free}}\) (the free variables), and blue dots for the pivot variables \(\Pi_{\mathrm{pivot}}\).

        \begin{center}
\begin{tabular}{c|cccccc}
$a=2$ & $\textcolor{red}{\bullet}$ &$\textcolor{red}{\bullet}$  &$\textcolor{red}{\bullet}$ &$\textcolor{red}{\bullet}$&$\textcolor{red}{\bullet}$&$\textcolor{blue}{\bullet}$ \\
$a=1$ & $\textcolor{red}{\bullet}$ & $\textcolor{red}{\bullet}$& $\textcolor{red}{\bullet}$&$\textcolor{blue}{\bullet}$&$\textcolor{blue}{\bullet}$&$\textcolor{blue}{\bullet}$\\
$a=0$ &$\textcolor{blue}{\bullet}$&$\textcolor{blue}{\bullet}$&$\textcolor{blue}{\bullet}$&$\textcolor{blue}{\bullet}$&$\textcolor{blue}{\bullet}$&$\textcolor{blue}{\bullet}$ \\
\hline
 & $\mu=1$ & $\mu=2$ & $\mu=3$ & $\mu=4$ & $\mu=5$ &$\mu=6$ \\
\end{tabular}
\end{center}
The goal is to find a vector $v\in S\cap T^{-1}(\beta)$. 
\begin{itemize}
    \item ($v\in T^{-1}(\beta)$) To solve the systems of equations $A_aX_a=B_a$ for a fixed $a$, we first select values for the free variables (the red dots), which then determine the values of the pivot variables (the blue dots) in the same row. 

    \item  ($v\in S$) We also need to ensure that the choice of values for all the variables can be lifted to sections $s_\mu$ of the line bundles $(K_C(D))^{\otimes \mu}$. For a fixed $\mu$, under the controllability assumption for $\mu>1$ and $\vxi\in \nm $ for $\mu=1$, a section $s_\mu$ can be found for any pivot variable values, which in turn determines the free variables. Pictorially, this means that the blue dots in a column determine the red dots in the same column. Here the values of the pivot variables for $a=0$ are uniquely determined. 
\end{itemize}
The idea of the proof then proceeds by moving from left to right across the columns. For each fixed $\mu$, the blue dots determine the remaining  red dots in the column. Once the red dots in a row are determined, they determine the remaining blue dots in that row. This process continues across all columns, and by the time we reach the final column, all the variables (both red and blue dots) are determined.

    \end{exmp}
        \begin{prop}\label{thm:criterion} Continuing with the notation in Proposition \ref{prop:intesection}, we have 
        \[t(\mu, i)= \gamma_{P^i}(\mu) \quad \textrm{for  } 1\leq \mu\leq  r,1\leq i\leq n\]
    where $P^i= P^{\xi^\circ_{i,1}}\cup \cdots \cup P^{\xi^\circ_{i,e(i)}}$. 
    \end{prop}
    \begin{proof}
    In the standard decomposition $\Pi=\Pi_{\textrm{pivot}}\sqcup \Pi_{\textrm{free}}$, 
           the elements in $\Pi_{\textrm{pivot}}$ correspond to the set of pivot variables which is in (non-canonical) bijection with the set of linear equations where the $a$ and $i$ entries in $(\mu, a,i)\in \Pi_{\textrm{pivot}}$ denotes the variables involved in $A_{a,i}X_{a,i}= B_{a,i}$. The set of linear equations are in bijection with the set $G(P^{\xi^\circ_{i,k}})$ where an equation labeled by $(u, a)\in G(P^{\xi^\circ_{i,k}})$ belongs to the system of equations $A_{a,i}X_{a,i}= B_{a,i}$. So, we have a (non-canonical) bijection which preserves the $a$ and $i$ entries on both sides while matching the $\mu$ and $u$ entries 
           \[\Pi_{\textrm{pivot}}\leftrightarrow \bigsqcup\limits_{1\leq i \leq n, 1\leq k\leq e(i)} G(P^{\xi^\circ_{i,k}}).\]  
            Since $t(\mu_0, i_0)$ counts the number of different $a$'s in $\Pi_{\textrm{pivot}}$ for each fixed $\mu_0$ and $i_0$, we can count this number on the other side as well. On the other side, the number of different $a$'s for each fixed $u_0$ and $i_0$ is given by the level function $\gamma_{P^{\xi^\circ_{i_0,k}}}(u_0)$.
            
            By the definition of $\Pi_{\textrm{pivot}}$, $t(\mu, i_0)$ is increasing in $1\leq \mu\leq r$ for fixed $i_0$. On the other side, arranging $\gamma_{P^{\xi^\circ_{i_0,k}}}(u_0)$ in an increasing order for a fixed $i_0$ is simply the level function associated to $P^i=\bigcup_{k=1}^{e(i_0)} P^{\xi^\circ_{i_0,k}}$. Hence, we have $t(\mu, i)= \gamma_{P^i}(\mu)$. 
    \end{proof}

    \begin{cor}\label{thm:nonemptiness of Hitchin bases}Suppose that the $OK$ condition holds. Then for any $\vec{\xi} \in \nm$, the affine space $\bmxi$ is non-empty. In particular, 
    \begin{enumerate}
        \item when $n\geq 3$ and $g=0$, if the inequalities
        \begin{equation}\label{eq:intro-defect}
             \sum_{i=1}^n \gamma_{P^i}(\mu) < (n-2)\mu +2
        \end{equation}
        hold for $\mu=2, \dots, r$, then $\bmxi\neq \emptyset.$ 
        \item when $n \geq 1$ and $g \geq 1$, if one of the partitions $P^i$ is not the singleton partition $m_1=r$, then $\bmxi \neq \emptyset$.
    \end{enumerate}
    \end{cor}
    \begin{proof}
    Combining Proposition \ref{prop:intesection} and Proposition \ref{thm:criterion}, we see that $\bmxi$ is non-empty if $(K_C(D))^{\otimes \mu}$ is controllable up to order $(\gamma_{P^1}(\mu),\dots, \gamma_{P^n}(\mu))$ at $(p_1,\dots,p_n)$ for $\mu=2,\dots, r$. The controllability can be guaranteed by the controllability inequalities in Lemma \ref{controllability2}: \begin{equation}\label{controllability3}
         \sum_{i=1}^n \gamma_{P^i}(\mu) < \mathrm{deg}((K_C(D))^{\otimes \mu}) -(2g-2) = (2g-2)(\mu-1) +n\mu  \quad \textrm{for  }\mu=2, \dots, r.
     \end{equation}
        \begin{enumerate}
            \item When $g=0$, the right-hand side of \eqref{controllability3} becomes $(n-2)\mu+ 2$. 
        \item When $g\geq 1$, note that we always have the weak inequalities $\gamma_{P^i}(\mu)\leq \mu$ for $2\leq \mu\leq r$. Hence, we always have 
        \[\sum_{i=1}^n \gamma_{P^i}(\mu) \leq \mu n \quad \textrm{for }2\leq \mu\leq r.\]
        For $g\geq2$, we have $\mu n<(2g-2)(\mu-1) + \mu n$, so the controllability inequalities always hold for all choices of partitions $P^i$. For $g=1$, if one of the partitions $P^{i_0}$ is not the singleton partition for some $i_0$, then $\gamma_{P^{i_0}}(\mu)<\mu$ is strict for $2\leq \mu\leq r$, so the controllability inequalities also hold.
        \end{enumerate}
    \end{proof}

\subsection{Non-emptiness of moduli spaces}\label{sec:nonempty moduli}
We can now combine the non-emptiness of $\bmxi$ and the relative spectral correspondence to construct stable $\vxi$-parabolic Higgs bundles. Before proceeding to the main theorem, we will need to slightly enhance the $OK$ condition to guarantee that there exists an integral curve in $\bmxi$.

\begin{lem}\label{lem:integral curves} 
Suppose $H^0(C,L(\vm)_r)\neq 0$. There exists $s=(s_1,\dots,s_r)\in A(\vm)_0$ such that the corresponding spectral curve $C_s\subset Tot(K_C(D))$ is integral. 
\end{lem}
\begin{proof}
 Let $s_1=\dots =s_{r-1}=0$ and $a_s$ be a non-zero section in $H^0(C,L_r)$ . So, the spectral curve is of the form of a cyclic cover $C_s=\{y^r+ s_r=0\}\subset Tot(K_C(D))$. Note that \cite[Remark 3.1]{bnr}
 here we treat $s_r$ as a section in $(K_C(D))^{\otimes r}$ so that $div(s_r) = Z+ \sum_{i=1}^n\gamma_{P^i}(r)p_i$ for some effective divisor $Z\in |L(\vm)_r|.$ Recall that a cyclic cover is integral if $div(s_r)$ is not of the form $mZ'$ for some $Z'\in |(K_C(D))^{\otimes k}|$ where $m\vert r$. Since $|L(\vm)_r|$ is a linear system with base point, by \cite[Remark 10.9.2]{hartshorne}, we can assume that  $Z$ is reduced and away from $\sum_{i=1}^n\gamma_{P^i}(r)p_i. $ Then, it is clear that $div(s_r)$ is not of the form $mZ'$. Hence, $C_s$ is integral.
\end{proof}

\begin{rem}
    If $H^0(C, L_{\vm}(r))=0$,  we must have $s_r=0$, so the spectral curve defined by  $y^r+s_1y^{r-1}+\dots +s_{r-1}y=0$ must be reducible as it always contains the zero section as an irreducible component. In particular, no spectral curve is integral in this case. For example, this happens when $C=\P^1$ and $\deg(L(\vm)_r) = r(n-2) - \sum_{i=1}^n \gamma_{P^i}(r) <0$. When $g>0$, it is easy to check that $\deg(L(\vm)_r)>0$ for all choice of $\vm$, so $\dim H^0(C, L(\vm)_r)>0.$ 
\end{rem}

\begin{lem}\label{lem:integral at zero}
    Suppose $H^1(C, \lmmu\otimes \O(-p_i))= 0$ for $\mu=2,\dots, r$ and $i=1,\dots, n$. Then there exists an integral curve $\Sigma\in B(\vm)_0$ contained in $Z_0.$ 
\end{lem}
\begin{proof}
    First, observe that $H^1(C,  L(\vm)_r\otimes \O(-p_i))= 0$ implies the assumption $H^0(C, L(\vm)_r)\neq 0$ in Lemma \ref{lem:integral curves}. Indeed, consider the long exact sequence    \begin{equation}\label{eq:les}
    H^0(C, \lmmu\otimes \O(-p_i)) \to H^0(C, \lmmu) \to H^0(p_i, \lmmu|_{p_i})\to H^1(C, \lmmu\otimes \O(-p_i))\to \dots 
    \end{equation}
    Setting $\mu=r$, the claim follows from the vanishing of the last term and the fact that $H^0(p_i, \lmmu|_{p_i})\neq 0$. Then Lemma \ref{lem:integral curves} says that there is an open subset $U\subset A(\vm)_0$ representing integral curves in $\Tot(K_C(D)).$ Let $s = (s_1,\dots,s_r)\in U$ with $C_s\subset \Tot(K_C(D))$. Let $p:Z_{\vxi}\to \Tot(K_C(D))$ be the composition of blow-ups. Under the identification $B(\vm)_0\cong A(\vm)_0$ in Corollary \ref{identification of hitchin bases}, 
    \[C_s\longleftrightarrow \sm:= p^* C_s- \sum_{i=1}^n\sum_{j=1}^{\ell(i)}m_{i,j} \Xi_{i,j}.\]

    If the curve $\sm$ is exactly the strict transform of $C_s$ in $Z_{\vxi}$, then $\sm$ is also integral. So, it suffices to check that $C_s$ and its total transform in the blow-up have multiplicty exactly $m_{i,j}$ through the blow-up center. By Remark \ref{rem:local condition(minimal)}, we need to check that vanishing order of $s_\mu$ at $p_i$ equals exactly $\gamma_{P_i}(\mu)$ for $\mu=1,\dots,r$ and all $p_i\in D$ (in fact, it suffices to check the minimal indices $\mu$). We claim that this is true for generic $(s_1,\dots,s_r)\in U$. Indeed, it follows again from the long exact sequence \eqref{eq:les} and the vanishing of $H^1$ in the assumption that \[\dim(H^0(C, \lmmu\otimes \O(-p_i))) <\dim( H^0(C, \lmmu)).\]
    \end{proof}

\begin{prop}\label{nonemptiness of Higgs} For every  $\vec{\xi} \in \nm$, the moduli space of stable $\vec{\xi}$-parabolic Higgs bundle $\mathcal{H}(\vec{m})_{\vec{\xi}}$ is non-empty in the following cases:
    \begin{enumerate}
        \item When $n\geq 3$ and $g=0$, if the inequalities
        \begin{equation}\label{eq:defect-integral}
             \sum_{i=1}^n \gamma_{P^i}(\mu) < (n-2)\mu +1
        \end{equation}
        hold for $\mu=2, \dots, r$
        \item When $n \geq 2$ and $g = 1$, if at least two of the partitions $P^i$ are not the singleton partition $m_{i,1}=r$.
        \item When $n\geq 1$ and $g\geq 2.$
    \end{enumerate}
    \end{prop}

\begin{proof}[Proof of Proposition \ref{nonemptiness of Higgs}]
 Note that when $g=0,1$, the conditions are stronger than the one in Theorem \ref{thm:nonemptiness of Hitchin bases}. So, $\bmxi\neq \emptyset$ in all cases.

    Moreover, similar to the argument in the proof of Theorem \ref{thm:nonemptiness of Hitchin bases}, the strengthened inequalities in this theorem 
    are required to guarantee that $H^1(C, \lmmu\otimes \O(-p_i))=0$. So, we can apply Lemma \ref{lem:integral at zero} which  guarantees that there exists an integral curve $\sm_0\in B(\vm)_0$ contained in $Z_0$. Then we can apply Proposition \ref{deformation of curves} to obtain an integral curve $\sm_{\vxi}$ in $\bmxi$ for all $\vxi\in \nm$. By choosing a line bundle $L$ on $\sm_{\vxi}$, it must be $\beta_{\vxi}$-twisted $A_{\vxi}$-Gieseker for any choice of parameters $\beta_{\vxi}$ and $A_{\vxi}$. So, the pure dimension one sheaf on $Z_{\vxi}$ formed by $L$ on $\sm_{\vxi}$ is an element in $\hvxi$. We can now apply the spectral correspondence (Theorem \ref{spectral-corr}) to obtain the desired stable $\vec{\xi}$-parabolic Higgs bundles. 
\end{proof}

\begin{rem}\label{rem:optimalcond nonempty}
\begin{enumerate}
    \item The condition in Proposition~\ref{nonemptiness of Higgs} is not optimal. In particular, when $g = 0$, it suffices to check \eqref{eq:defect-integral} only for $\mu \in J_{\min}$ (see Remark~\ref{rem:local condition(minimal)} for the definition).
    \item If one asks about the non-emptiness of $\m{H}(\vm)_{\vxi}$ for a particular $\vxi \in \nm$, rather than for all $\vxi \in \nm$, the conditions can be relaxed depending on $\vxi$. We use this observation to study the Deligne--Simpson problem in Section~\ref{sec:DSP}.
\end{enumerate}
\end{rem}

\begin{rem}
   When $\vec{m}=(\underline{1}, \dots, \underline{1})$, the non-emptiness of $\hvxi \cong \hpar_{\vec{\xi}}$ follows from that of meromorphic Higgs bundles studied by Markman \cite{markmanspectral}. Under the assumption that the linear system $(K_C(D))^{\otimes r}$ is very ample (which translates to some genus and number of marked points assumption), the generic spectral curve is smooth by Bertini's theorem. So, at least for some generic (more generic than our definition) $\vec{\xi}$, there exist $\vec{\xi}$-parabolic Higgs bundles. Also, since the Hitchin map $h:\hpar \to A$ is proper \cite{Yokogawa-compactification} and its image must contain the proper dense subset of smooth spectral curves, we see that the Hitchin map is surjective.

    However, when $\vec{m}\neq(\underline{1}, \dots, \underline{1})$, no spectral curve in $\Tot(K_C(D))$ is smooth or integral. Therefore, we may not be able to apply the similar argument in this case. 
\end{rem}

\subsection{Multiplicative Deligne--Simpson problem}\label{sec:DSP}
As outlined in the introduction, we study the multiplicative Deligne–Simpson problem (DSP, for short) via the spectral correspondence. For historical background and motivation of DSP, we refer the reader to Section \ref{sec:intro DSP}.

\begin{definition}
    Given conjugacy classes $C_1,\dots,C_n$ in $GL_r(\C)$, we say that the multiplicative DSP is solvable for the tuple of conjugacy classes $\{C_i\}$ if there exist irreducible solutions to the equation $T_1\cdots T_n =  Id_r$ with $T_i\in C_i$ where an irreducible solution means that the matrices $T_j$ have no common invariant subspace. 
\end{definition}

Let us recall some definitions in the tame non-abelian Hodge correspondnce (NAHC) of Simpson \cite{simpson-noncompact}. 

\begin{definition}
    A filtered local system on a punctured curve $C \setminus \{p_1,\dots,p_n\}$ is a local system $\mathbb{L}$ together with a decreasing, left-continuous filtration $ \bigcup_{\beta \in \mathbb{R}}\mathbb{L}_{\widetilde{p}_i}^\beta$ on the stalk $\mathbb{L}_{\widetilde{p}_i}$, preserved by the local monodromy $T_i$ at $p_i$, where $\widetilde{p}_i$ is a point nearby $p_i$. 
\end{definition}
One can define the filtered degree of a filtered local system and a stability condition. For our purposes, it suffices to note that when the filtrations are trivial i.e. when $\mathbb{L}_{\widetilde{p}_i}^0=\mathbb{L}_{\widetilde{p}_i}$ and $\mathbb{L}_{\widetilde{p}_i}^\epsilon=0$ for $\epsilon>0$, then these stable filtered local systems of filtered degree zero are irreducible local systems. 

There is also an analogous definition of a filtered Higgs bundle, which is equivalent to the definition of a parabolic Higgs bundle $(E,E^\bullet_D, \Phi, \va)$ on $C$ we use in the paper.  At each $p_i\in D$, one simply combines the quasi-parabolic structure $E^\bullet_{p_i}$ (indexed by $j=1,\dots, \ell(i)$) and the parabolic weights $\underline{\a}_i$ in Definition \ref{def:parabolicHiggs} into a decreasing, left-continuous filtration (indexed by the parabolic weights) of the fiber $E_{p_i}$: 
\[ E_{p_i}=E^0_{p_i} \supset E^{\a_{i,1}}_{p_1}\supset \dots \supset E^{\a_{i,\ell(i)-1}}_{p_i} \supset E^{\a_{i,\ell(i)}}_{p_i}=0 , \quad \textrm{where  }E_{p_i}^{\a_{i,j}}=E^j_{p_i}.\]
Then $\Phi_i$ also preserves the filtration. Hence, we can equivalently work with parabolic Higgs bundles rather than filtered Higgs bundles.

At each $p_i$, both the filtered local systems and parabolic Higgs bundles contains the data of a filtered vector space and an endomorphism preserving the filtration.  
\begin{definition}
    Let $V$ be a finite-dimensional vector space over $\C$ equipped with a decreasing, left-continuous filtration $\bigcup_{\beta \in \mathbb{R}}V^\beta$, and let $T \colon V \to V$ be an endomorphism preserving the filtration i.e. $T(V^\beta) \subseteq V^\beta$ for all $\beta$. The residue of $(V, \bigcup_{\beta \in \mathbb{R}}V^\beta , T)$ is the graded vector space  
\[
\mathrm{res}(V) = \bigoplus_\beta \res(V)_\beta
\]  
where $\res(V)_\beta=V^\beta / V^{\beta+\epsilon}$ for small $\epsilon > 0$, together with the natural induced endomorphism $\mathrm{res}(T)$ acting on $\mathrm{res}(V)$.
\end{definition}
Now, we can associate to the residue of $(V,\bigcup_{\beta \in \mathbb{R}}V^\beta,T)$ a collection of partitions as follows. A conjugacy class $C_0\subset GL_r(\C)$ is determined by its Jordan normal form (JNF), we can identify $C_0$ with a collection of partitions $\{P^{\lambda}\}$ labeled by its eigenvalues $\lambda$, where each partition $P^\lambda= (n_1,\dots,n_\ell)$ records the sizes of Jordan blocks for eigenvalues $\lambda$. As the $\res(T)$ restricts to an endomorphism on $\res(V)_\beta$ for each $\beta$, we take $\{P^{\beta, \lambda}\}$ corresponding to the conjugacy class of this endomorphism on $\res(V)_\beta$. Then, we define the \textit{residue diagram} of  $(V,\bigcup_{\beta \in \mathbb{R}}V^\beta,T)$ to be the collection of partitions $\{P^{\b,\lambda}\}$ labeled by the jumps $\beta$ of the filtration and the eigenvalues $\lambda.$ 
\begin{exmp}\label{exmp:trivialfiltrations}
    When the filtrations of a filtered local system on $C\setminus \{p_1,\dots,p_n\}$ are trivial with $\beta=0$, then the residue diagrams $\{P^{0,\lambda}\}$ describe the conjugacy class of the local monodromy around $p_i$ of the underlying local system. 

    Similarly, if the parabolic Higgs bundle $(E,F^\bullet_D, \Phi, \vec{\a})$ has trivial filtrations i.e. $F^0_{p_i}=E_{p_i}, F^1_{p_i}=0$ and $\a_{i,1}=0$ for $i=1,\dots, n$, then the residue diagrams $\{P^{0,\xi}\}$ describe the conjugacy class of $\Phi_i$.
\end{exmp}
\begin{exmp}
     A $\vxi$-parabolic Higgs bundle  is the same as a parabolic Higgs bundle whose residue diagram at each $p_i$ is given by $P^{\a_{i,j},\xi_{i,j}}= (1,\dots,1)$ (repeated $m_{i,j}$ times). 
\end{exmp}

Now, we can state Simpson's tame NAHC \cite[Theorem, Page 718]{simpson-noncompact}: There is a one-to-one correspondence between stable filtered local systems on $C\setminus\{p_1,\dots,p_n\}$ of filtered degree zero and stable parabolic Higgs bundles on $C$ of parabolic degree zero. Moreover, at each $p_i\in D$, if we denote by $\{P^{\b,\lambda}\}$ and $\{P^{\a,\xi}\}$ the residue diagram for the filtered local system and parabolic Higgs bundle in correspondence, respectively, then the residue diagrams are the same, with the labels permuted according to the following table \cite[page 719]{simpson-noncompact}:
\begin{center}
\begin{table}[h]
    \centering
    \caption{Simpson's table}\label{table:simpson}
    \begin{tabular}{|c|c|c|}
        \hline
         & parabolic Higgs bundle & filtered local system \\ 
        \hline
        weight/jump& $\a$ & $\b=-2b$ \\  \hline
        eigenvalue & $\xi=b+\sqrt{-1}c$ & $\lambda= \exp(-2\pi \sqrt{-1} \a+ 4\pi c)$ \\ 
        \hline
    \end{tabular}
    \label{simpson's table}
\end{table}
\end{center}

Hence, under the tame NAHC, the goal of producing irreducible local systems whose local monodromies have prescribed conjugacy classes is equivalent to producing parabolic Higgs bundles with prescribed residue diagrams. Recall that given parabolic data $(\vm, \vxi)$, we have a collection of partitions $\{P^{\xi_{i,j}^\circ}\}$ labeled by the unrepeated eigenvalues that decomposes the partitions $\{P^i\}$; see Section~\ref{sec:notation} for notation.

\begin{prop}\label{conjugacy class}
Fix the parabolic data $(\vm, \vxi)$. Suppose $(E, \Phi)$ is a Higgs bundle obtained from a line bundle $L$ on an integral curve $\Sigma\subset Z_{\vxi}$ corresponding to a member in $\bmxi$ via pushing forward $L$ from $\sm$. Then 
        \begin{enumerate}
            \item \label{conjugacy1-1} $(E, \Phi)$ has no Higgs subbundle.
            \item \label{conjugacy1-2}If we equip $(E, \Phi)$ with the trivial filtrations $(E,F^\bullet_D, \Phi, \vec{\a})$, then the residue diagram $\{P^{0, \xi^\circ_{i,j}}\} $ of $(E, F^\bullet_D, \Phi, \vec{\a})$ is given by the conjugate partition of $P^{\xi_{i,j}^\circ}$
            \[P^{0, \xi^\circ_{i,j}} = \widehat{P}^{\xi_{i,j}^\circ}.\]
        \end{enumerate}

    \end{prop}
    \begin{proof}
        Let $f_{\vxi}:Z_{\vxi}\to C$ be the composition of successive blow-ups and
         $q: \Sigma\hookrightarrow Z_{\vxi}\to M \to C$ be the projection map. Then $E = q_*L$ and $\Phi$ is obtained from the tautological section. Part (1) follows from the integrality of $\Sigma$. 
        
        For part (2), since the filtrations are trivial,  each partition $P^{0,\xi^\circ_{i,j}}$ describes the Jordan normal form of $\Phi_i$ restricted to the generalized eigenspace of the eigenvalue $\xi^\circ_{i,j}$. Since $\Sigma$ intersects the fiber $f_{\vxi}^{-1}(p_i)$ at the exceptional divisors of the blow-ups centered at the points $(p_i,\xi^\circ_{i,j})$, the restriction of $\Phi_i$ to its generalized eigenspaces of $\xi^\circ_{i,j}$ can be determined locally in terms of the restriction of $L$ to the fiber over $\xi^\circ_{i,j}$. By doing a translation $\Phi - \xi^\circ_{i,j}$, we can further reduce the local analysis to the nilpotent case i.e $\xi^\circ_{i,j}=0$. By applying Proposition \ref{prop:localconjugate}, the condition of $\bmxi$ implies that the Jordan normal form corresponding to eigenvalue $\xi^\circ_{i,j}$  is the conjugate of $P^{\xi_{i,j}^\circ}$, so $P^{0, \xi^\circ_{i,j}} = \widehat{P}^{\xi_{i,j}^\circ}$. 

     \end{proof}

Before stating the main theorem, we define the relevant genericity assumptions which are useful to guarantee the existence of integral curves and irreduciblity of local systems.  
\begin{definition}\label{def:multiplicative generic}
    We say that a collection of eigenvalues $\{\lambda_{i,j}|i=1, \cdots, n \text{  and }j=1, \cdots, r\}$ is:
    \begin{itemize}
        \item 
     Multiplicatively generic if for any non-empty subsets $\Lambda_1,\dots, \Lambda_n\subset \{1,\dots, r\}$ with $|\Lambda_1|=\dots = |\Lambda_n| = m<n$, we have  
    \[
    \prod_{i=1}^n\prod_{k \in \Lambda_i} \lambda_{i,k} \neq 1.
    \]
    \item Additively generic if for the same choice of subsets, we have
    \[
    \sum_{i=1}^n\sum_{k \in \Lambda_i} \lambda_{i,k} \neq 1.
    \]
    \end{itemize}
    \end{definition}
    For each $\vxi= (\uxi_1,\dots,\uxi_n)\in \nm$, the vector $\uxi_i= ( \xi_{i,1},\dots, \xi_{i,\ell(i)})$ together with its associated multiplicities $\um_i= (m_{i,1},\dots, m_{i,\ell(i)})$ determines a vector of length $r$ by repeating each entry $\xi_{i,j}$ exactly $m_{i,j}$ times. Then there is an open subset $\nm_{add}\subset \nm$ consisting of additively generic eigenvalues.
\begin{lem}\label{lem:integral curves(generic)}
Suppose that the $OK$ condition holds i.e. 
\begin{equation}\label{eq:controllability for generic}
\sum_{i=1}^n \gamma_{P^i}(\mu) < (n-2)\mu + 2    
\end{equation}
holds for $\mu=2,\dots, r$. Then for a general choice of $\vxi\in \nm_{add}$, there exists an integral curve $\sm\in \bmxi$ contained in $Z_{\vxi}$. 
\end{lem}
\begin{proof}
Let $s\in \bmxi\subset A$ and denote $C_s\subset \Tot(K_C(D))$ the corresponding spectral curve. Then as in the proof of Lemma \ref{lem:integral at zero}, we have the identification 
    \[C_s\longleftrightarrow \sm:= p^* C_s- \sum_{i=1}^n\sum_{j=1}^{\ell(i)}m_{i,j} \Xi_{i,j,\vxi}\subset Z_{\vxi}.\]
Since $\vxi$ is additively generic, this spectral curve $C_s$ is integral.  As in the proof of Lemma \ref{lem:integral at zero}, to conclude that $\sm$ is integral, it suffices to check that $C_s$ and its total transform in the blow-ups have multiplicities exactly $m_{i,j}$ through the blow-up centers, so that $\sm$ is the strict transform of $C_s$. By Remark \ref{rem:local condition(minimal)}, we need to guarantee that the extra equations parametrized by $G(P^{\xi_{i,j}^\circ})_{\mathrm{min}}$ are not satisfied for each eigenvalue $\xi_{i,j}^\circ$. If $\dim(\bmxi)>0$, then one can choose a general element in $\bmxi$ to achieve this. When $\dim(\bmxi)=0$, the condition that the extra equation is satisfied defines a closed subset of $\nm_{add}$, so we can choose a general $\vxi$ in $\nm_{add}$ so that this condition fails, as desired. 

\end{proof}

Let $C_i\subset GL_r(\C)$ be a conjugacy class. Identify $C_i$ with a collection of partitions $\{P^{\lambda_{i,1}}, \dots, P^{\lambda_{i,e(i)}}\}$, labeled by the eigenvalues $\lambda_{i,j}$ of $C_i$. Then we define the following partition of $r$
\begin{equation}\label{e:conj partition}
    P^i := \widehat{P}^{\lambda_{i,1}} \cup \dots \cup \widehat{P}^{\lambda_{i,e(i)}}
\end{equation}
where $\widehat{P}^{\lambda_{i,j}}$ is the conjugate partition of $P^{\lambda_{i,j}}$.

\begin{theorem}\label{thm:DSP2}
Let $n\geq 3$. Let $C_1, \dots, C_n \subset GL_r(\mathbb{C})$ be a collection of conjugacy classes whose collection of eigenvalues is multiplicatively generic. 
 
    Suppose that the following conditions hold: 
    \begin{enumerate}
    \item  $\prod_{i=1}^n\det(C_i)=1$. 
        \item $\sum_{i=1}^n \gamma_{P^i}(\mu) <(n-2)\mu+ 2$ for $\mu=2,\dots, r$.
    \end{enumerate}
    Then the DSP is solvable for the tuple of conjugacy classes $(C_1,\dots,C_n)$. 
\end{theorem}
\begin{proof}
\textbf{Step 1}: Reduction to a stable parabolic Higgs bundle with prescribed residue diagrams.

        By the preceding discussions, our goal is to find an irreducible local system whose residue diagram is $\{P^{0,\lambda_{i,j}}\}$ where $P^{0,\lambda_{i,j}} = P^{\lambda_{i,j}}$. According to the tame NAHC and Simpson's table \ref{table:simpson}, this amounts to finding a stable parabolic Higgs bundles of parabolic degree zero whose residue diagram is $\{P^{\a_{i,j},\tau_{i,j}}\}$ where 
\begin{equation}\label{prescribed residue diagram}
P^{\a_{i,j},\tau_{i,j}} = P^{0,\lambda_{i,j}}  \quad \textrm{and }\quad \a_{i,j} = -\mathrm{arg}(\lambda_{i,j}), \tau_{i,j} = \sqrt{-1}\log |\lambda_{i,j}|/4\pi.  
\end{equation}
For each fixed $i$, in order to avoid coincidence of $\tau_{i,j}$ for distinct $\lambda_{i,j}$, we introduce $\beta_{i,j}$ to each $\lambda_{i,j}$ such that the modified values $\tau'_{i,j} = - \beta_{i,j}/2+\sqrt{-1}\log|\lambda_{i,j}|/4\pi$ are distinct for distinct $\lambda_{i,j}$ and $\sum \tau'_{i,j} =0$. Let $\vec{\beta} = (\underline{\b}_1,\dots,\underline{\b}_n)$ be such a choice of $\beta_{i,j}$ where $\underline{\b}_i=(\beta_{i,j}).$ Under the multiplicative genericity assumption on $\lambda_{i,j}$, every local system with the prescribed conjugacy classes is irreducible. In particular, if we can produce a $\vec{\beta}$-stable filtered local system with the residue diagram 
\[P^{\beta_{i,j}, \lambda_{i,j}}= P^{\lambda_{i,j}},\] 
then the underlying local system $\bb{L}$ is also irreducible. Moreover, the conjugacy class of $\bb{L}$ is also determined by  $\{P^{\beta_{i,j}, \lambda_{i,j}}\}$ since we can assume that, for each fixed $i$, the weights $\beta_{i,j}$ are distinct for distinct $\lambda_{i,j}$ which ensures that the generalized eigenspace corresponding to the eigenvalue $\lambda_{i,j}$ is isomorphic to the associated graded piece labeled by $\beta_{i,j}$ under the natural linear map. To construct such a $\vec{\b}$-stable filtered local system via the tame NAHC, it is then equivalent to construct a stable parabolic Higgs bundles of parabolic degree zero whose residue diagram is $\{P^{\a_{i,j},\tau'_{i,j}}\}$ where
        \begin{equation*}
P^{\a_{i,j},\tau'_{i,j}} = P^{\beta_{i,j},\lambda_{i,j}} \quad \textrm{and  }\quad \a_{i,j} = -\mathrm{arg}(\lambda_{i,j}), \tau_{i,j}' = -\beta_{i,j}/2+ \sqrt{-1}\log |\lambda_{i,j}|/4\pi.
        \end{equation*}
Let us denote $R^{i,j}:=P^{\a_{i,j},\tau'_{i,j}}= P^{\lambda_{i,j}}$ for simplicity.
        
\vspace{0.5em}
        \noindent\textbf{Step 2}: Existence of a parabolic Higgs bundle realizing the desired residue diagrams.
        
        For each $i=1,\dots, n$, we can always choose a pair $(\um_i,\uxi_i)$ such that $\xi^\circ_{i,j}=\tau'_{i,j} $ and the associated collection of partitions $\{P^{\xi^\circ_{i,j}}\}_{j=1}^{e(i)}$ of the pair satisfies $P^{\xi_{i,j}^\circ} = \widehat{R}^{i,j}$. By construction, the union of these partitions is
        \[ P^{\xi^\circ_{i,1}} \cup \dots \cup P^{\xi^\circ_{i,e(i)}}  = \widehat{R}^{i,1}\cup \dots \cup \widehat{R}^{i,e(i)} = P^i, \]
         then the two conditions in Theorem \ref{thm:DSP2} guarantees that the hypotheses in Corollary \ref{thm:nonemptiness of Hitchin bases} are satisfied so that $\bmxi$ is non-empty. Moreover, Lemma \ref{lem:integral curves(generic)} ensures the existence of an integral curve in $\bmxi$ because the multiplicative genericity of $\lambda_{i,j}$ converts to the additive genericity of $\vxi$ and we can also perturb $\beta_{i,j}$ to make $\vxi$ a general element in $\nm_{add}$ if necessary. 
         
        Let $L$ be a line bundle on $\Sigma$ and $(E,\Phi)$ be the corresponding Higgs bundle under the spectral correspondence. By Proposition \ref{conjugacy class}, the residue diagrams of the Higgs bundle (treated as a parabolic Higgs bundle with trivial filtrations) are given by 
\[P^{0, \xi_{i,j}^\circ} = \widehat{P}^{\xi^\circ_{i,j}} = R^{i,j}.\]
In order to obtain the desired residue diagrams $P^{\a_{i,j},\tau'_{i,j}}$, we need to construct a new filtration for the Higgs bundle $(E,\Phi)$. Without loss of generality, assume that $\Phi_i$ is in its Jordan normal form with respect to a suitable basis. As the Jordan blocks corresponding to each eigenvalue $\xi_{i,j}^\circ= \tau'_{i,j}$ are described by $R^{i,j}$, we can explicitly construct a filtration $ E^\bullet_{p_i}$ of $E_{p_i}$ indexed by the parabolic weights $\a_{i,j}$ by choosing subsets of the basis vectors such that the residue diagram with respect to $E^\bullet_{p_i}$ is exactly $\{P^{\a_{i,j}, \tau'_{i,j}}\}$. This produces a parabolic Higgs bundle $(E, E^\bullet_D, \Phi, \va)$ with the desired residue diagram \eqref{prescribed residue diagram}. 

\noindent\textbf{Step 3}: Stability and parabolic degree zero conditions. 

It remains to check that $(E, E^\bullet_D,\Phi,\va)$ is stable and has parabolic degree $0$. Stability follows from the integrality of $\Sigma$ and Proposition \ref{conjugacy class} which guarantees that $(E, \Phi)$ has no Higgs subbundles. Therefore, $(E, E^\bullet_D, \Phi, \va)$ is automatically stable as a parabolic Higgs bundle. Note that the first condition is clearly necessary for the existence of solution to DSP. According to Simpson's table, the second condition converts to the conditions on the Higgs bundles side that $\sum \tau_{i,j}=0$ and the sum of $\a_{i,j}$ (counted with multiplicities) is an integer. Recall that the parabolic degree is defined as 
\[\deg(E) + \a_{i,j}\sum \dim( E_{p_i}^{j-1}/E_{p_i}^j).\]
As $\chi(L) = \chi(E)$, the degree of $E$ can be any value by choosing $L$ to have an appropriate degree. Since $\a_{i,j}\sum \dim( E_{p_i}^{j-1}/E_{p_i}^j)$ is an integer by the first condition, we can always select $\deg(L)$ such that the parabolic degree is zero.

\end{proof}
\begin{rem}
In Appendix \ref{comparison result}, we will show that this criterion is numerically equivalent to Simpson's criterion when one of the conjugacy classes has distinct eigenvalues.     
\end{rem}

\begin{rem}[Conjecture of Balasubramanian--Distler--Donagi]\label{rem:bdd conjecture}
        In the setup of Theorem \ref{thm:DSP2}, let $\vm = (\um_1,\dots, \um_n)$ where $\um_i= P^i$. The inequalities in condition (2) are equivalent to the $OK$ condition for the collection of line bundles $L_{\vm}(\mu)$ on $C=\P^1$ for $\mu=2,\dots,r$. Therefore, Theorem \ref{thm:DSP2} confirms the conjecture proposed in \cite{baladistlerdonagi}:   the $OK$ condition implies that multiplicative DSP is solvable for the tuple $(C_1,\dots,C_n)$ under the assumption that $(C_1,\dots,C_n)$ is multiplicatively generic.

\end{rem}

The same idea works for the case of higher genus curves. 
\begin{theorem}\label{thm:DSP higher genus}
Let $n\geq 1$ and $g>0$. 
Let $C_1, \dots, C_n \subset GL_r(\mathbb{C})$ be conjugacy classes whose collection of eigenvalues is multiplicatively generic. Suppose that $\prod_{i=1}^n\det(C_i)=1$ holds. For $g=1$, assume that at least one of the partitions $P^i$ is not the singleton partitions. Then there exist matrices $A_k,B_k, T_i\in GL_r(\C)$ with $T_i\in C_i$ for $k=1,\dots, g, i=1,\dots, n$ such that 
    \[ \prod_{k=1}^g (A_k,B_k)\prod_{i=1}^n T_i = Id_r\]
    where $(A_k,B_k) = A_kB_kA_k^{-1}B_k^{-1}$ and the matrices $A_k,B_k,T_i$ have no common invariant subspace. 
    
\end{theorem}
\begin{proof}
    The existence of the matrices $A_k,B_k, T_i$ is equivalent to the existence of irreducible local system on a genus $g>0$ Riemann surface $C\setminus \{p_1,\dots,p_n\}$ whose residue diagram is determined by the prescribed conjugacy classes $C_i$. We can proceed as in the proof of Theorem \ref{thm:DSP2}. 
\end{proof}

\begin{rem}
    The same strategy can be adapted to obtain a sufficient condition for the DSP (and its higher genus analogue) without the multiplicatively genericity assumption. The key modification lies in ensuring the existence of integral curves, which can be guaranteed by Lemma~\ref{lem:integral at zero} and Proposition~\ref{deformation of curves}. However, this requires the inequalities in Theorem~\ref{thm:DSP2} to be slightly strengthened. For $g\geq 2$, it can be shown that Theorem~\ref{thm:DSP higher genus}  remains valid even when the multiplicatively genericity condition is dropped. 
\end{rem}

\begin{rem}[Comparison with Kostov's criterion for DSP]\label{rem: kostov}
For multiplicatively generic eigenvalues, Kostov provided a sufficient and necessary condition for DSP \cite{kostovsurvey}. Formulated in terms of the notion of "defect" defined by Simpson in \cite[Section 2.7]{Simpson-middleconvolution}, Kostov's criterion takes the form of an iterative algorithm where the defect controls both the termination condition and the modification of conjugacy classes at each step. Roughly, Kostov's algorithm goes by iteratively verifying the inequalities (analogous to Simpson's criterion) for a given set of conjugacy classes; if the defect is negative, it modifies the conjugacy classes and repeats the verification until termination. 

Analogous to Simpson's definition of defect, we define 
\[ \delta(\mu) =  (n-2) r - \sum_{i=1}^n \gamma_{P^i}(\mu) \]
where $P^i$'s are defined in \eqref{e:conj partition}. This differs from  Simpson's original definition by incorporating conjugate partitions, as required by Proposition \ref{conjugacy class} which involves conjugate partitions. Then the inequalities in the second condition of Theorem \ref{thm:DSP2} says that $\delta(\mu)>-2$ for $\mu=2, \dots, r$. This implies that the cases covered by Theorem \ref{thm:DSP2} amounts to the situation where the negative defects must be $-1$ in Kostov's algorithm. The cases with higher defects in Kostov's algorithm will be taken up in future work. 

\end{rem}

\appendix
\section{Local behaviors}\label{sec:local behaviors}
In this section, we study the local conditions imposed by the linear system $|\sm_{\vm}(\vxi)|$ by explicitly writing down the local charts of the successive blow-ups. 

Let $M_0=\bb{A}^2$ with coordinates $(x,y)$. Alternatively, one may take $M_0= \Spec(\C[[x]][y])$ in the context of ruled surface, the computations and results below remain unchanged in either setting. Let $l_x := \{x=0\}$ in $M_0$ and $M_1$ be the blow-up of $M_0$ at the origin $c_1:=(0,0)$. Let $M_\ell\to \dots\to  M_1\to M_0$ be a sequence of blow-ups where $p^{j-1}_{j}: M_{j}\to M_{j-1}$ is the blow-up of $M_j$ at the intersection $c_j$ of the strict transform of $l_x$ and the exceptional divisor from the previous blow-up. Let $\um = P = (m_1, \dots, m_\ell)$ a partition of $r$ and let $(n_1,\dots,n_s)$ be the conjugate partition of $P$. 
    \begin{lem}\label{lem:localconditions}
        Let $X$ be a plane curve in $M_0$, defined by the equation $F(x,y)= y^r+ s_1(x)y^{r-1}+\dots+ s_r(x)=0$ where $s_\mu(x)\in \C[x]$. Let $M_\ell\to \dots\to M_1\to M_0$ be as above. Then the following conditions are equivalent
    \begin{enumerate}
            \item $X$ passes through $c_1$ with multiplicity at least $m_1$ and, for $j= 2,\dots, \ell$, the total transform of $X$ in $M_{j-1}$ passes through $c_j$ with multiplicity at least $m_1+ \dots+ m_j$;
                \item Let $v(\mu)$ be the vanishing order of $s_\mu(x)$ at $x=0$ for $1 \leq \mu\leq \ell$ and $\gamma(\mu)$ the level function associated to the partition $P$. Then $v(\mu)\geq \gamma(\mu)$ for all $1\leq \mu\leq \ell.$
                \item The partial derivative  $\left. \frac{\partial^u}{\partial y^u}\frac{\partial^a}{\partial x^a} \right\vert_{(0,0)}F(x,y)$ vanishes for $(u,a)\in G(P):= \{(u,a)| 0\leq a<\gamma({r-u})\}$. 
            \end{enumerate}
            
        \end{lem}

         \begin{proof}
            $(1)\implies(2)$.  
            Let us fix $1\leq \mu\leq r$ and focus on the term $T_\mu= s_{\mu}(x)y^{r-\mu}$. Suppose $X$ must pass through $c_1=(0,0)$ with multiplicity at least $m_1$. If $r-\mu\leq m_1$, then $s_\mu(x)$ must vanish at $x=0$ with multiplicity at least $d_1= m_1-(r-\mu)$ in which case we can write $s_\mu =s_\mu^1(x)x^{d_1}$. For $r-\mu>m_1$, we can still write $s_\mu =s_\mu^{1}(x)x^{d_1}$ with $d_1=0$. Now, if we blow up at the origin to get $M_1$, we have $x=u_1y$ (only this chart matters) and the term in the total transform $X_1\subset M_1$ becomes 
            \[T^1_\mu= s_\mu^{1}(u_1y) u_1^{d_1}y^{r-\mu+d_1}.\] 
            Since $X$ has multiplicity at least $m_1$ at $(0,0)$, the exceptional divisor must have at least multiplicity $m_1$, and so 
            \[r-\mu+d_1\geq m_1.\]
            Now, suppose $X_1$ passes through $c_2:=(u_1,y)=(0,0)$ with multiplicity at least $m_1+ m_2$. Following the same reasoning as before, the term $T_\mu^1$ in the total transform $X_2\subset M_2$ transforms into 
            \[T_\mu^2:= s^{2}_\mu(u_2y^2) u_2^{d_1+d_2} y^{r-\mu+2d_1+2d_2-m_1-m_2}\]
            for some $d_2\geq 0$ where we let $u_1=u_2y$. Moreover, the inequality 
            \[r-\mu+2d_1+2d_2\geq  m_1+m_2\] holds. Proceed inductively, for $j=1,\dots, l$, the term in the total transform after the $j$-th blow-ups becomes
    \[ T_\mu^j  := s_\mu^{j} (u_jy^j)u_j^{\Delta_\mu^j}y^{\nabla_j}, \quad \textrm{where}\quad  \Delta_\mu^j= \sum_{a=1}^j d_a ,\quad \nabla_j = r-\mu +j\sum_{a=1}^jd_a \quad\]
    for some $d_1,\dots, d_j\geq 0$. We also have that the inequality 
    \[\nabla_j\geq \sum_{a=1}^jm_a \]
    holds. Our goal is to give a lower bound of $\Delta_\mu^l$ in the last ($l$-th) blow-up since $s_\mu(x)= s_\mu^{l}(x)x^{\Delta_\mu^l}.$ Note that each inequality $\nabla_j\geq \sum_{a=1}^jm_a$ implies 
    \begin{equation}\label{inequalities}
        (\ast_j) \qquad \Delta_\mu^l\geq \Delta_\mu^j\geq \frac{1}{j}\left( \sum_{a=1}^j m_a -(r-\mu)\right) \quad \textrm{for }j=1,\dots, l.
    \end{equation}
    Denote by $B_\mu$ the lower bound $\frac{1}{j}\left( \sum_{a=1}^j m_a -(r-\mu)\right)$. 
    Let $n_{1}\geq \dots \geq n_{m_1}$ be the conjugate partition. Since we would like to relate the lower bound with the level function, we also view $1\leq \mu\leq r$ as the integers in the filling (1) of the Young diagram introduced in Section \ref{Young diagram}. Let $m_{j+1}+1\leq t \leq m_j$ be the index of the columns where $j=1,\dots, l$ and $m_{l+1}=0$. Then, for $\mu= \sum_{b=1}^{t}n_b$, by counting the number of boxes, the inequality $(\ast_j)$ becomes
    \begin{equation*}
        \Delta_\mu^l\geq \frac{1}{j}\left( \sum_{a=1}^j m_a -\sum^{m_1}_{t+1}n_b \right) = \frac{1}{j}(jt)  =t. 
    \end{equation*}
    By the definition of level function, $t=\gamma(\mu)$ for $\mu= \sum^t_{b=1}n_b$. This yields the desired lower bound for the boxes at the bottom of each column in the Young diagram. Since $B_\mu$ is strictly increasing in $\mu$ and the value of $B_\mu$ between the two boxes at the bottom of two adjacent columns differ exactly by $1$, all the boxes in the same column will have the value of $B_\mu$ which is precisely $\gamma(\mu)$. Hence, each $s_\mu(x)$ will have vanishing order at least $\gamma(\mu)$ at $x=0$. 

    \noindent $(2)\implies (1)$. Suppose the vanishing order of $s_\mu(x)$ is at least $\gamma(\mu)$, so $s_\mu(x)= s'_\mu(x)x^{\gamma(\mu)}$ and $F(x,y) = y^r+ s'_1 x^{\gamma(1)} y^{r-1} + \dots + s'_rx^{\gamma(r)}$. To check that $X$ passes through $(0,0)$ with multiplicity at least $m_1$, it suffices to show that the minimum of the sum of the exponents of each term is $m_1$ i.e. $\textrm{min}_{1\leq \mu\leq r}\{\gamma(\mu) + r-\mu\}= m_1$. Note that when $\mu=r$, then $\gamma(r)+r-r= m_1$ by definition. It can be checked directly that $m_1$ is indeed the minimum.  Proceed as in the first part, we see that the general expression for the pullback after the $j$-th blow-ups is 
    \begin{equation*}
        y^{r} +\sum_{\mu=1}^r s_\mu u_j^{\gamma(\mu)}y^{(j-1)\gamma(\mu)+ r-\mu }
    \end{equation*}
    Similarly, in order to say that the pullback $X_j$ of $X$ passes through $c_j$ with multiplicity at least $m_1+\dots+ m_j$, it suffices to show that
    \begin{equation}\label{minimum}
    \min_{1\leq \mu\leq r}\{j\gamma(\mu) +r-\mu \}= \sum^{j}_{a=1}m_a.
    \end{equation}
    Let $C_\mu:= j\gamma(\mu) +r-\mu$.
    We claim that the minimum is attained at the bottom of the $m_j$-th column of the Young diagram where $\mu_{\mathrm{min}}=\sum_{b=1}^{m_j} n_{b}$ and $\gamma(\mu_{\mathrm{min}}) = m_j$. As 
    \[j\gamma(\mu_{\mathrm{min}}) - \sum^{j}_{a=1} m_a =- \sum_{b= m_j+1}^{m_1} n_b = -(r-\mu_{\mathrm{min}}), \]
    it follows that $C_{\mu_{\mathrm{min}}} = \sum_{a=1}^jm_a$. 
    To check that $\mu_{\mathrm{min}}:=\sum_{b=1}^{m_j} n_{b}$ attains the minimum value, we write $\mu= \mu_{\mathrm{min}}+\delta$ where $1 \leq \mu_{\mathrm{min}}+\delta\leq r$. Note that if we write $g= \gamma({\mu_{\mathrm{min}}+\delta}) -\gamma(\mu_{\mathrm{min}})$, then
    \begin{align*}
    C_\mu-\sum_{a=1}^jm_a = j\gamma(\mu_{\mathrm{min}}+\delta)+r-(\mu_{\mathrm{min}}+\delta) -\sum^j_{a=1}m_a     
    = jg -\delta. 
    \end{align*}
    Note that $g$ and $\delta$ have the same sign but $g\leq\delta$. As $j=n_{m_j}$, it is easy to see that $jg-\delta\geq 0$. Hence, the equality \label{minimum} holds. 

        \noindent  $(2)\iff (3)$ Obvious. 
        \end{proof}

\begin{rem} \label{rem:local condition(minimal)}
Let $J_{\mathrm{min}}= \{\mu|\mu= \sum_{b=1}^{m_j} n_b$, for $j = 1, \dots, \ell \}$ be the set of minimal indices in the proof of Lemma \ref{lem:localconditions}. The implication (2) $\Rightarrow$ (1) further shows that if equalities are attained for the minimal indices in condition (2)—that is, if $v(\mu) = \gamma(\mu)$ for $\mu\in J_{\mathrm{min}}$—then the multiplicities are exact in condition (1): the curve $X$ passes through $c_1$ with multiplicity exactly $m_1$, and for each $j = 2, \dots, \ell$, the total transform of $X$ in $M_{j-1}$ passes through $c_j$ with multiplicity exactly $m_1 + \dots + m_j$.

 Let $G(P)_{\mathrm{min}}= \{(u,a)|u\in J_{\mathrm{min}}, a= \gamma(r-u) \}$. Then the additional assumption on the minimal indices in condition (2) translates into condition (3) by requiring that the partial derivative  $\left. \frac{\partial^u}{\partial y^u}\frac{\partial^a}{\partial x^a} \right\vert_{(0,0)}F(x,y) \neq 0$ for $(u,a)\in G(P)_{\mathrm{min}}$. Therefore, the non-vanishing of these partial derivatives at $G(P)_{\mathrm{min}}$ forces equality for minimal indices in condition (2), and thus exact multiplicity in condition (1).
\end{rem}

    Continuing with the notation and setup from Lemma 6.1. Consider the projection $\pi: M_0\to \bb{A}^1= \textrm{Spec}(\C[x])$ and denote by $p:M_\ell\to M_0$ and $f:M_\ell\to M_0\to \bb{A}^1$ the composition. Let $E_j\subset M_\ell$ be the strict transform of the exceptional divisor of the blow-up $M_j\to M_{j-1}$ under $M_\ell\to M_j$. Let $\Xi_j = E_j+\dots + E_\ell$. 

For any $\O_{\Sigma_0}$-module $F$, multiplication by $y$ induces an $\O_{M_0}$-module homomorphism $y:p_*F\to p_*F$. Pushing forward along $\pi$, we obtain a Higgs bundle $(f_*F,\Phi)$ where the Higgs field is given by 
\[\Phi:= \pi_*y:f_*F\to f_*F.\]
 This serves as the local model for the blow-ups of spectral curves, allowing us to analyze the local behavior of the induced Higgs field. 

    \begin{prop}\label{prop:localconjugate}
    Suppose that $X$ satisfies one of the conditions in Lemma \ref{lem:localconditions} such that the divisor $\Sigma= p^*X-m_1\Xi_1-\dots -m_\ell \Xi_\ell$ is effective. Assume that $\Sigma$ is integral. Let $ L$ be a line bundle on $\Sigma$. Then the Jordan normal form with eigenvalue $0$ of $\Phi$ restricted to $0\in \bb{A}^1$ is uniquely determined by the partition $(m_1, \dots, m_\ell)$. More precisely, if we denote by $(n_1, \dots ,n_s)$ the partition corresponding to the Jordan normal form where each $n_j$ represents a Jordan block of size $n_j$, then $(n_1,\dots , n_s)$ is given by the conjugate partition of $(m_1, \dots , m_\ell)$. 
    \end{prop}
    \begin{rem}\label{rem:JNF}
        Recall that in linear algebra, given a vector space $V$ of dimension $r$ and a nilpotent operator $B:V\to V$, the Jordan normal form is determined by the sequence of subspaces:
        \[  0\subset W_1\subset \dots \subset W_\ell=V ,\quad \textrm{where }W_j := \ker(B^j)\]
        for some $\ell$. Let $b_j= \dim(W_j/W_{j-1})$ for $j=1,\dots, \ell$ and $u_j$ the number of Jordan blocks of size $j$, then $u_j$ can be computed as follows: 
        \begin{align*}
            u_1+u_2+\dots + u_\ell &= b_1\\
            u_2+\dots + u_\ell &= b_2\\
            \vdots \\
            u_\ell&= b_\ell
        \end{align*}
        Note that $u_1+2u_2+\dots+ \ell u_\ell=r = b_1+ b_2+\dots+ b_\ell$ and the following partition of $r$ is exactly the conjugate partition of $r$        \begin{equation}\label{conjugate partition}
          \underbrace{\ell,\dots,\ell}_{u_\ell-\textrm{copies}},\dots, \underbrace{1,\dots, 1}_{u_1-\textrm{copies}} 
        \end{equation}    
        Therefore, the Jordan normal form of $B$ in the form of partition is given by \eqref{conjugate partition} which is the conjugate partition of $(b_1, \dots,b_\ell)$.

        Note that this resembles the intersection numbers of the curve classes $\sm_{\vm}(\vxi)$ studied in Proposition \ref{p:class}: 
        \begin{align*}
            E_1\cdot \Sigma +E_2\cdot \Sigma +\dots + E_\ell\cdot \Sigma &= \Xi_1\cdot \Sigma = m_1\\
            E_2\cdot \Sigma+\dots + E_\ell\cdot \Sigma &= \Xi_2\cdot \Sigma =m_2\\
            \vdots \\
            E_\ell\cdot \Sigma &= \Xi_\ell\cdot \Sigma =m_\ell
        \end{align*}
        In fact, we will show that each $E_j\cdot \Sigma$ is realized as the number of Jordan blocks of size $j$. 

        \end{rem}
        \begin{proof}[Proof of Proposition \ref{prop:localconjugate}]        
        To analyze the local behavior of the Higgs field, we work chart by chart on the successive blow-ups $M_l$. The variety $M_l$ can be covered by $2\ell$ open subsets $U_j= \Spec R_j,V_j=\Spec R'_j$, $j=1, \dots, \ell$ where  
        \[R_j= \C[u_{j-1},v_{j},y]/(y-u_{j-1}v_j), \quad R'_j= \C[u_j,y,u_{j-1}]/(u_{j-1}-u_jy) \]
        where $u_0= x$. Let $F_j(u_{j-1},v_j)$ be the local equation of $\Sigma_\ell$ in $U_j.$ Since $L$ is locally free, we can represent the restriction of $L$ to $\Sigma_\ell\cap U_j$ as an $R_j$-module as 
        \[L_j= \frac{\C[u_{j-1},v_j]}{(F_j(u_{j-1},v_j))}.\]
        Let $N_0=L_j/(x)$ be the scheme-theoretic restriction of $L_j$ to the fiber over $x=0$. Then $f_*L|_{x=0}$ is locally represented by $N_0$ as a module over $\C[x]/(x)\cong \C$.  
        As $\sm$ is assumed to be integral, it does not contain any $E_j$ as an irreducible component, so $N_0$ is a finite dimensional vector space. The Higgs field $\Phi$ restricted to $f_*L|_{x=0}$ is induced by the action of multiplication by $y$ on $N_0.$ 
        
        Note that $U_j$ is obtained from blowing up at the origin in $V_{j-1}$ where we have the relation
        \[x = u_0= u_1y = \dots = u_{j-1}y^{j-1}. \]
        So, we also have the relation on $U_j$
        \[y^j = y^{j-1}\cdot (u_{j-1}v_j) = xv_j.\]
        This means that the action of multiplication by $y$ is nilpotent of order $j$ on $N_0$. Following the idea in Remark \ref{rem:JNF}, we would like to show that $N_0$ contains a vector subspace preserved by the operator $y$ such that the Jordan normal form of $y$ on it consists of exactly $e_j:=E_j\cdot \Sigma_\ell$ Jordan blocks of size $j$. As we run the construction over all $U_j$ for $j=1,\dots,\ell$, we get the desired Jordan normal form of $\Phi$ restricted to the whole $f_*L|_{x=0}$

        Consider $N_1=N_0/(u_j)\cong L_j/(u_{j-1})$ the restriction of $N_0$ to the $E_j\cap U_j$. For simplicity, we assume that $\sm_\ell$ does not pass through the intersection $E_j\cap E_{j+1}$ (otherwise, we will need to consider more charts). Since $L$ is a line bundle, the dimension of $N_1$ (as $\C$-vector space) is given exactly by $e_j$. 
        After dividing $F_j(u_{j-1},v_k)$ by its leading coefficient to make it monic, we can write 
        \begin{equation}\label{eq:local spectral curve}
            F_j(u_{j-1},v_j) = v_j^{e_j} +u_{j-1}f + c  
        \end{equation}
        where $c\in \C$ is non-zero and  $f$ is a polynomial in $v_j$ with coefficients in $\C[u_{j-1}]$ and degree $<e_j$. 
        Hence, $N_1$ is always spanned by the basis $1,v_j,\dots, v_j^{e_j-1}$.
         Now, consider the following set of vectors in $N_0$ by repeatedly applying $y$ to the basis vectors 
         $1,v_j,\dots, v_j^{e_j-1}$
         \begin{equation}\label{basis vectors}
         T= \{1,v_j,\dots, v_j^{e_j-1}, y,yv_j,\dots, yv_j^{e_j-1}, \dots, y^{j-1},y^{j-1}v_j,\dots, y^{j-1}v_j^{e_j-1}. \}
         \end{equation}
        We claim that this set of vectors is linearly independent. Suppose 
        \begin{equation}\label{linear expression}
        \sum_{a=0}^{j-1} \sum_{b=0}^{e_j-1} c_{ab} y^a v_j^b=\sum_{a=0}^{j-1} \sum_{b=0}^{e_j-1} c_{ab} u_{j-1}^a v_j^{a+b} =0 \in  N_0.
        \end{equation}
        The linear combination \eqref{linear expression} in $N_1$ is given by 
        \[ c_{00} + c_{01}v_j + \dots + c_{0(e_j-1)}v_j^{e_j-1}=0 \]
        which implies that $c_{00} = c_{01}= \dots = c_{0(e_j-1)}=0$ since $1,v_j,\dots, v_j^{e_j-1}$ is a basis in $N_1$. Since $y= u_{j-1}v_j$, the linear combination \eqref{linear expression} in $L_j/(u_{j-1}^2)$ is given by 
        \[ c_{10}y + c_{11}yv_j + \dots + c_{1(e_j-1)}yv_j^{e_j-1}= u_{j-1}(c_{10}v_j + c_{11}v_j^2 + \dots + c_{1(e_j-1)}v_j^{e_j})=  0\in L_j/(u_{j-1}^2) \]
        which implies that 
        \[ u_{j-1}(c_{10}v_j + c_{11}v_j^2 + \dots + c_{1(e_j-1)}v_j^{e_j}) = w u_{j-1}^2\in L_j\quad \textrm{for some }w\in L_j. \]
        Since the curve $\Sigma_\ell$ is assumed to be integral, $L_j$ is a torsion-free module over $R_j/(F_j(u_{j-1},v_j))$, so we obtain
        \[c_{10}v_j + c_{11}v_j^2 + \dots + c_{1(e_j-1)}v_j^{e_j} = w u_{j-1} \in L_j.\]
        Using the expression in \eqref{eq:local spectral curve}, we get 
        \[ c_{10}v_j + c_{11}v_j^2 + \dots + c_{1(e_j-2)}v_j^{e_j-1} - c_{1(e_j-1)}c = w'u_{j-1} \in L_j \quad \textrm{for some }w\in L_j. \]
        As the left-hand side is a polynomial in $v_j$ with degree $<e_j$ which is not divisble by $u_{j-1}$ in $L_j$, it must be zero. It follows that $c_{10}= c_{11}=\dots = c_{1(e_j-1)}c=0$, and hence $c_{1(e_j-1)}=0$. Proceed inductively for $L_j/(u_{j-1}^k)$, $k=3,\dots, e_{j-1}-1$, we see that all $c_{ab}=0.$ Then it follows from the linear independence of the set of vectors \eqref{basis vectors} that if we rearrange the order of the vectors in $T$ \eqref{basis vectors} and consider the subspace $\langle T\rangle \subset N_0$ they span, the operator $y$ restricted to $\langle T\rangle$ is in its Jordan normal form with exactly $e_j$ Jordan blocks of size $j$ with respect to the basis vectors in $T$.  
    \end{proof}

    \section{Comparison with Simpson's result}\label{comparison result}
    We shall now compare our result with Simpson's result \cite{simpson-productofmatrices} when one of the conjugacy class has distinct eigenvalues. 

    In the setup of DSP, we are  given a set of conjugacy classes $C_1,\dots,C_n$. 
    For each conjucacy class $C_i \in GL_r(\mathbb{C})$, we define
    \begin{itemize}
        \item $d_i$ is the dimension of $C_i$. Note that $d_i=r^2-\dim_\mathbb{C} Z(C_i)$ where $Z(C_i)$ is the group of centralizers of $C_i$.
        \item $R_i:=\min_{\lambda \in \mathbb{C}}\mathrm{rank}(Y-\lambda I)$ for a matrix $Y$ from $C_i$. Note that $r-R_i$ is the maximal number of Jordan blocks of $J(Y)$ with one and the same eigenvalue. 
    \end{itemize}
    \begin{theorem}\cite{simpson-productofmatrices}\label{thm:SimpsonDSP}
    For generic eigenvalues and when one of the conjugacy classes, say $C_{i_0}$, has distinct eigenvalues, then the DSP is solvable if and only if
    \begin{enumerate}
        \item $d_1+\cdots+d_n \geq  2r^2-2 $,
        
        \item $R_1+\cdots+\widehat{R_i}+\cdots R_n \geq r \quad \forall i $.
    \end{enumerate}
        
    \end{theorem}
    The folllowing numerical equivalence is observed in \cite{baladistlerdonagi}.\begin{prop}\label{prop:nodefect condition}
    When one of the conjugacy classes, say $C_{i_0}$, has distinct eigenvalues. Then the following inequalities 
        \begin{equation}
        \begin{aligned}
          &(\a)\quad   d_1+\cdots+d_n \geq  2r^2-2, \\
           & (\beta) \quad R_1+\cdots+\widehat{R_i}+\cdots +R_n \geq r \quad \forall i,
            \end{aligned}
            \label{simpson criterion}
        \end{equation}
    hold if and only if 
\begin{equation}\label{controllable criterion}
     \sum_{i=1}^n \gamma_{P^i}(\mu) <(n-2)\mu+ 2\quad  \text{for  } \mu=2,\dots, r
     \end{equation}
     where $P^i$'s are defined in \eqref{e:conj partition}.
    \end{prop}
    
\begin{proof}
(\eqref{controllable criterion} $\implies$ \eqref{simpson criterion}) 
    Suppose $C_{i_0}$ has distinct eigenvalues, which means that each $P^{\lambda_{i_0,j}}$ is the singleton partition of $1$, and the partition $P^{i_0}=(\underline{1})$. Moreover, the maximal number of Jordan blocks corresponding to the same eigenvalue $r-R_i$ is simply the number of columns of $P^{i}$, or equivalently, $\gamma_{P^{i}}(r)$.

    Since $\gamma_{P^{i_0}}(r)=1$, we have $R(p_{i_0})= r-1$ which is the maximal value for $R_i$, so the second condition in Theorem \ref{thm:SimpsonDSP} is equivalent to the single inequality
    \[R_1+\cdots+\widehat{R_{i_0}}+\cdots R_n \geq r.\]
    Using the identification $r- R_i = \gamma_{P^i}(r)$ and $\gamma_{P^{i_0}}=1$, we can write it as 
    \begin{equation*}
    (n-1)r - \sum ^n_{i=1}\gamma_{P^i}(r) + \gamma_{P^{i_0}}(r) \geq r \iff \sum_{i=1}^n\gamma_{P^{i}}(r)<(n-2)r+2.  
    \end{equation*}

    Note that $d_i$ can be computed in terms of the partition $P^i$ (e.g. using \cite[Remark 14]{kostovsurvey}),  we have $d_i = r(r+1)- 2\sum_{\mu=1}^r \gamma_{P^i}(\mu)$. Then the inequality $(\a)$ is simply 
    \[ nr(r+1) -2\sum_{i=1}^n \sum_{\mu=1}^r \gamma_{P^i}(\mu) \geq 2r^2-2. \]
    A direct computation shows that this is a consequence of summing up all the inequalities in \eqref{controllable criterion} for $\mu=2,\dots,r $ and the equality $\sum_{i=1}^n\gamma_{P^i}=n$ on both sides. 

\noindent(\eqref{simpson criterion} $\implies$ \eqref{controllable criterion}) For this direction, we will make use a property of level functions. 
    \begin{lem}\label{lem:level function}
        For any partition $P$ and any $2\leq \mu\leq r$, if  $\gamma_P(\mu-1) = \gamma_P(\mu)-1$ and $\gamma_P(\mu-2) = \gamma_P(\mu-1)-1$, then $\gamma_P(\mu+ \epsilon)= \gamma_P(\mu)+\epsilon$ for $\epsilon\geq 1$.  
    \end{lem}
    \begin{proof}
    Indeed, the number of columns is decreasing from left to right, so the number of columns will stabilize once if there is a column with only one block.     
    \end{proof}
    
Assume for simplicity that $i_0=n.$ 
            Suppose that there exists $\mu_0>1$ such that 
        \[\begin{aligned}
            \sum_{i=1}^n\gamma_{P^i}(\mu_0) &\geq (n-2)\mu_0+2 ,\\
            (\ast_\mu)\quad \sum_{i=1}^n\gamma_{P^i}(\mu) &\leq (n-2)\mu+1\quad  \text{ for } 1< \mu < \mu_0.
        \end{aligned}\]
        As each level function can only increase by $1$, so the total increment from $\mu_0-1$ to $\mu_0$ is bounded
        \[\sum_{i=1}^{n}\gamma_{P^i}(\mu_0) -  \sum_{i=1}^{n}\gamma_{P^i}(\mu_0-1) \leq n-1 \]
        and by combining with the inequality $(\ast_{\mu_0-1})$, we obtain
        \[ \sum_{i=1}^{n}\gamma_{P^i}(\mu_0) \leq (n-1) + \sum_{i=1}^{n}\gamma_{P^i}(\mu_0-1)\leq (n-2)\mu_0+2. \]
        Hence, $\sum_{i=1}^n \gamma_{P^i}(\mu_0) = (n-2)\mu_0+2$ in which case the total increment from $\mu_0-1$ to $\mu_0$ satisfies
        \[\sum_{i=1}^{n}\gamma_{P^i}(\mu_0) -  \sum_{i=1}^{n}\gamma_{P^i}(\mu_0-1)= n-1\]
        and $\sum _{i=1}^n \gamma_{P^i}(\mu_0-1) = (n-2)(\mu_0-1)+1 $.
        In other words, $\gamma_{P^i}(\mu_0-1) = \gamma_{P^i}(\mu_0)-1$ for $i=1,\dots, n-1$. 

        \noindent \textbf{Case 1}: Suppose $\mu_0-2>1.$ By the inequality $(\ast_{\mu_0-2})$, the total increment from $\mu_0-2$ to $\mu_0-1$ satisfies
        \[ \sum_{i=1}^{n}\gamma_{P^i}(\mu_0-1) -  \sum_{i=1}^{n}\gamma_{P^i}(\mu_0-2) \geq n-2.\]
        This implies that $\gamma_{P^i}(\mu_0-1) = \gamma_{P^i}(\mu_0-2)-1$ for at least $n-2$ indices $i\in \{1,\dots, n-1\}$. By Lemma \ref{lem:level function}, it follows that $\gamma_{P^i}(\mu_0+\epsilon)= \gamma_{P^i}(\mu_0)+\epsilon$ for at least $n-2$ indices $i$, so the total increment from $\mu_0$ to $\mu_0+\epsilon$ satisfies
        \[\sum_{i=1}^{n}\gamma_{P^i}(\mu_0+\epsilon) -  \sum_{i=1}^{n}\gamma_{P^i}(\mu_0) \geq (n-2) \epsilon \implies \sum_{i=1}^{n}\gamma_{P^i}(\mu_0+\epsilon) \geq (n-2)(\mu_0+\epsilon)+2. \]
        In particular, when $\mu_0+\epsilon =r $, the last inequality violates the hypothesis. 
        
        \noindent \textbf{Case 2}: Suppose $\mu_0-2=1.$ In this case, we have $\sum_{i=1}^n\gamma_{P^i}(\mu_0-2) =\sum_{i=1}^n\gamma_{P^i}(1) = n$ by the definition of a level function. So, the total increment from $\mu_0-2$ to $\mu_0-1$ is 
        \[ \sum_{i=1}^n\gamma_{P^i}(\mu_0-1) - \sum_{i=1}^n\gamma_{P^i}(\mu_0-2) =n-3. 
         \]
        By Lemma \ref{lem:level function}, it follows that $\gamma_{P^i}(\mu_0+\epsilon)= \gamma_{P^i}(\mu_0)+\epsilon$ for at least $n-3$ indices $i$, which we can assume to be $i=1,\dots, n-3$ and so  $\gamma_{P^i}(\mu) = \mu$ for $i=1,\dots, n-3$. The only freedom left are the choices of $\gamma_{P^i}(\mu)$ for $i=n-2,n-1$. We want to show that the sum of all $\gamma_{P^i}(\mu)$ violates $(\a)$. It suffices to consider the $\gamma_{P^i}(\mu)$ with the lowest growth for $i=n-1,n-2$. As $\gamma_{P^i}(1)=1$ and $\gamma_{P^i}(2)=1$ are fixed, the level function with the lowest growth is
        \[ \gamma_{P^i}(1) =1, \gamma_{P^i}(2)=1,\gamma_{P^i}(3)=2, \gamma_{P^i}(4)=2, \gamma_{P^i}(5)=3, \gamma_{P^i}(6)=3,\dots   \]
        i.e. the $\gamma_{P^i}(\mu) =\lceil \mu/2 \rceil.$ Then we see that 
        \[\sum_{i=1}^n\gamma_{P^i}(\mu) = (n-3)\mu+ 2\lceil \mu/2 \rceil + 1 \implies \sum_{i=1}^n\gamma_{P^i}(\mu) = \begin{cases}
            (n-2) \mu + 2&\text{if $\mu$ is odd} \\
            (n-2)\mu+ 1&\text{if $\mu$ is even}
        \end{cases} \]
        A direct computation shows that the sum of $\sum_{i=1}^n\gamma_{P^i}(\mu)$ violates $(\a)$.

\end{proof}
\section*{Acknowledgement}
We would like to thank Tony Pantev, Szil\'{a}rd Szab\'{o}, Bin Wang, Xueqing Wen, Qizheng Yin for many
useful discussions. S.L. was supported by the Institute for Basic Science (IBS-R003-D1).

\printbibliography

\end{document}